\documentclass{article}
\usepackage[utf8]{inputenc}
\usepackage{authblk}
\usepackage[linesnumbered,lined,ruled,commentsnumbered]{algorithm2e}

\title{From NeurODEs to AutoencODEs: \\a mean-field control framework for width-varying Neural Networks}
\author[1]{Cristina Cipriani \thanks{cristina.cipriani@ma.tum.de}}
\author[2]{Massimo Fornasier \thanks{massimo.fornasier@cit.tum.de}}
\author[3]{Alessandro Scagliotti \thanks{scag@ma.tum.de}}
\affil[1,2,3]{Technical University Munich, Department of Mathematics, Munich, Germany}
\affil[1,2]{Munich Data Science Institute (MDSI), Munich, Germany}
\affil[1,2,3]{Munich Center for Machine Learning (MCML), Munich, Germany}
\date{}

\usepackage[maxbibnames=99]{biblatex}
\usepackage{graphicx}
\usepackage{amsthm}

\newcommand{\mycomment}[1]{}
\newcommand{\MF}{{\mathcal{F}}}

\newcommand{\MG}{{\mathcal{G}}}
\newcommand{\R}{\mathbb{R}}
\newcommand{\RR}{\mathbb{R}}

\newcommand{\MC}{\mathcal{C}}

\newcommand{\BarR}{\Bar{R}}
\newcommand{\Bart}{\Bar{t}}

\newcommand{\A}{\mathcal{A}}
\newcommand{\I}{\mathcal{I}}
\newcommand{\mc}[1]{\mathcal{#1}}
\newcommand{\norm}[1]{\left\lVert#1 \, \right\rVert}
\newcommand{\BPhi}{\boldsymbol{\Phi}}
\newcommand{\supp}{\textnormal{supp}}

\newtheorem{thm}{Theorem}[section]
\newtheorem{proposition}[thm]{Proposition}
\newtheorem{cor}[thm]{Corollary}
\newtheorem{lem}[thm]{Lemma}
\newtheorem{rmk}{Remark}[section]
\newtheorem{definition}[thm]{Definition}

\renewcommand{\tilde}{\widetilde}
\renewcommand{\hat}{\widehat}
\renewcommand{\bar}{\overline}
\newtheorem{assum}{Assumption}

\RestyleAlgo{ruled}
\SetKwComment{Comment}{// }{}

\usepackage{amssymb, color}
\usepackage[margin=1in]{geometry}
\usepackage{amsfonts} 
\usepackage{empheq}
\usepackage{xcolor}
\usepackage{nicematrix}
\usepackage{multirow}

\numberwithin{equation}{section}
\providecommand{\keywords}[1]
{
  \small	
  \textbf{\textit{Keywords---}} #1
}

\begin{document}

\maketitle

\begin{abstract}
The connection between Residual Neural Networks (ResNets) and continuous-time control systems (known as NeurODEs) has led to a mathematical analysis of neural networks which has provided interesting results of both theoretical and practical significance. However, by construction, NeurODEs have been limited to describing constant-width layers, making them unsuitable for modeling deep learning architectures with layers of variable width.
In this paper, we propose a continuous-time Autoencoder, which we call AutoencODE, based on a modification of the controlled field that drives the dynamics. This adaptation enables the extension of the mean-field control framework originally devised for conventional NeurODEs.
In this setting, we tackle the case of low Tikhonov regularization, resulting in potentially non-convex cost landscapes. While the global results obtained for high Tikhonov regularization may not hold globally, we show that many of them can be recovered in regions where the loss function is locally convex.
Inspired by our theoretical findings, we develop a training method tailored to this specific type of Autoencoders with residual connections, and we validate our approach through numerical experiments conducted on various examples.
\end{abstract}


\section{Introduction}
In recent years, the field of artificial intelligence has witnessed remarkable progress across diverse domains, including computer vision and natural language processing. In particular, neural networks have emerged as a prominent tool, revolutionizing numerous machine learning tasks. Consequently, there is an urgent demand for a robust mathematical framework to analyze their intricate characteristics. 
A deep neural network can be seen as map $\phi:\R^{d_{\mathrm{in}}}\to\R^{d_{\mathrm{out}}}$, obtained as the composition of $L\gg 1$ applications $\phi=\phi_L\circ\ldots \circ \phi_1$, where, for every $n=1,\ldots,L$, the function $\phi_n:\R^{d_n}\to\R^{d_{n+1}}$ (also referred as \textit{the $n$-th layer} of the network) depends on a \textit{trainable} parameter $\theta_n\in \R^{m_n}$. The crucial process of choosing the values of the parameters $\theta_1,\ldots,\theta_L$ is known as the \textit{training of the network}. For a complete survey on the topic, we recommend the textbook \cite{Goodfellow-et-al-2016}. \\
Recent advancements have explored the link between dynamical systems, optimal control, and deep learning, proposing a compelling perspective. In the groundbreaking work \cite{he2016deep}, it was highlighted how the problem of training very deep networks can be alleviated by the introduction of a new type layer called “Residual Block”. This consists in using the identity map as skip connection, and after-addition activations. 
In other words, every layer has the following form:
\begin{equation}\label{eq:resnet}
    X_{n+1} = \phi_n(X_n) = X_n + \MF(X_n, \theta_n),
\end{equation}
where $X_{n+1}$ and $X_n$ are, respectively, the output and the input of the $n$-th layer.
This kind of architecture is called \textit{Residual Neural Network} (or ResNet). It is important to observe that, in order to give sense to the sum in \eqref{eq:resnet}, in each layer the dimension of the input should coincide with the dimension of the output. 
In the practice of Deep Learning, this novel kind of layer has turned out to be highly beneficial, since it is effective in avoiding the ``vanishing of the gradients" during the training \cite{bengio1994learning}, or the saturation of the network's accuracy \cite{he2015convolutional}. Indeed, before \cite{he2016deep}, these two phenomena had limited for long time the large-scale application of deep architectures.\\
Despite the original arguments in support of residual blocks being based on empirical considerations, their introduction revealed nevertheless a more mathematical and rigorous bridge between residual deep networks and controlled dynamical systems. Indeed, what makes Residual Neural Networks particularly intriguing is that they can be viewed as discretized versions of continuous-time dynamical systems. This dynamical approach was proposed independently in \cite{weinan2017proposal} and  \cite{haber2017stable}, and it was greatly popularized in the machine learning community under the name of NeurODEs by \cite{chen2018neural}. 
This connection with dynamical systems relies on reinterpreting the iteration \eqref{eq:resnet} as a step of the forward-Euler approximation of the following dynamical system:
\begin{equation}\label{eq:neurode}
    \dot{X}(t) = \MF(X(t), \theta(t)),
\end{equation}
where $t\mapsto \theta(t)$ is the map that, instant by instant, specifies the value of the parameter $\theta$. 
Moreover, the training of these neural networks, typically formulated as empirical risk minimization, can be reinterpreted as an optimal control problem. Given a labelled dataset $\{(X^i
_0, Y^i_0)\}^N_{i=1}$ of size $N \geq1$, the depth of the time-continuous neural network \eqref{eq:neurode} is denoted by $T > 0$. 
Then, training this network amounts to learning the control signals $\theta \in L^2([0,T],\R^m)$ in such a way that the terminal output $X^i_T$ of \eqref{eq:neurode} is close to it corresponding label $Y^i_0$ for all $i=1, \ldots,N$, with respect to some distortion measure $\ell(\cdot,\cdot) \in C^1$. A typical choice is $\ell(x, y) := |x - y|^2$, which is often referred as the \textit{squared loss function} in the machine learning literature. Therefore, it is possible to formulate the following optimal control problem
\begin{equation*}
\inf_{\theta \in L^2([0,T];\R^m)} J^N(\theta) := 
\left\{
\begin{aligned}
& \frac{1}{N} \sum_{i=1}^N \ell \big(X^i(T),Y^i(T) \big) + \lambda \int_0^T|\theta(t)|^2\, dt, \\
& \,\; \textnormal{s.t.} \,\,
\left\{
\begin{aligned}
& \dot X^i(t) = \MF(t,X^i(t),\theta(t) ), \hspace{1.45cm} \dot Y^i(t) =0, \\ 
& (X^i(t),Y^i(t))\big |_{t=0} = (X_0^i,Y_0^i), ~~ i \in \{1,\dots,N\},
\end{aligned}
\right.
\end{aligned}
\right.
\end{equation*}
where, differently from \eqref{eq:neurode}, we admit here the explicit dependence of the dynamics on the time variable.
Notice that the objective function also comprises of Tikhonov regularization, tuned by the parameter $\lambda$, which plays a crucial role in the analysis of this control problem. 
The benefit of interpreting the training process in this manner results in the possibility of exploiting established results from the branch of mathematical control theory, to better understand this process. 
A key component of optimal control theory is a set of necessary conditions, known as Pontryagin Maximum Principle (PMP), that must be satisfied by any (local) minimizer $\theta$. These conditions were introduced in \cite{pontryagin1987} and have served as inspiration for the development of innovative algorithms \cite{li2017maximum} and network structures \cite{chang2018reversible} within the machine learning community.\\
This work specifically addresses a variant of the optimal control problem presented above, in which the focus is on the case of an infinitely large dataset. 
This formulation gives rise to what is commonly known as a \textit{mean-field optimal control problem}, where the term ``mean-field'' emphasizes the description of a multiparticle system through its averaged effect. In this context, the focus is on capturing the collective behavior of the system rather than individual particle-level dynamics, by considering the population as a whole. As a consequence, the parameter $\theta$ is shared by the entire population of input-target pairs, and the optimal control is required to depend on the initial distribution $\mu_0(x,y)\in \mathcal{P}(\R^d\times \R^d)$ of the input-target pairs. Therefore, the optimal control problem needs to be defined over spaces of probability measures, and it is formulated as follows:
\begin{equation*}
 \inf_{\theta \in L^2([0,T];\R^m)} J(\theta) := 
\left\{
\begin{aligned}
& \int_{\RR^{2d}}\ell(x,y) \, d\mu_T(x,y)+\lambda\int_0^T|\theta(t)|^2 \, dt ,  \\
& \hspace{0.2cm} \text{s.t.} ~ \left\{
\begin{aligned}
& \partial_t\mu_t(x,y)+\nabla_x\cdot (\mathcal F(t,x,\theta_t)\mu_t(x,y))=0 & t\in[0,T],\\
& \mu_t|_{t=0}(x,y)=\mu_0(x,y),
\end{aligned}
\right.
\end{aligned}
\right.
\end{equation*}
This area of study has gained attention in recent years, and researchers have derived the corresponding Pontryagin Maximum Principle in various works, such as \cite{weinan2018mean} and \cite{bonnet2023measure}. It is worth mentioning that there are other types of mean-field analyses of neural networks, such as the well-known work \cite{mei2019mean}, which focus on mean-field at the parameter level, where the number of parameters is assumed to be infinitely large. However, our approach in this work takes a different viewpoint, specifically focusing on the control perspective in the case of an infinitely large dataset. \\
One of the contributions of this paper is providing a more accessible derivation of the necessary conditions for optimality, such as the well-known Pontryagin Maximum Principle. Namely, we characterize the stationary points of the cost functional, and we are able to recover the PMP that was deduced in \cite{bonnet2023measure} under the assumption of large values of the regularization parameter $\lambda$, and whose proof relied on an infinite-dimensional version of the Lagrange multiplier rule.
This alternative perspective offers a clearer and more intuitive understanding of the PMP, making it easier to grasp and apply it in practical scenarios.\\
In addition, we aim at generalizing the applicability of the results presented in \cite{bonnet2023measure} by considering a possibly non-convex regime, corresponding to small values of the parameter $\lambda>0$. As mentioned earlier, the regularization coefficient $\lambda$ plays a crucial role in determining the nature of the cost function. Indeed, when $\lambda$ is sufficiently large, the cost function is convex on the sub-level sets, and it is possible to prove the existence and uniqueness of the solution of the optimal control problem that arises from training NeurODEs. 
Additionally, in this highly-regularized scenario, desirable properties of the solution, such as its continuous dependence on the initial data and a bound on the generalization capabilities of the networks, have been derived in \cite{bonnet2023measure}.\\
However, in practical applications, a large regularization parameter may cause a poor performance of the trained NeurODE on the task. In other words, in the highly-regularized case, the cost functional is unbalanced towards the $L^2$-penalization, at the expenses of the term that promotes that each datum $X^i_0$ is driven as close as possible to the corresponding target $Y^i_0$. 
This motivated us to investigate the case of low Tikhonov regularization. While we cannot globally recover the same results as in the highly-regularized regime, we find interesting results concerning local minimizers. 
Moreover, we also show that the (mean field) optimal control problem related to the training of the NeurODE induces a gradient flow in the space of admissible controls. 
The perspective of the gradient flow leads us to consider the well-known minimizing movement scheme, and to introduce a proximal stabilization term to the cost function in numerical experiments. 
This approach effectively addresses the well-known instability issues (see \cite{chernousko1982method}) that arise when solving numerically optimal control problems (or when training NeurODEs) with itearive methods based on the PMP.
It is important to note that our stabilization technique differs from previous methods, such as the one introduced in \cite{li2017maximum}.\\

\noindent
\paragraph{From NeurODEs to AutoencODEs.}
Despite their huge success, it should be noted that NeurODEs (as well as ResNets, their discrete-time counterparts) in their original form face a limitation in capturing one of the key aspects of modern machine learning architectures, namely the discrepancy in dimensionality between consecutive layers. 
As observed above, the use of skip connections with identity mappings requires a ``rectangular'' shape of the network, where the width of the layers are all identical and constant with respect to the input's dimension. This restriction poses a challenge when dealing with architectures that involve layers with varying dimensions, which are common in many state-of-the-art models. 
Indeed, the inclusion of layers with different widths can enhance the network's capacity to represent complex functions and to capture intricate patterns within the data.
In this framework, Autoencoders have emerged as a fundamental class of models specifically designed to learn efficient representations of input data by capturing meaningful features through an encoder-decoder framework. More precisely, the encoder compresses the input data into a lower-dimensional latent space, while the decoder reconstructs the original input from the compressed representation. 
The concept of Autoencoders was first introduced in the 1980s in \cite{rumelhart1985learning}, and since then, it has been studied extensively in various works, such as \cite{hinton2006fast}, among many others. 
Nowadays, Autoencoders have found numerous applications, including data compression, dimensionality reduction, anomaly detection, and generative modeling. Their ability to extract salient features and capture underlying patterns in an unsupervised manner makes them valuable tools in scenarios where labeled training data is limited or unavailable. Despite their huge success in practice, there is currently a lack of established theory regarding the performance guarantees of these models.\\
{Prior works, such as \cite{esteve2020large}, have extended the control-theoretic analysis of NeurODEs to more general width-varying neural networks. Their model is based on an integro-differential equation that was first suggested in \cite{liu2020selection} in order to study the continuum limit of neural networks with respect to width and depth. In such an equation the state variable has a dependency on both time and space since the changing dimension over time is viewed as an additional spatial variable. 
In \cite[Section~6]{esteve2020large} the continuous space-time analog of residual neural networks proposed in \cite{liu2020selection} has been considered and discretized in order to model variable width ResNets of various types, including convolutional neural networks. 
The authors assume a simple time-dependent grid, and use forward difference discretization for the time derivative and Newton-Cotes for discretizing the integral term, but refer to more sophisticated moving grids in order to possibly propose new types of architectures. In this setting, they are also able to derive some stability estimates and generalization properties in the overparametrized regime, making use of turnpike theory in optimal control \cite{geshkovski2022turnpike}. 
In principle, there could be several different ways to model width-varying neural networks with dynamical systems, e.g., forcing some structure on the control variables, or formulating a viability problem. In this last case, a possibility could be to require admissible trajectories to visit some lower-dimensional subsets during the evolution. For an introduction to viability theory, we recommend the monograph \cite{aubin2010}, while we refer to \cite{bonnet2022viability, bonnet2023viability} for recent results on viability theory for differential inclusions in Wasserstein spaces.\\
In contrast, our work proposes a simpler extension of the control-theoretical analysis. It is based on a novel design of the vector field that drives the dynamics, allowing us to develop a continuous-time model capable of accommodating various types of width-varying neural networks. This approach has the advantage of leveraging insights and results obtained from our previous work \cite{bonnet2023measure}. Moreover, the simplicity of our model facilitates the implementation of residual networks with variable width and allows us to test their performance in machine learning tasks.}
In order to capture width-varying neural networks, we need to extend the previous control-theoretical framework to a more general scenario, in particular we need to relax some of the assumptions of \cite{bonnet2023measure}. This is done in Subsection~\ref{subsec:Neur_to_Auto}, where we introduce a discontinuous-in-time dynamics that can describe a wider range of neural network architectures. 
By doing so, we enable the study of Autoencoders (and, potentially, of other width-varying architectures) from a control-theoretic point of view, with the perspective of getting valuable insights into their behavior.\\
Furthermore, we also generalize the types of activation functions that can be employed in the network. The previous work \cite{bonnet2023measure} primarily focused on sigmoid functions, which do not cover the full range of activations commonly employed in practice. 
Our objective is to allow for unbounded activation functions, which are often necessary for effectively solving certain tasks. By considering a broader set of activation functions, we aim at enhancing the versatility and applicability of our model.\\
Furthermore, in contrast to \cite{bonnet2023measure}, we introduce a stabilization method to allow the numerical resolution of the optimal control problem in the low-regularized regime, as previously discussed. This stabilization technique provides the means to test the architecture with our training approach on various tasks: from low-dimensional experiments, which serve to demonstrate the effectiveness of our method, to more sophisticated and high-dimensional tasks such as image reconstruction. In Section \ref{sec:num_exp}, we present all the experiments and highlight noteworthy behaviors that we observe. An in-depth exploration of the underlying reasons for these behaviors is postponed to future works.\\

\noindent
The structure of the paper is the following: Section~\ref{sec:neurODEs} discusses the dynamical model of NeurODEs and extends it to the case of width-varying neural networks, including Autoencoders, which we refer to as AutoencODEs. 
In Section~\ref{sec:mean-field}, we present our mean-field analysis, focusing on the scenario of an infinitely large dataset. We formulate the mean-field optimal control problem, we derive a set of necessary optimality conditions, and we provide a convergence result for the finite-particles approximation. At the end of this section, we compare our findings with the ones previously obtained in \cite{bonnet2023measure}. 
Section~\ref{sec:Alg} covers the implementation and the description of the training procedure, and we compare it with other  methods for NeurODEs existing in the literature.
Finally, in Section~\ref{sec:num_exp}, we present the results of our numerical experiments, highlighting interesting properties of the AutoencODEs that we observe.


\subsection*{Measure-theoretic preliminaries}
Given a metric space $(X,d_X)$, we denote by $\mathcal{M}(X)$ the space of signed Borel measures in $X$ with finite total variation, and by $\mathcal{P}(X)$ the space of probability measures, while $\mathcal{P}_c(X) \subset \mathcal{P}(X)$ represents the set of probability measures with compact support. Furthermore, $\mathcal{P}_c^N(X) \subset \mathcal{P}_c(X)$ denotes the subset of empirical or atomic probability measures.
Given $\mu \in \mathcal{P}(X)$ and $f: X \to Y$ , with $f$ $\mu-$measurable, we denote with $f_{\#}\mu \in \mathcal{P}(Y)$ the push-forward measure defined by $f_{\#}\mu(B) = \mu(f^{-1}(B))$ for any Borel set $B\subset Y$. Moreover, we recall the change-of-variables formula
\begin{equation}\label{eq:push_forw}
    \int_Y g(y)\, d\big(f_{\#}\mu\big )(y) = \int_X g \circ f(x) \,d\mu(x)
\end{equation}
whenever either one of the integrals makes sense.\\
We now focus on the case $X=\R^d$ and  briefly recall the definition of the Wasserstein metrics of optimal transport in the following definition, and refer to \cite[Chapter 7]{ambrosio2005gradient} for more details.
\begin{definition}
Let $1\leq p < \infty$ and $\mathcal{P}_p(\RR^{d})$ be the space of Borel probability measures on $\RR^{d}$ with finite $p$-moment. In the sequel, we endow the latter with the $p$-\emph{Wasserstein metric}
\begin{equation*}\label{eq:wassdis}
W_p^{p}(\mu, \nu):=\inf\left\{\int_{\RR^{2d}} |z-\hat{z}|^{p}\ d\pi(z,\hat{z})\ \big| \ \pi \in \Pi(\mu, \nu)\right\} ,
\end{equation*}
where $\Pi(\mu, \nu)$ denotes the set of \emph{transport plan} between $\mu$ and $\nu$, that is the collection of all Borel probability measures on $\R^d\times \R^d$ with marginals $\mu$ and $\nu$ in the first and second component respectively.
\end{definition} 

It is a well-known result in optimal transport theory that when $p =1$ and $\mu,\nu \in \mathcal{P}_c(\R^d)$, then the following alternative representation holds for the Wasserstein distance
\begin{equation}\label{eq:Kanto}
W_1(\mu,\nu)=\sup \left\{\int_{\RR^d}\varphi(x) \,d \big(\mu-\nu\big)(x) \, \big| \, \varphi\in \mbox{Lip}(\R^d),~ \mbox{Lip}(\varphi)\leq 1\right\}\,,
\end{equation}
by Kantorovich's duality \cite[Chapter 6]{ambrosio2005gradient}. Here, $\mbox{Lip}(\R^d)$ stands for the space of real-valued Lipschitz continuous functions on $\R^d$, and $\mbox{Lip}(\varphi)$ is the Lipschitz constant of a mapping $\varphi$ defined ad 
\begin{equation*}
    Lip(\varphi) := \sup_{x,y \in \R^d , x \neq y} \frac{\|\varphi(x)-\varphi(y)\|}{\|x-y\|}
\end{equation*}


\section{Dynamical Model of NeurODEs}\label{sec:neurODEs}
\subsection{Notation and basic facts}
In this paper, we consider controlled dynamical systems in $\R^d$, where the velocity field is prescribed by a function $\MF: [0,T] \times \R^d \times \R^m \to \R^d$ that satisfies these basic assumptions.
\begin{assum}\label{ass:block1}
The vector field $\MF: [0,T] \times \R^d \times \R^m \to \R^d$ satisfies the following:
\begin{enumerate}
\item[$(i)$] For every $x \in \R^d$ and every $\theta \in \R^m$, the map $t \mapsto \MF(t,x,\theta)$ is measurable in $t$.
\item[$(ii)$] For every $R>0$ there exists a constant  $L_{R} >0$ such that, for every $\theta \in \R^m$, it holds
\begin{equation*}
|\mathcal F(t,x_1,\theta)-\mathcal F(t,x_2,\theta)|\leq L_{R}(1+|\theta|) |x_1-x_2| ,\quad \mbox{ for a.e. } t\in [0,T] \mbox{ and every }x_1, x_2 \in B_R(0),  
\end{equation*}
from which it follows that $|\MF(t,x,\theta)| \leq L_R(1 + |x|)(1+ |\theta|)$ for a.e. $t\in [0,T]$.
\item[$(iii)$] For every $R>0$ there exists a constant  $L_{R} >0$ such that, for every $\theta_1, \theta_2 \in \R^m$, it holds
\begin{equation*}
|\MF(t,x,\theta_1)-\MF(t,x,\theta_2)| \leq L_R(1 + |\theta_1| + |\theta_2|)|\theta_1-\theta_2| ,\quad \mbox{ for a.e. } t\in [0,T] \mbox{ and every }x \in B_R(0). 
\end{equation*}
\end{enumerate}
\end{assum}
\noindent
The control system that we are going to study is
\begin{equation}\label{eq:ode}
    \begin{cases}
    \dot{x}(t) = \MF(t,x(t), \theta(t)),  & \mbox{a.e. in } [0,T],\\
    x(0) = x_0,
    \end{cases}
\end{equation}
where $\theta \in L^2([0,T], \R^m)$ is the control that drives the dynamics.
Owing to Assumption \ref{ass:block1}, the classical Carathéodory Theorem (see \cite[Theorem 5.3]{hale1969ordinary})
 guarantees that, for every $\theta \in L^2([0,T], \R^m)$ and for every $x_0 \in \R^d$, the Cauchy problem \eqref{eq:ode} has a unique solution $x : [0,T] \to \R^d$. Hence, for every $(t, \theta) \in [0,T] \times L^2([0,T], \R^m)$, we introduce the flow map $\Phi^\theta_{(0,t)}: \R^d \to \R^d$ defined as 
\begin{equation}\label{eq:flow}
    \Phi^\theta_{(0,t)}(x_0) := x(t), 
\end{equation}
where $t \mapsto x(t)$ is the absolutely continuous curve that solves \eqref{eq:ode}, with Cauchy datum $x(0)=x_0$ and corresponding to the admissible control $t \mapsto \theta(t)$.
Similarly, given $0\leq s< t\leq T$, we write $\Phi^\theta_{(s,t)}: \R^d \to \R^d$ to denote the flow map obtained by prescribing the Cauchy datum at the more general instant $s\geq 0$.
We now present the properties of the flow map defined in \eqref{eq:flow} that describes the evolution of the system: we show that is well-posed, and we report  some classical properties.
\begin{proposition}\label{prop:flow_basics}
For every $t \in [0,T]$ and for every $\theta \in L^2([0,T], \R^m)$, let $\MF$ satisfy Assumption \ref{ass:block1}. Then, the flow $\Phi^\theta_{(0,t)}: \R^d \to \R^d$ is well-defined for any $x_0 \in \R^d$ and it satisfies the following properties.
\begin{itemize}
    \item For every $R>0$ and $\rho>0$, there exists a constant $\bar R>0$ such that
    \begin{equation*}
        |\Phi^\theta_{(0,t)}(x)| \leq \bar{R}
    \end{equation*}
    for every $x \in B_R(0)$ and every $\theta \in L^2([0,T],\R^m)$ such that $||\theta ||_{L^2} \leq \rho$.
    \item For every $R>0$ and $\rho>0$, there exists a constant $\bar{L}>0$ such that, for every $t \in [0,T]$, it holds
    \begin{equation*}
        |\Phi^\theta_{(0,t)}(x_1)- \Phi^\theta_{(0,t)}(x_2)| \leq \bar{L}|x_1-x_2|
    \end{equation*}
    for every $x_1, x_2 \in B_R(0)$ and every $\theta \in L^2([0,T],\R^m)$ such that $||\theta ||_{L^2} \leq \rho$. 
    \item For every $R>0$ and $\rho>0$, there exists a constant $\bar{L}>0$ such that, for every $t_1, t_2\in [0,T]$, it holds 
    \begin{equation*}
        |\Phi^\theta_{(0,t_2)}(x)- \Phi^\theta_{(0,t_1)}(x)| \leq \bar{L}|t_2-t_1|^\frac{1}{2}
    \end{equation*}
    for every $x \in B_R(0)$ and every $\theta \in L^2([0,T],\R^m)$ such that $||\theta ||_{L^2} \leq \rho$.
    \item For every $R>0$ and $\rho>0$, there exists a constant $\bar{L}>0$ such that, for every $t \in [0,T]$, it holds
    \begin{equation*}
        |\Phi^{\theta_1}_{(0,t)}(x)- \Phi^{\theta_2}_{(0,t)}(x)|_2 \leq \bar{L}\|\theta_1-\theta_2\|_{L^2}
    \end{equation*}
    for every $x \in B_R(0)$ and every $\theta_1, \theta_2 \in L^2([0,T],\R^m)$ such that $\|\theta_1 \|_{L^2}, \|\theta_2 \|_{L^2} \leq \rho$.
\end{itemize}
\end{proposition}
\begin{proof}
The proof is postponed to the Appendix (see Lemmata \ref{lem:bound_traj}, \ref{lem:lip_flow_init}, \ref{lem:lip_flow_time}, \ref{lem:lip_flow_ctrl}).
\end{proof}

Even though the framework introduced in Assumption \ref{ass:block1} is rather general, in this paper we specifically have in mind the case where the mapping $\MF: [0, T] \times \mathbb{R}^d \times \mathbb{R}^m \rightarrow \mathbb{R}^d$ represents the feed-forward dynamics associated to residual neural networks. 
In this scenario, the parameter $\theta \in \mathbb{R}^m$ encodes the \textit{weights} and \textit{shifts} of the network, i.e., $\theta = (W, b)$, where $W \in \mathbb{R}^{d \times d}$ and $b \in \mathbb{R}^d$. Moreover, the mapping $\MF$ has the form:
\begin{equation*}
    \MF(t, x, \theta) = \sigma(W x + b),
\end{equation*}
where $\sigma: \mathbb{R}^d \rightarrow \mathbb{R}^d$ is a nonlinear function acting component-wise, often called in literature  \textit{activation function}. In this work, we consider sigmoidal-type activation functions, such as the hyperbolic tangent function:
\begin{equation*}
    \sigma(x) = \tanh(x),
\end{equation*}
as well as smooth approximations of the Rectified Linear Unit (ReLU) function,  which is defined as:
\begin{equation}\label{eq:relu_def}
    \sigma(x) = \max\{0, x\}.
\end{equation}
We emphasize the need to consider smoothed versions of the ReLU function due to additional differentiability requirements on $\MF$, which will be further clarified in Assumption \ref{ass:block2}. Another useful activation function covered by Assumption \ref{ass:block2} is the Leaky Rectified Linear Unit (Leaky ReLU) function:
\begin{equation}\label{eq:leaky_relu_def}
    \sigma(x) = \max\{0,x\} - \max\{ -\alpha x,0 \} 
\end{equation}
where $\alpha \in [0, 1]$ is a predetermined parameter that allows the output of the function to have negative values. The smooth approximations of \eqref{eq:relu_def} and \eqref{eq:leaky_relu_def} that we consider will be presented in Section \ref{sec:Alg}.

\subsection{From NeurODEs to AutoencODEs} \label{subsec:Neur_to_Auto}
As explained in the Introduction, NeurODEs and ResNets --their discrete-time counterparts-- face the limitation of a ``rectangular'' shape of the network because of formulas \eqref{eq:neurode} and \eqref{eq:resnet}, respectively. 
To overcome this fact, we aim at designing a continuous-time model capable of describing width-varying neural networks, with a particular focus on Autoencoders, as they represent the prototype of neural networks whose layers operate between spaces of different dimensions. 
Indeed, Autoencoders consist of an \textit{encoding phase}, where the layers' dimensions progressively decrease until reaching the ``latent dimension'' of the network. Subsequently, in the \textit{decoding phase}, the layers' widths are increased until the same dimensionality as the input data is restored. For this reason, Autoencoders are prominent examples of width-varying neural networks, since the changes in layers' dimensions lie at the core of their functioning.
Sketches of encoders and Autoencoders are presented in Figure \ref{fig:encoder_autoencoder}. 
Finally, we insist on the fact that our model can encompass as well other types of architectures. In this regard, in Remark \ref{rmk:u-net} we discuss how our approach can be extended to U-nets.
\begin{figure}[ht!]
\begin{center}
    \includegraphics[scale=0.3]{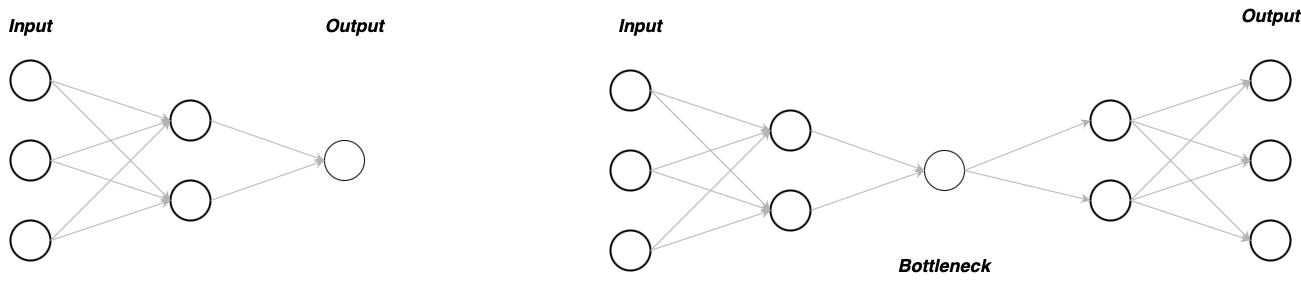}
\end{center}
\caption{Left: network with an encoder structure. Right: Autoencoder.}
\label{fig:encoder_autoencoder}
\end{figure}

\paragraph{Encoder} Our goal is to first model the case of a network which sequentially reduces the dimensionality of the layers' outputs. 
For this purpose, we artificially force some of the components not to evolve anymore, while we let the others be active part of the dynamics. More precisely, given an input variable $x_0 \in \RR^{d}$, we denote with $(\I_j)_{j=0,\ldots,r}$ an increasing filtration, where each element $\I_j$ contains the sets of indices whose corresponding components are \textit{inactive}, i.e., they are constant and do not contribute to the dynamics. Clearly, since the layers' width will decrease sequentially, the filtration of inactive components $\I_j$ will increase, i.e.
\begin{equation*}
    \emptyset =: \I_0 \subsetneq \I_1 \subsetneq ... \subsetneq \I_r \subsetneq \{1,\ldots,d\}, \quad r < d, \quad j=0,\ldots, r.
\end{equation*}
On the other hand, the sets of indices of \textit{active} components define an decreasing filtration $\A_j := \{1, \ldots,d\}\setminus \I_j$ for $j=0,\ldots,r$. As opposed to before, the sets of active components $(\A_j)_{j=0,\ldots,r}$ satisfy
\begin{equation*}
    \{1,\ldots, d\} =: \A_0 \supsetneq \A_1 \supsetneq ... \supsetneq \A_r \supsetneq \emptyset, \quad r < d, \quad j=0,\ldots, r.
\end{equation*}
We observe that, for every $j=0,\ldots,r$,  the sets $\A_j$ and $\I_j$ provide a partition of $\{1,\ldots,d\}$. A visual representation of this model for encoders is presented on the left side of Figure \ref{fig:our_encoder_autoencoder}.\\
Now, in the time interval $[0,T]$, let us consider $r+1$ nodes $0 = t_0 < t_1 <... < t_r<t_{r+1}=T$. For $j=0,\ldots,r$, we denote with $[t_j,t_{j+1}]$ the sub-interval and, for every $x \in \R^d$, we use the notation $x_{\I_j}:=(x_i)_{i\in \I_j}$ and $x_{\A_j}:=(x_i)_{i\in \A_j}$ to access the components of $x$ belonging to $\I_j$ and $\A_j$, respectively. Hence, the controlled dynamics for any $t \in [t_j, t_{j+1}]$ can be described by 
\begin{equation}\label{eq:encoder_ode}
    \begin{cases}
    \dot{x}_{\I_j}(t) = 0,\\
    \dot{x}_{\A_j}(t) = \MG_j(t, x_{\A_j}(t), \theta(t)), 
    \end{cases}
\end{equation}
where $\MG_j: \left [ t_j, t_{j+1}\right ] \times \RR^{| \A_j|} \times \RR^m \to \RR^{|\A_j|}  $, for $j= 0,\ldots,r$, and $x(0) = x_{\A_0}(0) = x_0$.
Furthermore, the dynamical system describing the encoding part is 
\begin{equation*}
    \begin{cases}
      \dot{x}(t) = \MF(t,x(t), \theta(t)), &\mbox{a.e } t \in [0,T],\\
      x(0) = x_0
    \end{cases}
\end{equation*}
where, for $t \in [t_j, t_{j+1}]$, we define the discontinuous vector field as follows
\begin{equation*}
    \big(\MF(t,x,\theta)\big)_k = \begin{cases}
      \left (\MG(t, x_{\A_j},\theta) \right)_k, & \mbox{if }k \in \A_j,\\
      0, & \mbox{if }k \in \I_j.
    \end{cases}
\end{equation*}
\begin{rmk}\label{rmk:dim_theta}
Notice that $\theta(t) \in \R^m$ for every $t \in [0,T]$, 
according to the model that we have just described. However, it is natural to expect that, since $x$ has varying active components, in a similar way the controlled dynamics $\MF(t,x,\theta)$ shall not explicitly depend at every $t \in [0,T]$ on every component of $\theta$. 
\end{rmk}

\begin{figure}[ht]
\begin{center}
    \includegraphics[scale = 0.3]{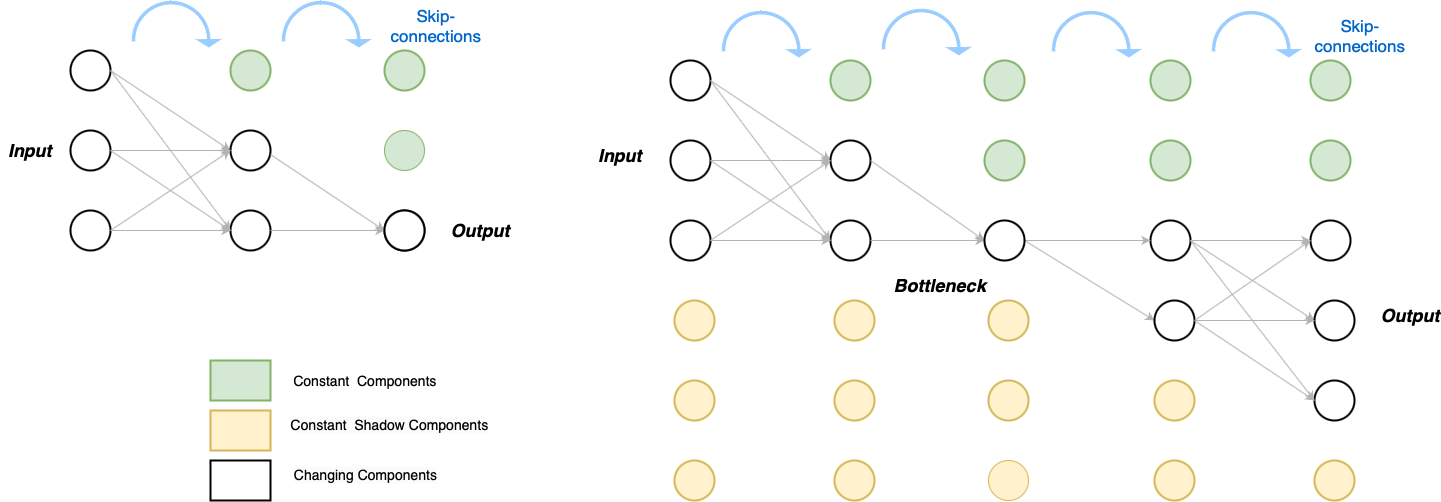}
\end{center}
\caption{Left: Embedding of an encoder into a dynamical system. Right: model for an Autoencoder.}
\label{fig:our_encoder_autoencoder}
\end{figure}

\paragraph{Autoencoder} We now extend the previous model to the case of networks which not only decrease the dimensionality of the layers, but they are also able to increase the layers' width in order to restore the original dimension of the input data. 
Here we denote by $z_0 \in \RR^{\tilde d}$ the input variable, and we fictitiously augment the input's dimension, so that we consider the initial datum $x_0 = (z_0, \underbar{0})\in \RR^{d} =~ \RR^{\tilde d} \times \RR^{\tilde d}$, where $\underbar{0} \in \RR^{\tilde d}$. We make use of the following notation for every $x \in \R^d$:
\begin{equation*}
    x = \big ( (z_i)_{i=1,\ldots,\tilde d}, (z^H_i)_{i= 1,\ldots,\tilde d}) \big )
\end{equation*}
where $z^H$ is the augmented (or \textit{shadow}) part of the vector $x$. 
In this model, the time horizon $[0,T]$ is splitted using the following time-nodes:
\begin{equation*}
    0=t_0 \leq t_1 \leq ... \leq t_r \leq ... \leq t_{2r} \leq t_{2r+1}:=T
\end{equation*}
where $t_r$, which was the end of the encoder in the previous model, is now the instant corresponding to the bottleneck of the autoencoder. 
Similarly as before, we introduce two families of partitions of $\{1,\ldots, \tilde d \}$ modeling the active and non-active components of, respectively, $z$ and $z^H$. The first filtrations are relative to the encoding phase and they involve the component of $z$:
\begin{alignat*}{2}
    & \begin{aligned} & \begin{cases}
    \I_{j-1} \subsetneq I_j & \text{if }  1 \leq j \leq r, \\
    \I_j = \I_{j-1} & \text{if }  j>r,
  \end{cases}
  \end{aligned}
    & \hskip 6em &
  \begin{aligned}
  & \begin{cases}
  \A_{j-1} \supsetneq \A_j & \text{if } 1 \leq j \leq r, \\
  \A_j = \A_{j-1} & \text{if } j>r.
  \end{cases} 
  \end{aligned}
\end{alignat*}
where $\I_0 := \emptyset$, $\I_r \subsetneq \{1,\ldots, \tilde d\}$ and $\A_0 =\{ 1, \ldots, \tilde d \}$, $\A_r \supsetneq \emptyset$. The second filtrations, that aim at modeling the decoder, act on the shadow part of $x$, i.e., they involve the components of $z^H$:
\begin{alignat*}{2}
    & \begin{aligned} & \begin{cases}
    \I^H_{j-1} = \{1,\ldots, d\} & \text{if }  1 \leq j \leq r, \\
    \I^H_j \subsetneq \I^H_{j-1} & \text{if }  r < j \leq 2r,
  \end{cases}
  \end{aligned}
    & \hskip 6em &
  \begin{aligned}
  & \begin{cases}
  \A^H_{j-1} = \emptyset & \text{if } 1 \leq j \leq r, \\
  \A^H_j \supsetneq \A^H_{j-1} & \text{if } r < j \leq 2r.
  \end{cases}
  \end{aligned}
\end{alignat*}
While the encoder structure acting on the input data $z_0$ is the same as before, in the decoding phase we aim at activating the components that have been previously turned off during the encoding. 
However, since the information contained in the original input $z_0$ should be first compressed and then decompressed, we should not make use of the values of the components that we have turned off in the encoding and hence, we cannot re-activate them. 
Therefore, in our model the dimension is restored by activating components of $z^H$, the shadow part of $x$, which we recall was initialized equal to $\underbar{0} \in \R^{\tilde d}$. 
This is the reason why we introduce sets of active and inactive components also for the shadow part of the state variable. A sketch of this type of model is presented on the right of Figure \ref{fig:our_encoder_autoencoder}.
Moreover, in order to be consistent with the classical structure of an autoencoder, the following identities must be satisfied:
\begin{enumerate}
\item $\A_j \cap \A_j^H = \emptyset$  for every $j=1,\ldots,2r$,
\item $\A_{2r} \cup \A^H_{2r} = \{1,\ldots, \tilde d\}$.
\end{enumerate} 
The first identity formalizes the constraint that the active component of $z$ and those of $z^H$ cannot overlap and must be distinct, while the second identity imposes that, at the end of the evolution, the active components in $z$ and $z^H$ should sum up exactly to ${1,\ldots,\tilde d}$.
Furthermore, from the first identity we derive that $\A_j \subseteq (\A_j^H)^C = \I_j^H$ and, similarly, $\A_j^H \subseteq \I_j$ for every $j=1,\ldots,2r$. Moreover, $\A_r$ satisfies the inclusion $\A_r \subseteq \A_j$ for every $j=1,\ldots,2r$, which is consistent with the fact that layer with the smallest width is located in the bottleneck, i.e., in the interval $[t_r,t_{r+1}]$. Finally, from the first and the second assumption, we obtain that $\A_{2r}^H = \I_{2r}$, i.e., the final active components of $z^H$ coincide with the inactive components of $z$, and, similarly, $\I_{2r}^H = \A_{2r}$.
Finally, to access the active components of $x = (z, z^H)$, we make use of the following notation: 
\[
x_{\A_j}= (z_k)_{k \in \A_j},  \quad x_{\A^H_j} = (z^H_k)_{k \in \A^H_j}\quad \text{and} \quad x_{\A_j, \A^H_j} = (z_{\A_j}, z^H_{\A^H_j}),
\]
and we do the same for the inactive components:
$$x_{\I_j}= (z_k)_{k \in \I_j},  \quad x_{\I^H_j} = (z^H_k)_{k \in \I^H_j}\quad \text{and} \quad x_{\I_j, \I^H_j} = (z_{\I_j}, z^H_{\I^H_j}).$$\\
\noindent
We are now in position to write the controlled dynamics in the interval $t_j \leq t \leq t_{j+1}$:
\begin{equation}
    \begin{cases}
    \dot{x}_{\I_j, \I^H_j}(t) = 0, \\
  \dot{x}_{\A_j, \A^H_j}(t) = \MG_j(t, x_{\A_j, \A^H_j}(t), \theta(t)),
    \end{cases}
\end{equation}
where $\MG_j: \left [ t_j, t_{j+1} \right]\times \RR^{| \A_j | + | \A_j^H |} \times \RR^{m} \to \RR^{| \A_j | + | \A_j^H |}$ , for $j=0,\ldots, 2r$, and $x^H_{\I_0}(0) = \underbar{0}$, $x_{\A_0}(0) =x_0$. As before, we define the discontinuous vector field $\MF$ for $t \in [t_j, t_{j+1}]$, as follows
\begin{equation*}
    \big(\MF(t,x,\theta)\big)_k = \begin{cases}
      \left (\MG(t, x_{\A_j},\theta) \right)_k, & \mbox{if }k \in \A_j \cup \A_j^H\\
      0, & \mbox{if }k \in \I_j \cup \I^N_j.
    \end{cases}
\end{equation*}
Hence, we are now able to describe any type of width-varying neural network through a continuous-time model depicted by the following dynamical system
\begin{equation*}
    \begin{cases}
    \dot{x}(t) = \MF(t,x(t), \theta(t)) &\mbox{a.e. in } [0,T],\\
    x(0) = x_0.
    \end{cases}
\end{equation*}
It is essential to highlight the key difference between the previous NeurODE model in (2.6) and the current model: the vector field $\MF$ now explicitly depends on the time variable $t$ to account for sudden dimensionality drops, where certain components are forced to remain constant. As a matter of fact, the resulting dynamics exhibit high discontinuity in the variable $t$. To the best of our knowledge, this is the first attempt to consider such discontinuous dynamics in NeurODEs. Previous works, such as \cite{haber2017stable, weinan2017proposal}, typically do not include an explicit dependence on the time variable in the right-hand side of NeurODEs, or they assume a continuous dependency on time, as in \cite{bonnet2023measure}. Furthermore, it is worth noting that the vector field $\MF$ introduced to model autoencoders satisfies the general assumptions outlined in Assumption 1 at the beginning of this section.

\begin{rmk}\label{rmk:u-net}
The presented model, initially designed for Autoencoders, can be easily extended to accommodate various types of width-varying neural networks, including architectures with long skip-connections such as U-nets \cite{ronneberger2015u}. While the specific details of U-nets are not discussed in detail, their general structure is outlined in Figure \ref{fig:sketch_unet}. U-nets consist of two main components: the contracting path (encoder) and the expansive path (decoder). These paths are symmetric, with skip connections between corresponding layers in each part. Within each path, the input passes through a series of convolutional layers, followed by a non-linear activation function (often ReLU), and other operations (e.g., max pooling) which are not encompassed by our model.
The long skip-connections that characterize U-nets require some modifications to the model of autoencoder described above.
If we denote with $\tilde{d}_i$ for $i=0, \ldots, r$ the dimensionality of each layer in the contracting path, we have that $\tilde d_{2r-i}=\tilde d_{i}$ for every $i=0, \ldots, r$.
Then, given an initial condition $z_0 \in \mathbb{R}^{\tilde{d}_0}$, we embed it into the augmented state variable 
\begin{equation*}
    x_0 = (z_0, \underline{0}), \mbox{ where } \underline{0} \in \mathbb{R}^{\tilde{d}_1 + \ldots + \tilde{d}_{r}}.
\end{equation*}
As done in the previous model for autoencoder, we consider time-nodes $0=t_0<\ldots<t_{2r}=T$, and in each sub-interval we introduce a controlled dynamics with the scheme of active/inactive components depicted in Figure \ref{fig:sketch_unet}.
\end{rmk}

\begin{figure}[ht!]
\begin{center}
    \includegraphics[scale=0.1]{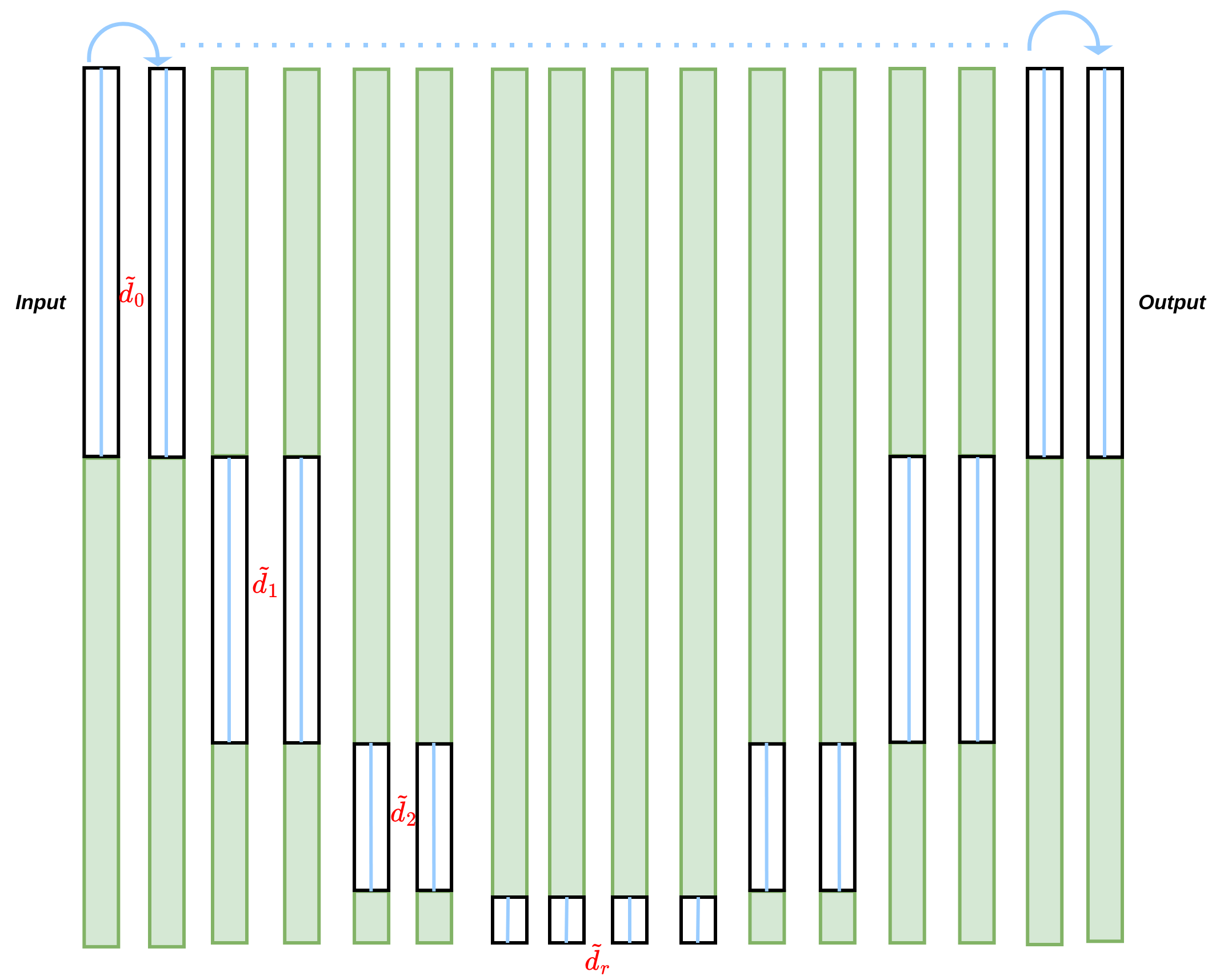}
\end{center}
\caption{Embedding of the U-net into a higher-dimensional dynamical system.}
\label{fig:sketch_unet}
\end{figure}

\section{Mean-field Analysis}\label{sec:mean-field}
In this section, we extend the dynamical model introduced in Section 2 to its mean-field limit, which corresponds to the scenario of an infinitely large dataset. Within this framework, we formulate the training of NeurODEs and AutoencODEs as a mean-field optimal control problem and provide the associated necessary optimality conditions. It is worth noting that our analysis covers both the high-regularized regime, as studied in previous work \cite{bonnet2023measure}, as well as the low-regularized regime, which has not been extensively addressed before. In this regard, we dedicate a subsection to a detailed comparison with the results obtained in \cite{bonnet2023measure}.
Additionally, we investigate the case of finite-particles approximation and we establish a quantitative bound on the generalization capabilities of these networks.

\subsection{Mean-field dynamical model}
In this section, we employ the same view-point as in \cite{bonnet2023measure}, and we consider the case of a dataset with an infinite number of observations. In our framework, each datum is modeled as a point $x_0 \in \R^d$, and it comes associated to its corresponding label $y_0 \in \R^{d}$. Notice that, in principle, in Machine Learning applications the label (or \textit{target}) datum $y_0$ may have dimension different from $d$. 
However, the labels' dimension is just a matter of notation and does not represent a limit of our model. Following \cite{bonnet2023measure}, we consider the curve  $t \mapsto (x(t),y(t))$ which satisfies
\begin{equation} \label{eq:ctrl_sys_extend}
\dot{x}(t)=\mathcal F(t,x(t),\theta(t)) \qquad \text{and} \qquad \dot{y}(t)=0
\end{equation}
for a.e. $t\in [0,T]$, and $(x(0),y(0))=(x_0,y_0)$.
We observe that the variable $y$ corresponding to the labels is not changing, nor it is affecting the evolution of the variable $x$.
We recall that the flow associated to the dynamics of the variable $x$ is denoted by $\Phi_{(0,t)}^\theta:\R^d \to \R^d$ for every $t\in [0,T]$, and it has been defined in \eqref{eq:flow}. 
Moreover, in regards to the full dynamics prescribed by \eqref{eq:ctrl_sys_extend}, for every admissible control $\theta\in L^2([0,T],\R^m)$ we introduce the extended flow $\BPhi_{(0,t)}^\theta:\R^d\times \R^d\to \R^d\times \R^d$,  which reads 
\begin{equation} \label{eq:flow_extend}
    \BPhi_{(0,t)}^\theta(x_0,y_0) = (\Phi^\theta_{(0,t)}(x_0), y_0)
\end{equation}
for every $t\in[0,T]$ and for every $(x_0,y_0)\in \R^d\times\R^d$. 
We now consider the case of an infinite number of labeled data $(X_0^i,Y_0^i)_{i\in I}$, where $I$ is an infinite set of indexes.  In our mathematical model, we understand this data distribution as a compactly-supported probability measure $\mu_0 \in \mathcal{P}_c(\R^{d}\times \R^{d})$. 
Moreover, for every $t \in [0,T]$, we denote by $t\mapsto\mu_t$ the curve of probability measures in $\mathcal{P}_c(\R^{d}\times \R^{d})$ that models the evolution of the solutions of \eqref{eq:ctrl_sys_extend} corresponding to the Cauchy initial conditions $(X_0^i,Y_0^i)_{i\in I}$. In other words, the curve $t\mapsto\mu_t$ satisfies the following continuity equation:
\begin{equation}
\label{eq:pde}
\partial_t\mu_t(x,y) + \nabla_x\cdot \big(\MF(t,x,\theta_t)\mu_t(x,y) \big)=0, \qquad \mu_t|_{t=0}(x,y)=\mu_0(x,y),
\end{equation}
understood in the sense of distributions, i.e.
\begin{definition}
\label{def:weak}
For any given $T>0$ and $\theta\in L^2([0,T],\R^m)$, we say that $\mu\in\mc{C}([0,T],\mc{P}_c(\RR^{2d}))$ is a weak solution of \eqref{eq:pde} on the time interval $[0,T]$ if   
\begin{equation}
\label{eqweak}
\int_0^T\int_{\RR^{2d}} \big( \partial_t \psi(t,x,y) + \nabla_x \psi(t,x,y) \cdot \MF(t,x,\theta_t) \big) \, d\mu_t(x,y)\, dt = 0,
\end{equation}
for every test function $\psi\in \MC_c^1((0,T)\times \RR^{2d})$.
\end{definition}
\noindent
Let us now discuss the existence and the characterisation of the solution.

\begin{proposition} \label{prop:exist_uniq_pde}
Under Assumptions \ref{ass:block1}, for every $\mu_0 \in \mc{P}_c(\R^{2d})$ we have that \eqref{eq:pde} admits a unique solution $t \mapsto \mu_t$ in the sense of Definition \ref{def:weak}. Moreover, we have that for every $t \in [0,T]$
\begin{equation}\label{eq:sol_flow}
    \mu_t = \BPhi_{(0,t) \#}^\theta\mu_0.
\end{equation}
\end{proposition}
\begin{proof}
Existence and uniqueness of the measure solution of \eqref{eq:pde} follow from \cite[Proposition~2.1, Theorem~3.1 and Remark~2.1]{ambrosio2008transport}.
\end{proof}

\noindent
From the characterisation of the solution of \eqref{eq:pde} provided in \eqref{eq:sol_flow}, it follows that the curve $t \mapsto \mu_t$ inherits the properties of the flow map $\Phi^\theta$ described in Proposition \ref{prop:flow_basics}. These facts are collected in the next result.

\begin{proposition}\label{prop:flow_pde_basics}
Let us fix $T > 0$ and $\mu_0 \in \mathcal{P}_c(\R^{2d})$,
and let us consider $\MF: [0, T ]\times \R^d \times \R^m \to \R^d$ satisfying Assumption \ref{ass:block1}. Let $\theta \in L^2([0,T], \R^m)$ be an admissible control, and let $t \mapsto \mu_t$ be the corresponding solution  of \eqref{eq:pde}. Then, the curve $t \mapsto \mu_t$ satisfies the properties listed below.
\begin{itemize}
    \item For every $R>0$ and $\rho>0$, there exists $\bar R>0$ such that, for every $t \in [0,T]$, it holds that
    \begin{equation*}
        \supp( \mu_t) \subset B_{\bar{R}}(0)
    \end{equation*}
     for every $\theta \in L^2([0,T],\R^m)$ such that $\|\theta \|_{L^2} \leq \rho$, and for every $\mu_0$ such that $\supp(\mu_0)\subset B_R(0)$.
    \item For every $R>0$ and $\rho>0$, there exists $\bar{L}>0$ such that, for every $t \in [0,T]$, it holds that
    \begin{equation*}
        W_1(\mu_t, \nu_t) \leq \bar{L} W_1(\mu_0,\nu_0)
    \end{equation*}
     for every $\theta \in L^2([0,T],\R^m)$ such that $\|\theta \|_{L^2} \leq \rho$, and for every initial conditions $\mu_0,\nu_0$ such that the supports satisfy $\supp(\mu_0),\supp(\nu_0)\subset B_R(0)$,
    where $\mu_t = \BPhi^\theta_{(0,t)\#}\mu_0$ and $\nu_t = \BPhi^\theta_{(0,t)\#}\nu_0$.
    \item For every $R>0$ and $\rho>0$, there exists $\bar{L}>0$ such that, for every $t_1,t_2 \in [0,T]$, it holds that
    \begin{equation*}
        W_1(\mu_{t_1}, \mu_{t_2}) \leq \bar{L}\cdot |t_1-t_2|^\frac{1}{2}
    \end{equation*}
    for every $\theta \in L^2([0,T],\R^m)$ such that $\|\theta \|_{L^2} \leq \rho$, and for every $\mu_0$ such that $\supp(\mu_0)\subset B_R(0)$.
    \item 
    For every $R>0$ and $\rho>0$, there exists $\bar{L}>0$ such that, for every $t \in [0,T]$, it holds that
    \begin{equation*}
        W_1(\mu_t, \nu_t) \leq \bar{L}  \|\theta_1- \theta_2\|_{L^2}
    \end{equation*}
    for every $\theta_1,\theta_2\in L^2([0,T],\R^m)$ such that $\|\theta \|_{L^2} , \| \theta_2\|_{L^2} \leq \rho$, and for every initial condition $\mu_0$ such that $\supp(\mu_0)\subset B_R(0)$, 
    where $\mu_t = \BPhi^{\theta_1}_{(0,t)\#}\mu_0$ and $\nu_t = \BPhi^{\theta_2}_{(0,t)\#}\mu_0$.
\end{itemize}
\end{proposition}
\begin{proof}
All the results follow from Proposition \ref{prop:exist_uniq_pde} and from the properties of the flow map presented in Proposition \ref{prop:flow_basics}, combined with the Kantorovich duality \eqref{eq:Kanto} for the distance $W_1$, and the change-of-variables formula \eqref{eq:push_forw}. Since the argument is essentially the same for all the properties, we detail the computations only for the second point, i.e., the Lipschitz-continuous dependence on the initial distribution.
Owing to \eqref{eq:Kanto},
for any $t \in [0,T]$, for any $\varphi \in \mathrm{Lip}(\R^{2d})$ such that its Lipschitz constant $Lip(\varphi)\leq1$, it holds that
\begin{equation*}
    \begin{split}
        W_1(\mu_t, \nu_t) &\leq \int_{\R^{2d}} \varphi(x,y) \,d(\mu_t - \nu_t)(x,y) 
         = \int_{\R^{2d}} \varphi(\Phi^\theta_{(0,t)}(x),y)\, d(\mu_0-\nu_0)(x,y) 
        \leq \bar{L} W_1(\mu_0,\nu_0),
    \end{split}
\end{equation*}
where the equality follows from the definition of push-forward and from \eqref{eq:flow_extend}, while the constant $\bar{L}$ in the second inequality descends from the local Lipschitz estimate of $\Phi^\theta_{(0,t)}$ established in Proposition \ref{prop:flow_basics}.
\end{proof}


\subsection{Mean-field optimal control}

Using the transport equation \eqref{eq:pde}, we can now formulate the mean-field optimal control problem that we aim to address. To this end, we introduce the functional $J:L^2([0,T],\R^m)\to \R$, defined as follows:
\begin{equation}
\label{eq:mf_functional}
 J(\theta) = 
\left\{
\begin{aligned}
& \int_{\RR^{2d}}\ell(x,y) \, d\mu_T(x,y)+\lambda\int_0^T|\theta(t)|^2 \, dt ,  \\
& \hspace{0.2cm} \text{s.t.} ~ \left\{
\begin{aligned}
& \partial_t\mu_t(x,y)+\nabla_x\cdot (\mathcal F(t,x,\theta_t)\mu_t(x,y))=0 & t\in[0,T],\\
& \mu_t|_{t=0}(x,y)=\mu_0(x,y),
\end{aligned}
\right.
\end{aligned}
\right.
\end{equation}
for every admissible control $\theta\in L^2([0,T],\R^m)$.
%
%
The objective is to find the optimal control $\theta^*$ that minimizes $J(\theta^*)$, subject to the PDE constraint \eqref{eq:pde} being satisfied by the curve $t\mapsto \mu_t$. The term "mean-field" emphasizes that $\theta$ is shared by an entire population of input-target pairs, and the optimal control must depend on the distribution of the initial data. We observe that when the initial measure $\mu_0$ is empirical, i.e.
\begin{equation*}
\mu_0 := \mu_0^N = \frac{1}{N} \sum_{i=1}^N \delta_{(X_0^i,Y_0^i)},
\end{equation*}
then minimization of \eqref{eq:mf_functional} reduces to a classical finite particle optimal control problem with ODE constraints. \\
We now state the further regularity hypotheses that we require, in addition to the one contained in Assumption~\ref{ass:block1}.
\begin{assum}
\label{ass:block2}
For any given $T>0$, the vector field $\MF$ satisfies the following.
\begin{enumerate}
\item[$(iv)$] For every $R>0$ there exists a constant  $L_{R} >0$ such that, for every $x_1,x_2 \in B_R(0) $, it holds
\begin{equation*}
|\nabla_x \MF(t,x_1,\theta)-\nabla_x \MF(t,x_2,\theta)| \leq L_R(1 + |\theta|^2)|x_1-x_2| ,\quad \mbox{ for a.e. } t\in [0,T] \mbox{ and every }\theta \in \R^m.
\end{equation*}
\item[$(v)$] There exists another constant  $L_{R} >0$ such that, for every $\theta_1,\theta_2 \in \R^m $, it holds
\begin{equation*}
|\nabla_\theta \MF(t,x,\theta_1)-\nabla_\theta \MF(t,x,\theta_2)| \leq L_R|\theta_1 - \theta_2| ,\quad \mbox{ for a.e. } t\in [0,T] \mbox{ and every }x \in B_R(0).
\end{equation*}
From this, it follows that
$|\nabla_\theta \MF(t,x,\theta)| \leq L_R(1+|\theta|)$ for every $x\in B_R(0)$ and for every $\theta\in \R^m$.
\item[$(vi)$] There exists another constant  $L_{R} >0$ such that, for every $\theta_1,\theta_2 \in \R^m $, it holds
\begin{equation*}
|\nabla_x \MF(t,x,\theta_1)-\nabla_x \MF(t,x,\theta_2)| \leq L_R(1 + |\theta_1| + |\theta_2|)|\theta_1-\theta_2| ,\quad \mbox{ for a.e. } t\in [0,T] \mbox{ and every }x \in B_R(0).
\end{equation*}
From this, it follows that
$|\nabla_x \MF(t,x,\theta)| \leq L_R(1+|\theta|^2)$ for every $x\in B_R(0)$ and for every $\theta\in \R^m$.
\item[$(vii)$] There exists another constant  $L_{R} >0$  such that 
\begin{equation*}
|\nabla_\theta \MF(t,x_1,\theta)-\nabla_\theta \MF(t,x_2,\theta)| \leq L_R(1 + |\theta|)|x_1-x_2| ,\quad \mbox{ for a.e. } t\in [0,T] \mbox{ and every }x_1,x_2 \in B_R(0).
\end{equation*}
\end{enumerate}	
\end{assum}
Additionally, it is necessary to specify the assumptions on the function $\ell$ that quantifies the discrepancy between the output of the network and its corresponding label.
\begin{assum}\label{ass:ell}
The function $\ell: \R^d\times \R^d \mapsto \R_+$ is $C^1$-regular and non-negative. Moreover, for every $R>0$, there exists a constant $L_R >0$ such that, for every  $x_1,x_2 \in B_R(0)$, it holds
\begin{equation}
    |\nabla_x\ell(x_1,y_1)-\nabla_x \ell(x_2,y_2)| \leq L_R \left( |x_1-x_2| + |y_1-y_2|\right).
\end{equation}
\end{assum}

Let us begin by establishing a regularity result for the reduced final cost, which refers to the cost function without the regularization term.
\begin{lem}[Differentiability of the cost]
\label{lem:FinalCostReg}
Let $T,R > 0$ and $\mu_0 \in \mathcal{P}_c(\R^{2d})$ be such that $\supp(\mu_0) \subset B_R(0)$, and let us consider $\MF:[0,T]\times \R^d\times \R^m\to \R^d$ and $\ell:\R^d\times\R^d\to \R$ that satisfy, respectively, Assumptions~\ref{ass:block1}-\ref{ass:block2}  and Assumption~\ref{ass:ell}. Then, the reduced final cost 
\begin{equation}
\label{eq:FinalReduced}
J_{\ell} : \theta \in L^2([0,T];\R^m) \mapsto
\left\{
\begin{aligned}
& \int_{\R^{2d}}\ell(x,y) \, d\mu_T^{\theta} (x,y),  \\
& \; \, \textnormal{s.t.} \,\,
\left\{ 
\begin{aligned}
& \partial_t \mu_t^{\theta}(x,y) + \nabla_x \big( \MF(t,x,\theta_t) \mu_t^{\theta}(x,y) \big) = 0, \\
& \mu_{t}^{\theta}|_{t=0}(x,y) = \mu_0(x,y), 
\end{aligned}
\right.
\end{aligned}
\right.
\end{equation}
is Fr\'echet-differentiable. Moreover, using the standard Hilbert space structure of $L^2([0,T],\R^m)$, we can represent the differential of $J_\ell$ at the point $\theta_0$ as the function:
\begin{equation}\label{eq:grad_J_ell}
   \nabla_{\theta} J_{\ell}(\theta) : t  \mapsto \int_{\R^{2d}} \nabla_{\theta} \MF ^ \top \big( t , \Phi^{\theta_0}_{(0,t)}(x),\theta(t) \big) \cdot \mathcal{R}^{\theta}_{(t,T)}(x)^{ \top}  \cdot \nabla_x \ell ^{ \top} \big(\Phi^{\theta}_{(0,T)}(x), y \big) \,d\mu_0(x,y)
\end{equation}
for a.e. $t\in[0,T]$.
\end{lem}
Before proving the statement, we need to introduce the linear operator $\mathcal{R}^\theta_{\tau,s}(x):\R^d \to \R^d$ with $\tau,s\in [0,T]$, that is related to the linearization along a trajectory of the dynamics of the control system \eqref{eq:ode}, and that appears in \eqref{eq:grad_J_ell}.
Given an admissible control $\theta\in L^2([0,T],\R^m)$, let us consider the corresponding trajectory curve $t\mapsto \Phi_{(0,t)}^\theta(x)$ for $t\in[0,T]$, i.e., the solution of \eqref{eq:ode} starting at the point $x\in \R^d$ at the initial instant $t=0$. Given any $\tau \in [0,T]$, we consider the following linear ODE in the phase space $\R^{d\times d}$:
\begin{equation} \label{eq:ode_R}
    \begin{cases}
      \frac{d}{ds}\mathcal{R}^\theta_{(\tau,s)}(x) = \nabla_x \MF(s, \Phi_{(0,s)}^\theta(x),\theta(s)) 
      \cdot \mathcal{R}_{(\tau,s)}^\theta(x) &
      \mbox{for a.e. } s\in[0,T],\\
      \mathcal{R}_{(\tau,\tau)}^\theta(x) =
      \mathrm{Id}.
    \end{cases}
\end{equation}
We insist on the fact that, when we write $\mathcal{R}_{(\tau,s)}^\theta(x)$, $x$ denotes the starting point of the trajectory along which the dynamics has been linearized.
We observe that, using Assumption \ref{ass:block2}-$(iv)-(vi)$ and Caratheodory Theorem, it follows that \eqref{eq:ode_R} admits a unique solution, for every $x\in \R^d$ and for every $\tau\in[0,T]$. Since it is an elementary object in Control Theory, the properties of $\mathcal{R}^\theta$ are discussed in the Appendix (see Proposition \ref{prop:resolvent}). We just recall here that the following relation is satisfied:
\begin{equation}\label{eq:different_flow}
    \mathcal{R}^\theta_{\tau,s}(x) =
    \nabla_x \Phi^\theta _{(\tau, s)}\big|_{\Phi^\theta _{(0, \tau)}(x)}
\end{equation}
for every $\tau,s \in [0,T]$ and for every $x\in \R^d$ (see, e.g., \cite[Theorem~2.3.2]{Bressan_Piccoli}). Moreover,  for every $\tau,s \in [0,T]$ the following identity holds:
\begin{equation*}
    \mathcal{R}^\theta_{\tau,s}(x)\cdot 
    \mathcal{R}^\theta_{s,\tau}(x) = \mathrm{Id},
\end{equation*}
i.e., the matrices $\mathcal{R}^\theta_{\tau,s}(x), \mathcal{R}^\theta_{s,\tau}(x)$ are one the inverse of the other. From this fact, it is possible to deduce that
\begin{equation} \label{eq:back_evol_R}
    \frac{\partial}{\partial \tau}\mathcal{R}^\theta_{\tau,s}(x)
    = -
    \mathcal{R}^\theta_{\tau,s}(x) \cdot
    \nabla_x \MF (\tau,\Phi_{(0,\tau)}^\theta(x),\theta(\tau))
\end{equation}
for almost every $\tau,s\in [0,T]$ (see, e.g., \cite[Theorem 2.2.3]{Bressan_Piccoli} for the details).

\begin{proof}[Proof of Lemma \ref{lem:FinalCostReg}]
Let us fix an admissible control $\theta \in L^2([0,T];\R^m)$ and let $\mu^{\theta}_\cdot \in \mathcal{C}^0([0,T];\mathcal{P}_c(\R^{2d}))$ be the unique solution of the continuity equation \eqref{eq:pde}, corresponding to the control $\theta$ and satisfying $\mu^\theta|_{t=0}=\mu_0$. According to Proposition~\ref{prop:exist_uniq_pde}, this curve can be expressed as $\mu_t^{\theta} = \BPhi^{\theta}_{(0,t)\#} \mu_0$ for every $t\in[0,T]$, where the map $\BPhi^{\theta}_{(0,t)}=(\Phi_{(0,t)}^\theta,\mathrm{Id}):\R^{2d}\to\R^{2d}$ has been introduced in \eqref{eq:flow_extend} as the flow of the extended control system \eqref{eq:ctrl_sys_extend}. 
In particular, we can rewrite the terminal cost $J_\ell$ defined in \eqref{eq:FinalReduced} as 
\begin{equation*}
J_{\ell}(\theta) = \int_{\R^{2d}} \ell( \Phi^{\theta}_{(0,T)}(x),y )d\mu_0(x,y). 
\end{equation*}
In order to compute the gradient $\nabla_\theta J_\ell$, we preliminarily need to focus on the differentiability with respect to $\theta$ of the mapping $\theta\mapsto \ell( \Phi_{(0,T)}^\theta(x),y)$, when $(x,y)$ is fixed. Indeed, given another control $\vartheta \in L^2([0,T];\R^m)$ and $\varepsilon > 0$, from Proposition \ref{prop:diff_ctrl} it descends that
\begin{equation} \label{eq:tayl_exp_flow}
\begin{split}
\Phi^{\theta+\varepsilon \vartheta}_{(0,T)}(x) &= \Phi^{\theta}_{(0,T)}(x) + \varepsilon \xi^\theta(T) + o_{\theta}(\varepsilon) \\
&= \Phi^{\theta}_{(0,T)}(x) + \varepsilon \int_0^T\mathcal{R}^{\theta}_{(s,T)}(x) \nabla_{\theta} \MF ( s , \Phi^{\theta}_{(0,s)}(x) , \theta(s) ) \vartheta(s) ds + o_{\theta}(\varepsilon) 
\end{split} \qquad \mbox{ as } \varepsilon\to 0,
\end{equation}
where $o_{\theta}(\varepsilon)$ is uniform for every $x \in B_R(0)\subset \R^d$, and as $\vartheta$ varies in the unit ball of $L^2$. 
Owing to Assumption~\ref{ass:ell}, for every $x,y,v \in B_R(0)$ we observe that 
\begin{equation} \label{eq:tayl_exp_ell}
    |\ell(x+\varepsilon v +o(\varepsilon),y) -\ell(x,y) -\varepsilon \nabla_x \ell(x,y)\cdot v| \leq |\nabla_x \ell(x,y)| o(\varepsilon) + \frac{1}{2}L_{R}|\varepsilon v + o(\varepsilon)|^2 \qquad
    \mbox{as } \varepsilon \to 0.
    \end{equation}
Therefore, combining \eqref{eq:tayl_exp_flow} and \eqref{eq:tayl_exp_ell}, we obtain that
\begin{equation*}
    \ell(\Phi^{\theta + \varepsilon \vartheta}_{(0,T)}(x),y) - \ell(\Phi^\theta_{(0,T)}(x),y) = 
      \varepsilon \int_0^T\big( \nabla_x \ell(\Phi^\theta_{(0,T)}(x),y) \cdot \mathcal{R}^\theta_{(s,T)}(x)
      \cdot \nabla_\theta \MF(s, \Phi^\theta_{(0,s)}(x),\theta(s)) \big) \cdot \vartheta(s) ds + o_\theta(\varepsilon).
\end{equation*}
Since the previous expression is uniform for $x,y\in B_R(0)$, then if we integrate both sides of the last identity with respect to $\mu_0$, we have that
\begin{equation} \label{eq:expans_J_ell}
    J_\ell(\theta +\varepsilon\vartheta) 
    - J_\ell(\theta) = 
          \varepsilon \int_{\R^{2d}}\int_0^T\big( \nabla_x \ell(\Phi^\theta_{(0,T)}(x),y) \cdot \mathcal{R}^\theta_{(s,T)}(x)
          \cdot \nabla_\theta \MF(s, \Phi^\theta_{(0,s)}(x),\theta(s)) \big) \cdot \vartheta(s) \, ds\, d\mu_0(x,y) + o_\theta(\varepsilon).
\end{equation}
This proves the Fr\'echet differentiability of the functional $J_\ell$ at the point $\theta$.
We observe that, from Proposition~\ref{prop:flow_basics}, Proposition~\ref{prop:resolvent} and Assumption~\ref{ass:block2}, it follows that the function $s\mapsto \nabla_x \ell(\Phi^\theta_{(0,T)}(x),y) \cdot \mathcal{R}^\theta_{(s,T)}(x)\cdot \nabla_\theta \MF(s, \Phi^\theta_{(0,s)}(x),\theta(s))$ is uniformly bounded in $L^2$, as $x,y$ vary in $B_R(0)\subset \R^d$. Then, using Fubini Theorem, the first term of the expansion \eqref{eq:expans_J_ell} can be rewritten as
\begin{equation*}
   \int_0^T \left( \int_{\R^{2d}}\nabla_x \ell(\Phi^\theta_{(0,T)}(x),y) \cdot \mathcal{R}^\theta_{(s,T)}(x)
          \cdot \nabla_\theta \MF(s, \Phi^\theta_{(0,s)}(x),\theta(s)) \, d\mu_0(x,y) \right) \cdot \vartheta(s) \, ds.
\end{equation*}
Hence, from the previous asymptotic expansion and from Riesz Representation Theorem, we deduce \eqref{eq:grad_J_ell}.
\end{proof}

We now prove the most important result of this subsection, concerning the Lipschitz regularity of the gradient $\nabla_\theta J_\ell$.

\begin{proposition} \label{prop:lip_grad}
Under the same assumptions and notations as in Lemma~\ref{lem:FinalCostReg}, we have that the gradient $\nabla_\theta J_\ell:L^2([0,T],\R^m)\to L^2([0,T],\R^m)$ is Lipschitz-continuous on every bounded set of $L^2$.
More precisely, given $\theta_1,\theta_2 \in L^2([0,T];\R^m)$, there exists a constant $\mathcal{L}(T,R,\|\theta_1\|_{L^2}, \| \theta_2\|_{L^2}) > 0$ such that 
\begin{equation*}
\big\| \nabla_{\theta} J_{\ell}(\theta_1) - \nabla_{\theta} J_{\ell}(\theta_2) \big\|_{L^2} \leq \mathcal{L}(T,R,\|\theta_1\|_{L^2}, \| \theta_2\|_{L^2}) \, \big\| \theta_1 - \theta_2 \big\|_{L^2}.
\end{equation*}
\end{proposition}

\begin{proof}
%
Let us consider two admissible controls $\theta_1,\theta_2 \in L^2([0,T],\R^m)$ such that $\| \theta_1 \|_{L^2}, \| \theta_1 \|_{L^2} \leq C$. In order to simplify the notations, given $x\in B_R(0)\subset \R^d$, we define the curves $x_1:[0,T]\to\R^d$ and $x_2:[0,T]\to\R^d$ as 
\begin{equation*}
    x_1(t) := \Phi_{(0,t)}^{\theta_1}(x), \quad
    x_2(t) := \Phi_{(0,t)}^{\theta_2}(x)
\end{equation*}
for every $t\in[0,T]$, where the flows $\Phi^{\theta_1},\Phi^{\theta_2}$ where introduced in \eqref{eq:flow}. We recall that, in virtue of Proposition \ref{prop:flow_basics}, $x_1(t),x_2(t)\in B_{\bar R}(0)$ for every $t\in[0,1]$.
Then, for every $y\in B_R(0)$, we observe that
\begin{equation} \label{eq:comput_lip_grad_triang}
\begin{split}
    \Big|
    \nabla_{\theta} \MF ^ \top \big( t , x_1(t),  &\theta_1(t) \big)    \mathcal{R}^{\theta_1}_{(t,T)}(x)^{ \top}   \nabla_x \ell ^{ \top} \big(x_1(T), y \big) - 
    \nabla_{\theta} \MF ^ \top \big( t , x_1(t),\theta_2(t) \big)  \mathcal{R}^{\theta_2}_{(t,T)}(x)^{ \top}   \nabla_x \ell ^{ \top} \big(x_2(T), y \big)
    \Big|\\ &
     \leq 
    \left|
    \nabla_{\theta} \MF ^ \top \big( t , x_1(t),\theta_1(t) \big) \right|
    \left|
    \mathcal{R}^{\theta_1}_{(t,T)}(x)^{ \top}
    \right|
    \left|
    \nabla_x \ell ^{ \top} \big(x_1(T), y \big) - \nabla_x \ell ^{ \top} \big(x_2(T), y \big)
    \right|\\ &
    \quad  +
    \left|
    \nabla_{\theta} \MF ^ \top \big( t , x_1(t),\theta_1(t) \big) 
    \right|
        \left|
    \mathcal{R}^{\theta_1}_{(t,T)}(x)^{ \top} - \mathcal{R}^{\theta_2}_{(t,T)}(x)^{ \top}
    \right|
    \left|
    \nabla_x \ell ^{ \top} \big(x_2(T), y \big)
    \right|
    \\ &
    \quad  +
        \left|
    \nabla_{\theta} \MF ^ \top \big( t , x_1(t),\theta_1(t) \big) -
    \nabla_{\theta} \MF ^ \top \big( t , x_2(t),\theta_2(t) \big) 
    \right|
    \left|
    \mathcal{R}^{\theta_2}_{(t,T)}(x)^{ \top}
    \right|
    \left|
    \nabla_x \ell ^{ \top} \big(x_2(T), y \big)
    \right|
\end{split}
\end{equation}
for a.e. $t\in [0,T]$. We bound separately the three terms at the right-hand side of \eqref{eq:comput_lip_grad_triang}. As regards the first addend, from Assumption \ref{ass:block2}-$(v)$, Assumption \ref{ass:ell}, Proposition \ref{prop:resolvent} and Lemma \ref{lem:lip_flow_ctrl}, we deduce that there exists a positive constant $C_1>0$ such that
\begin{equation}\label{eq:comput_lip_grad_1}
    \begin{split}
    \left|
    \nabla_{\theta} \MF ^ \top \big( t , x_1(t),\theta_1(t) \big) \right|
    \left|
    \mathcal{R}^{\theta_1}_{(t,T)}(x)^{ \top}
    \right|&
    \left|
    \nabla_x \ell ^{ \top} \big(x_1(T), y \big) - \nabla_x \ell ^{ \top} \big(x_2(T), y \big)
    \right|\\
    &\leq 
    C_1 \left( 1 + |\theta_1(t)| \right)
    \| \theta_1 -\theta_2 \|_{L^2}
    \end{split}
\end{equation}
for a.e. $t\in [0,T]$. Similarly, using again Assumption \ref{ass:block2}-$(v)$, Assumption \ref{ass:ell}, and Proposition \ref{prop:resolvent} on the second addend at the right-hand side of \eqref{eq:comput_lip_grad_triang}, we obtain that there exists $C_2>0$ such that
\begin{equation}\label{eq:comput_lip_grad_2}
    \left|
    \nabla_{\theta} \MF ^ \top \big( t , x_1(t),\theta_1(t) \big) \right|
    \left|
    \mathcal{R}^{\theta_1}_{(t,T)}(x)^{ \top} - \mathcal{R}^{\theta_2}_{(t,T)}(x)^{ \top}
    \right|
    \left| \nabla_x \ell ^{ \top} \big(x_2(T), y \big)
    \right|
    \leq 
    C_2 \left( 1 + |\theta_1(t)| \right)
    \| \theta_1 -\theta_2 \|_{L^2}
\end{equation}
for a.e. $t\in [0,T]$.
Moreover, the third term can be bounded as follows: 
\begin{equation}\label{eq:comput_lip_grad_3}
\begin{split}
    |
     \nabla_{\theta} \MF ^ \top \big( t , x_1(t),\theta_1(t) \big) - & \nabla_{\theta} \MF ^ \top \big( t , x_2(t),\theta_2(t) \big) |
    \left|
    \mathcal{R}^{\theta_2}_{(t,T)}(x)^{ \top}
    \right|
    \left| \nabla_x \ell ^{ \top} \big(x_2(T), y \big)
    \right| \\
& 
\leq C_3 \Big[ (1 + |\theta_1(t)|) \| \theta_1 -\theta_2\|_{L^2} + |\theta_1(t) -  \theta_2(t)| \Big]
\end{split}
\end{equation}
for a.e. $t\in [0,T]$, where we used Assumption \ref{ass:block2}-$(v)-(vii)$, Proposition \ref{prop:resolvent} and Lemma \ref{lem:lip_flow_ctrl}. Therefore, combining \eqref{eq:comput_lip_grad_triang}-\eqref{eq:comput_lip_grad_3}, we deduce that
\begin{equation} \label{eq:comput_lip_grad_fin}
    \begin{split}
    \Big|
    \nabla_{\theta} \MF ^ \top \big( t , x_1(t),  \theta_1(t) \big)  &   \mathcal{R}^{\theta_1}_{(t,T)}(x)^{ \top}   \nabla_x \ell ^{ \top} \big(x_1(T), y \big) - 
    \nabla_{\theta} \MF ^ \top \big( t , x_1(t),\theta_2(t) \big)  \mathcal{R}^{\theta_2}_{(t,T)}(x)^{ \top}  \nabla_x \ell ^{ \top} \big(x_2(T), y \big)
    \Big|\\ &
    \leq \bar C \Big[ (1 + |\theta_1(t)|) \| \theta_1 -\theta_2\|_{L^2} + |\theta_1(t) -  \theta_2(t)| \Big]
    \end{split}
\end{equation}
for a.e. $t\in [0,T]$. We observe that the last inequality holds for every $x,y\in B_R(0)$. Therefore, if we integrate both sides of \eqref{eq:comput_lip_grad_fin} with respect to the probability measure $\mu_0$, recalling the expression of the gradient of $J_\ell$ reported in \eqref{eq:grad_J_ell}, we have that
\begin{equation}
    |\nabla_\theta J_\ell(\theta_1)[t] - \nabla_\theta J_\ell(\theta_1)[t]|
    \leq \bar C \Big[ (1 + |\theta_1(t)|) \| \theta_1 -\theta_2\|_{L^2} + |\theta_1(t) -  \theta_2(t)| \Big]
\end{equation}
for a.e. $t\in [0,T]$, and this concludes the proof.
\end{proof}

From the previous result we can deduce that the terminal cost $J_\ell:L^2([0,T],\R^m)\to \R$ is locally semi-convex.

\begin{cor}[Local semiconvexity of the cost functional]
\label{cor:Semiconvexity}
Under the same assumptions and notations as in Lemma~\ref{lem:FinalCostReg}, let us consider a bounded subset $\Gamma \subset L^2([0,T];\R^m)$. Then, $\nabla_\theta J:L^2([0,T])\to L^2([0,T])$ is Lipschitz continuous on $\Gamma$. Moreover, there exists a constant $\mathcal{L}(T,R,\Gamma) > 0$ such that the cost functional $J:L^2([0,T],\R^m)\to \R$ defined in \eqref{eq:mf_functional} satisfies the following semiconvexity estimate: 
\begin{equation}
\label{eq:Semiconvexity}
J \big((1-\zeta)\theta_1 + \zeta \theta_2 \big) \leq (1-\zeta) J(\theta_1) + \zeta J(\theta_2) - (2\lambda - \mathcal{L}(T,R,\Gamma)) \tfrac{\zeta(1-\zeta)}{2} \| \theta_1 - \theta_2\|_2^2 
\end{equation}
for every $\theta_1,\theta_2 \in \Gamma$ and for every $\zeta \in [0,1]$. In particular, if $\lambda > \tfrac{1}{2} \mathcal{L}(T,R,\Gamma)$, the cost functional $J$ is strictly convex over $\Gamma$.
\end{cor}

\begin{proof}
We recall that $J(\theta) = J_\ell(\theta) + \lambda\| \theta\|_{L^2}^2$, where $J_\ell$ has been introduced in \eqref{eq:FinalReduced}.
Owing to Proposition \ref{prop:lip_grad}, it follows that $\nabla_\theta J_\ell$ is Lipschitz continuous on $\Gamma$ with constant $\mathcal{L}(T,R,\Gamma)$. This implies that $J$ is Lipschitz continuous as well on $\Gamma$. Moreover, it descends that 
\begin{equation*}
    J_{\ell} \big( (1-\zeta)\theta_1 + \zeta \theta_2 \big) \leq 
    (1-\zeta) J_{\ell}(\theta_1) + \zeta J_{\ell}(\theta_2) + \mathcal{L}(T,R,\Gamma) \tfrac{\zeta(1-\zeta)}{2} \| \theta_1-\theta_2 \|_{L^2}^2
\end{equation*}
 for every $\theta_1,\theta_2\in\Gamma$ and for every $\zeta\in[0,1]$.
On the other hand, recalling that 
\begin{equation*}
\| (1-\zeta) \theta_1 + \zeta \theta_2 \|_{L^2}^2 = (1-\zeta)\|\theta_1\|_{L^2}^2 + \zeta\|\theta_2\|_{L^2}^2 - 
\zeta(1-\zeta) \| \theta_1 - \theta_2 \|_{L^2}^2
\end{equation*}
for every $\theta_1,\theta_2\in L^2$, we immediately deduce \eqref{eq:Semiconvexity}. 
\end{proof}

\begin{rmk} \label{rmk:exist_minim}
When the parameter $\lambda>0$ that tunes the $L^2$-regularization is large enough, we can show that the functional $J$ defined by \eqref{eq:mf_functional} admits a unique global minimizer.
Indeed, since the control identically $0$ is an admissible competitor, we have that
\begin{equation*}
    \inf_{\theta\in L^2} J(\theta) \leq J(0) = J_\ell(0),
\end{equation*}
where we observe that the right-hand side is not affected by the value of $\lambda$. Hence, recalling that $J(\theta) = J_\ell(\theta)+\lambda\| \theta \|_{L^2}^2$, we have that the sublevel set $\{ \theta: J(\theta) \leq J_\ell(0) \}$ is included in the ball $B_\lambda :=\{\theta: \| \theta\|_{L^2}^2 \leq \frac1\lambda J_\ell(0) \}$. Since these balls are decreasing as $\lambda$ increases, owing to Corollary \ref{cor:Semiconvexity}, we deduce that there exists a parameter $\bar \lambda>0$ such that the cost functional $J$ is strongly convex when restricted to $B_{\bar \lambda}$. 
Then, Lemma \ref{lem:FinalCostReg} guarantees that the functional $J:L^2([0,T],\R^m)\to\R$ introduced in \eqref{eq:mf_functional} is continuous with respect to the strong topology of $L^2$, while the convexity implies that it is weakly lower semi-continuous as well.
Being the ball $B_{\bar \lambda}$ weakly compact, we deduce that the restriction to $B_{\bar \lambda}$ of the functional $J$ admits a unique minimizer $\theta^*$. However, since $B_{\bar \lambda}$ includes the sublevel set $\{ \theta: J(\theta) \leq J_\ell(0) \}$, it follows that $\theta^*$ is actually the unique global minimizer.
It is interesting to observe that, even though $\lambda$ is chosen large enough to ensure existence (and uniqueness) of the global minimizer, it is not possible to conclude that the functional $J$ is globally convex. This is essentially due to the fact that Corollary \ref{cor:Semiconvexity} holds only on bounded subsets of $L^2$.
\end{rmk}

Taking advantage of the representation of the gradient of the terminal cost $J_\ell$ provided by \eqref{eq:grad_J_ell}, we can  formulate the necessary optimality conditions for the cost $J$ introduced in \eqref{eq:mf_functional}.
In order to do that, we introduce the function $p:[0,T]\times \R^d\times \R^d\to\R^d$  as follows:
\begin{equation} \label{eq:def_p}
    p_t(x,y) :=
     \nabla_x \ell (\Phi_{(0,T)}^\theta(x),y)
     \cdot \mathcal{R}^\theta_{(t,T)}(x),
\end{equation}
where $ \mathcal{R}^\theta_{(t,T)}(x)$ is defined according to \eqref{eq:ode_R}. We observe that $p$ (as well as $\nabla_x \ell$) should be understood as a row vector. 
Moreover, using \eqref{eq:back_evol_R}, we deduce that, for every $x,y\in \R^d$, the $t\mapsto p_t(x,y)$ is solving the following backward Cauchy problem:
\begin{equation}\label{eq:ode_p}
    \frac{\partial}{\partial t} p_t(x,y) =
      -p_t(x,y)\cdot \nabla_x 
      \MF(t,\Phi_{(0,t)}^\theta(x),\theta(t))
      ,\qquad
      p_T(x,y) = \nabla_x \ell (\Phi_{(0,T)}^\theta(x),y).
\end{equation}
Hence, we can equivalently rewrite $\nabla_\theta J_\ell$ using $p$:
\begin{equation}\label{eq:diff_Jell_with_p}
    \nabla_\theta J_\ell(\theta)[t]
    = \int_{\R^{2d}} \nabla_{\theta} \MF ^ \top \big( t , \Phi^{\theta}_{(0,t)}(x),\theta(t) \big) \cdot p^{\top}_t( 
    x,y) \,d\mu_0(x,y)
\end{equation}
for almost every $t\in [0,T]$.
Therefore, recalling that $J(\theta) = J_\ell(\theta) + \lambda\|\theta \|_{L^2}^2$, we deduce that the stationary condition $\nabla_\theta J(\theta^*) = 0$ can be rephrased as 
\begin{equation}\label{eq:nec_opt_cond}
    \left\{
    \begin{aligned}
      &\partial_t \mu^*_t(x,y) + \nabla_x\cdot \big(\mathcal F(t,x ,\theta^*(t))\mu_t^*(x,y)\big)=0, & \mu_t^*|_{t=0}(x,y)=\mu_0(x,y),\\
      &\partial_t p^*_t(x,y) =
      -p^*_t(x,y)\cdot \nabla_x 
      \MF(t,\Phi_{(0,t)}^{\theta^*}(x),\theta^*(t)),
      &
      p^*_t|_{t=T}(x,y) = \nabla_x \ell (\Phi_{(0,T)}^{\theta^*}(x),y), \\
      &\theta^*(t) = -\frac1{2\lambda} \int_{\R^{2d}} \nabla_\theta \MF ^ \top \big( t , \Phi^{\theta^*}_{(0,t)}(x),{\theta^*}(t) \big) \cdot p^{*\,\top}_t(x,y) \,d\mu_0(x,y). & 
    \end{aligned}
    \right.
\end{equation}

\begin{rmk}\label{rmk:role_p}
The computation of $p$ through the backward integration of \eqref{eq:ode_p} can be interpreted as the control theoretic equivalent of the ``back-propagation of the gradients''. We observe that, in order to check whether \eqref{eq:nec_opt_cond} is satisfied, it is sufficient to evaluate $p^{*}$ only on $\supp(\mu_0)$. Moreover, the evaluation of $p^{*}$ on different points $(x_1,y_1),(x_2,y_2)\in \supp(\mu_0)$ involves the resolution of two uncoupled backward ODEs. This means that, when dealing with a measure $\mu_0$ that charges only finitely many points, we can solve the equation \eqref{eq:ode_p} in parallel for every point in $\supp(\mu_0)$.
\end{rmk}

In virtue of Proposition \ref{prop:lip_grad}, we can study the gradient flow induced by the cost functional $J:L^2([0,T],\R^m)\to\R$  on its domain. More precisely, given an admissible control $\theta_0\in L^2([0,T],\R^m)$, we consider the gradient flow equation:
\begin{equation} \label{eq:grad_flow}
    \begin{cases}
      \dot \theta(\omega) = -\nabla_\theta J(\theta(\omega)) & \mbox{for } \omega\geq 0,\\
      \theta(0) =\theta_0.
    \end{cases}
\end{equation}
In the next result we show that the gradient flow equation \eqref{eq:grad_flow} is well-posed and that the solution is defined for every $\omega\geq 0$. In the particular case of linear-control systems, the properties of the gradient flow trajectories has been investigated in \cite{scag2022gradient}. 

\begin{lem} \label{lem:grad_flow}
Let $T,R > 0$ and $\mu_0 \in \mathcal{P}_c(\R^{2d})$ be a probability measure such that $\supp(\mu_0) \subset B_R(0)$, and let us consider $\MF:[0,T]\times \R^d\times \R^m\to \R^d$ and $\ell:\R^d\times\R^d\to \R$ that satisfy, respectively, Assumptions~\ref{ass:block1}-\ref{ass:block2}  and Assumption~\ref{ass:ell}.
Then, for every $\theta_0\in L^2([0,T],\R^m)$, the gradient flow equation \eqref{eq:grad_flow} admits a unique solution $\omega\mapsto \theta(\omega)$ of class $C^1$ that is defined for every $\omega\in [0,+\infty)$.
\end{lem}

\begin{proof}
Let us consider $\theta_0 \in L^2([0,T],\R^m)$, and let us introduce the sub-level set 
\[
\Gamma:=\{ \theta\in L^2([0,T],\R^m): J(\theta)\leq J(\theta_0) \},
\]
where $J$ is the functional introduced in \eqref{eq:mf_functional} defining the mean-field optimal control problem. Using the fact that the end-point cost $\ell:\R^d\times\R^d \to \R_+$ is non-negative, we deduce that $\Gamma\subset \{ \theta\in L^2([0,T],\R^m): \| \theta\|_{L^2}^2 \leq \frac1\lambda J(\theta_0) \}$. Hence, from Proposition \ref{prop:lip_grad} it follows that the gradient field $\nabla_\theta J$ is Lipschitz (and bounded) on $\Gamma$. Hence, using a classical result on ODE in Banach spaces (see, e.g., \cite[Theorem 5.1.1]{ladas1972differential}), it follows that the initial value problem \eqref{eq:grad_flow} admits a unique small-time solution $\omega\mapsto \theta(\omega)$ of class $C^1$ defined for $\omega\in[-\delta,\delta]$, with $\delta>0$.
Moreover, we observe that
\begin{equation*}
    \frac{d}{d\omega} J(\theta(\omega)) =
    \langle \nabla_\theta J(\theta(\omega)), \dot \theta(\omega)) \rangle
    = -\| \nabla_\theta J(\theta(\omega)) \|_{L^2} \leq 0,
\end{equation*}
and this implies that $\theta(\omega)\in \Gamma$ for every $\omega\in [0,\delta]$. Hence, it is possible to recursively extend the solution to every interval of the form $[0, M]$, with $M>0$.  
\end{proof}

We observe that, under the current working assumptions, we cannot provide any convergence result for the gradient flow trajectories. This is not surprising since, when the regularization parameter $\lambda>0$ is small, it is not even possible to prove that the functional $J$ admits minimizers. Indeed, the argument presented in Remark \ref{rmk:exist_minim} requires the regularization parameter $\lambda$ to be sufficiently large.

We conclude the discussion with an observation on a possible discretization of \eqref{eq:grad_flow}. If we fix a sufficiently small parameter $\tau > 0$, given an initial guess $\theta_0$, we can consider the sequence of controls $(\theta^\tau_k)_{k\geq0}\subset L^2([0,T],\R^m)$ defined through the Minimizing Movement Scheme:
\begin{equation}\label{eq:mms}
\theta^\tau_0 = \theta_0, \qquad
\theta^\tau_{k+1} \in \arg\min_\theta\left[ J(\theta) + \frac1{2\tau} \|\theta-\theta^\tau_k\|^2_{L^2}\right]
\quad \mbox{for every } k\geq0.
\end{equation}

\begin{rmk}\label{rmk:nec_cond_MMS}
We observe that the minimization problems in \eqref{eq:mms} are well-posed as soon as the functionals $\theta \mapsto J^\tau_{\theta_k}(\theta):= J(\theta) + \frac1{2\tau} \|\theta-\theta^\tau_k\|^2_{L^2}$ are strictly convex on the bounded sublevel set $K_{\theta_0} := \{ \theta: J(\theta)\leq J(\theta^\tau_0) \}$, for every $k\geq0$. Hence, the parameter $\tau>0$ can be calibrated by means of the estimates provided by Corollary \ref{cor:Semiconvexity}, considering the bounded set $K_{\theta_0}$.  
Then, using and inductive argument, it follows that, for every $k\geq0$, the functional $J^\tau_{\theta_k}:L^2([0,T],\R^m)\to\R$ admits a unique global minimizer $\theta^\tau_{k+1}$. 
Also for $J^\tau_{\theta_k}$ we can derive the necessary conditions for optimality satisfied by $\theta_{k+1}$, that are analogous to the ones formulated in \eqref{eq:nec_opt_cond}, and that descend as well from the identity $\nabla_\theta J^\tau_{\theta_k}(\theta^\tau_{k+1})=0$: 
\begin{equation}\label{eq:nec_opt_mms}
    \left\{
    \begin{aligned}
      &\partial_t \mu_t(x,y) + \nabla_x\cdot \big (\mathcal F(t,x ,\theta^\tau_{k+1}(t))\mu_t(x,y) \big)=0, & \hspace{-2.5cm} \mu_t|_{t=0}(x,y)=\mu_0(x,y),\\
      &\partial_t p_t(x,y) =
      -p_t(x,y)\cdot \nabla_x 
      \MF(t,\Phi_{(0,t)}^{\theta^\tau_{k+1}}(x),\theta^\tau_{k+1}(t)),
      & \hspace{-2.5cm}
      p_t|_{t=T}(x,y) = \nabla_x \ell (\Phi_{(0,T)}^{\theta^\tau_{k+1}}(x),y), \\
      &\theta^\tau_{k+1}(t) = -\frac1{1+ 2\lambda\tau} \left(
      \theta^\tau_{k}(t) -\tau 
      \int_{\R^{2d}} \nabla_\theta \MF ^ \top \big( t , \Phi^{\theta^\tau_{k+1}}_{(0,t)}(x),{\theta^\tau_{k+1}}(t) \big) \cdot p^{\top}_t(x,y) \,d\mu_0(x,y)\right). &
    \end{aligned}
    \right.
\end{equation}
Finally, we observe that the mapping $\Lambda^\tau:L^2([0,T],\R^m )\to L^2([0,T],\R^m )$ defined for a.e. $t\in [0,T]$ as 
\begin{equation}\label{eq:Lambda_fixed}
    \Lambda_{\theta^\tau_{k}}^\tau(\theta)[t] :=   -\frac1{1+ 2\lambda\tau} \left(
      \theta^\tau_{k}(t) -\tau 
      \int_{\R^{2d}} \nabla_\theta \MF ^ \top \big( t , \Phi^{\theta}_{(0,t)}(x),{\theta}(t) \big) \cdot p^{\top}_t(x,y) \,d\mu_0(x,y)
      \right)
\end{equation}
is a contraction on $K_{\theta_0}$ as soon as 
\begin{equation*}
    \frac\tau{1+ 2\lambda\tau} \mathrm{Lip}
    \left(\nabla_\theta J_\ell|_{K_{\theta_0}}
    \right)
    <1.
\end{equation*}
\end{rmk}

For every $\tau>0$ such that the sequence $(\theta_k^\tau)_{k\geq0}$ is defined, we denote with $\tilde \theta^\tau:[0,+\infty)\to L^2([0,T],\R^m)$ the piecewise affine interpolation obtained as
\begin{equation} \label{eq:aff_int_mms}
   \tilde  \theta^\tau(\omega) = 
    \theta_k^\tau + \frac{\theta^\tau_{k+1} - \theta^\tau_{k}}{\tau} (\omega - k\tau)
    \quad \mbox{for } \omega\in [k\tau, (k+1)\tau].
\end{equation}
We finally report a classical result concerning the convergence of the piecewise affine interpolation $\tilde \theta^\tau$ to the gradient flow trajectory solving \eqref{eq:grad_flow}.   

\begin{proposition}\label{prop:convergence_mms} Under the same assumptions and notations as in Lemma \ref{lem:grad_flow}, let us consider an initial point $\theta_0\in L^2([0,T],\R^m)$ and a sequence $(\tau_j)_{j\in \mathbb N}$ such that $\tau_j\to 0$ as $j\to \infty$, and let $(\tilde \theta^{\tau_j} )_{j\in \mathbb N}$ be the sequence of piecewise affine curves defined by \eqref{eq:aff_int_mms}.
Then, for every $\Omega>0$, there exists a subsequence $(\tilde \theta^{\tau_{j_k}} )_{k\in \mathbb N}$ converging uniformly on the interval $[0,\Omega]$ to the solution of \eqref{eq:grad_flow} starting from $\theta_0$.
\end{proposition}
\begin{proof}
The proof follows directly from \cite[Proposition 2.3]{santambrogio2017euclidean}.
\end{proof}

\subsection{Finite particles approximation} \label{subsec:fin_part}
In this section, we study the stability of the mean-field optimal control problem \eqref{eq:mf_functional} with respect to finite-samples distributions. More precisely, assume that we are given samples $\{(X_0^i,Y_0^i)\}_{i=1}^N$ of size $N \geq 1$ independently and identically distributed according to $\mu_0 \in \mathcal{P}_c(\R^{2d})$, and consider the empirical loss minimization problem
\begin{equation}
\label{eq:cost_finite_sample}
\inf_{\theta \in L^2([0,T];\R^m)} J^N(\theta) := 
\left\{
\begin{aligned}
& \frac{1}{N} \sum_{i=1}^N \ell \big(X^i(T),Y^i(T) \big) + \lambda \int_0^T|\theta(t)|^2\, dt \\
& \,\; \textnormal{s.t.} \,\,
\left\{
\begin{aligned}
& \dot X^i(t) = \MF(t,X^i(t),\theta(t) ), \hspace{1.5cm} \dot Y^i(t) =0, \\ 
& (X^i(t),Y^i(t))\big |_{t=0} = (X_0^i,Y_0^i), ~~ i \in \{1,\dots,N\}.
\end{aligned}
\right.
\end{aligned}
\right.
\end{equation}

By introducing the empirical measure $\mu_0^N \in \mathcal{P}_c^N(\R^{2d})$, defined as 
\begin{equation*}
\mu_0^N := \frac{1}{N} \sum_{i=1}^N \delta_{(X_0^i,Y_0^i)}, 
\end{equation*}
the cost function in \eqref{eq:cost_finite_sample} can be rewritten as
\begin{equation}\label{eq:J_finite}
    J^N(\theta) = \int_{\RR^{2d}} \ell(\Phi_{(0,T)}^\theta(x),y) \,d\mu_0^N(x,y)  + \lambda \|\theta\|_{L^2}^2
\end{equation}
for every $\theta\in L^2([0,T],\R^m)$, and the empirical loss minimization problem in \eqref{eq:cost_finite_sample} can be recast as a mean-field optimal control problem with initial datum $\mu^N_0$. 
In this section we are interested in studying the asymptotic behavior of the functional $J^N$ as $N$ tends to infinity. More precisely, we consider a sequence of probability measures $(\mu_0^N)_{N\geq1}$ such that $\mu_0^N$ charges uniformly $N$ points, and such that
\begin{equation*}
W_1(\mu_0^N,\mu_0) ~\underset{N \to +\infty}{\longrightarrow}~ 0. 
\end{equation*}
Then, in Proposition \ref{prop:unif_conv_J} we study the uniform convergence of $J^N$ and of $\nabla_\theta J^N$ to $J$ and $\nabla_\theta J^N$, respectively, where $J:L^2([0,T],\R^m)\to \R$ is the functional defined in \eqref{eq:mf_functional} and corresponding to the limiting measure $\mu_0$.
Moreover, in Theorem \ref{thm:gen_error}, assuming the existence of a region where the functionals $J^N$ are uniformly strongly convex, we provide an estimate of the so called \textit{generalization error} in terms of the distance $W_1(\mu_0^N,\mu_0)$.

\begin{proposition}[Uniform convergence of $J^N$ and $\nabla_\theta J^N$]\label{prop:unif_conv_J}
Let us consider a probability measure $\mu_0 \in \mathcal{P}_c(\R^{2d})$ and a sequence $(\mu_0^N)_{N\geq1}$ such that $\mu_0^N\in \mathcal{P}^N_c(\R^{2d})$ for every $N\geq1$. Let us further assume that $W_1(\mu_0^N,\mu_0)\to 0$ as $N\to\infty$, and that there exists $R>0$ such that $\supp(\mu_0), \supp(\mu^N_0) \subset B_R(0)$ for every $N\geq1$.
Given $T>0$, let $\MF :[0,T]\times \R^d \times \R^m \to \R^d$ and $\ell:\R^d \times \R^d \to \R$ satisfy, respectively, Assumptions \ref{ass:block1}-\ref{ass:block2} and Assumption \ref{ass:ell}, and let $J,J^N:L^2([0,T],\R^m)\to\R$ be the cost functionals defined in \eqref{eq:mf_functional} and \eqref{eq:cost_finite_sample}, respectively.
Then, for every bounded subset $\Gamma\subset L^2([0,T],\R^m)$, we have that
\begin{equation} \label{eq:unif_conv_J}
    \lim_{N\to\infty}\,\, \sup_{\theta\in \Gamma}
    |J^N(\theta)-J(\theta)| =0
\end{equation}
and 
\begin{equation} \label{eq:unif_conv_grad}
    \lim_{N\to\infty}\,\, \sup_{\theta\in \Gamma} \| \nabla_\theta J^N(\theta)  - \nabla_\theta J(\theta) \|_{L^2}=0,
\end{equation}
where $J$ was introduced in \eqref{eq:mf_functional}, and $J^N$ is defined as in \eqref{eq:J_finite}.
\end{proposition}
\begin{proof}
Since we have that $J(\theta) = J_\ell(\theta) + \lambda\|\theta \|_{L^2}^2$ and $J^N(\theta) = J_\ell^N(\theta) + \lambda\|\theta \|_{L^2}^2$, it is sufficient to prove \eqref{eq:unif_conv_J}-\eqref{eq:unif_conv_grad} for $J_\ell$ and $J_\ell^N$, where we set
\begin{equation*}
    J^N_{\ell}(\theta) := \int_{\RR^{2d}} \ell(\Phi_{(0,T)}^\theta(x),y) \, d\mu_0^N(x,y) 
\end{equation*}
for every $\theta\in L^2$ and for every $N\geq1$.
We first observe that, for every $\theta \in L^2([0,T],\R^{m})$ such that $\| \theta \|_{L^2}\leq \rho$, from Proposition \ref{prop:flow_pde_basics} it follows that 
$\supp(\mu_0), \supp(\mu^N_0) \subset B_{\bar R}(0)$, for some $\bar R>0$. Then, denoting with $t\mapsto \mu_t^N$ and $t\mapsto \mu_t$ the solutions of the continuity equation \eqref{eq:pde} driven by the control $\theta$ and with initial datum, respectively, $\mu_0^N$ and $\mu_0$, we compute
\begin{equation} \label{eq:unif_conv_J_comput}
\begin{split}
    |J^N_\ell(\theta)-J_\ell(\theta)| &= \left |\int_{\R^{2d}} \ell(\Phi_{(0,T)}^\theta(x),y) \,\big(d\mu_0^N-d\mu_0\big)(x,y) \right|
    = 
    \left |\int_{\R^{2d}} \ell(x,y) \,\big(d\mu_T^N-d\mu_T\big)(x,y) \right|
    \\
    & \leq \bar{L}_1  \bar{L}_2  W_1(\mu_0^N,\mu_0),
\end{split}
\end{equation}
where we have used \eqref{eq:flow_extend} and Proposition \ref{prop:exist_uniq_pde} in the second identity, and we have indicated with $\bar{L}_1$ the Lipschitz constant of $\ell$ on $B_{\bar R}(0)$, while $\bar{L}_2$ descends from the continuous dependence of solutions of \eqref{eq:pde} on the initial datum (see Proposition \ref{prop:flow_pde_basics}). We insist on the fact that both $\bar{L}_1,\bar{L}_2$ depend on $\rho$, i.e., the upper bound on the $L^2$-norm of the controls.\\
We now address the uniform converge of $\nabla_\theta J^N_\ell$ to $\nabla_\theta J_\ell$ on bounded sets of $L^2$. As before, let us consider an admissible control $\theta$ such that $\| \theta\|_{L^2}\leq \rho$.
Hence, using the representation provided in \eqref{eq:grad_J_ell}, for a.e. $t \in [0,T]$ we have:
\begin{equation}\label{eq:lip_cont_grad}
\begin{split}
     \big |\nabla_\theta J^N_\ell(\theta)[t] - & \nabla_\theta J_\ell(\theta)[t] \big | = \\ & \left |\int_{\R^{2d}} \nabla_{\theta} \MF ^ \top \big( t , \Phi^{\theta_0}_{(0,t)}(x),\theta_0(t) \big) \cdot \mathcal{R}^{\theta_0}_{(t,T)}(x)^{ \top}  \cdot \nabla_x \ell ^{ \top} \big(\Phi^{\theta_0}_{(0,T)}(x), y \big) \, \big (d\mu^N_0-d\mu_0 \big)(x,y) \right|,
\end{split}
\end{equation}
In order to prove uniform convergence in $L^2$ norm, we have to show that the integrand is Lipschitz-continuous in $(x,y)$ for a.e. $t\in[0,T]$, where the Lipschitz constant has to be $L^2$-integrable as a function of the $t$ variable.
First of all, combining Assumption \ref{ass:block2}$-(v)$  and Lemma \ref{lem:lip_flow_init},  we can prove that there exists constants $C_1,\bar L_3>0$ (depending on $\rho$) such that 
\begin{equation}\label{eq:comp_nabla_F}
    \begin{split}
        | \nabla_\theta \MF(t, \Phi^\theta_{(0,t)}(x), \theta(t))| \leq& C_1 (1+ |\theta(t)|),\\
        |\nabla_\theta \MF(t, \Phi^\theta_{(0,t)}(x_1), \theta(t)) - \nabla_\theta \MF(t, \Phi^\theta_{(0,t)}(x_2), \theta(t))| \leq &\bar{L}_3  \bar{L}_2  (1+|\theta(t)|)|x_1-x_2|
    \end{split}
\end{equation}
for a.e. $t\in[0,T]$. We recall that the quantity $\bar L_2>0$ (that already appeared in \eqref{eq:unif_conv_J_comput}) represents the Lipschitz constant of the flow $\Phi_{(0,t)}$ with respect to the initial datum.
Moreover, from Proposition \ref{prop:resolvent}, it descends that
\begin{equation}\label{eq:comp_R}
    \begin{split}
        |\mathcal{R}^\theta_{(t,T)}(x)| \leq & C_2, \\
        |\mathcal{R}^\theta_{(t,T)}(x_1)-\mathcal{R}^\theta_{(t,T)}(x_2)|\leq &\bar{L}_4  |x_1-x_2| 
    \end{split}
\end{equation}
for every $t\in[0,T]$, where the constants $C_2, \bar{L}_4$ both depend on $\rho$.
Finally, owing to Assumption \ref{ass:ell} and Proposition \ref{prop:flow_basics}, we deduce
\begin{equation}\label{eq:comp_nabla_ell}
    \begin{split}
        | \nabla_x \ell(\Phi^\theta_{(0,T)}(x),y)| \leq & C_3,\\
        |\nabla_x \ell(\Phi^\theta_{(0,T)}(x_1),y_1)-\nabla_x \ell(\Phi^\theta_{(0,T)}(x_2),y_2)| \leq & \bar{L}_5(\bar{L}_2 |x_1-x_2|+|y_1-y_2|)
    \end{split}
\end{equation}
for every $x,y \in B_{R}(0)$,
where the constants $C_3, \bar{L}_2$ and the Lipschitz constant $\bar{L}_5$ of $\nabla_x \ell$ depend, once again, on $\rho$.
Combining \eqref{eq:comp_nabla_F}, \eqref{eq:comp_R} and \eqref{eq:comp_nabla_ell}, we obtain that there exists a constant $\tilde L_\rho>0$ such that
 \begin{equation*}
     |\nabla_\theta J^N[t]-\nabla_\theta J[t]| \leq \tilde L_\rho (1 + |\theta(t)|) W_1(\mu_0^N, \mu_0),
 \end{equation*}
for a.e. $t\in[0,T]$. 
Observing that the right-hand side is $L^2$-integrable in $t$, the previous inequality yields
 \begin{equation*}
     \|\nabla_\theta J^N-\nabla_\theta J\|_{L^2} \leq \tilde L_\rho (1+\rho) W_1(\mu_0^N, \mu_0),
 \end{equation*}
and this concludes the proof.
\end{proof}

In the next result we provide an estimate of the \textit{generalization error} in terms of the distance $W_1(\mu_0^N,\mu_0)$. In this case, the important assumption is that there exists a sequence $(\theta^{*,N})_{N\geq1}$ of local minimizers of the functionals $(J^N)_{N\geq1}$, and that it is contained in a region where $(J^N)_{N\geq1}$ are uniformly strongly convex.

\begin{thm}\label{thm:gen_error}
Under the same notations and hypotheses as in Proposition \ref{prop:unif_conv_J}, let us further assume that the functional $J$ admits a local minimizer $\theta^*$ and, similarly, that, for every $N\geq1$, $\theta^{*,N}$ is a local minimizer for $J^N$. Moreover, we require that there exists a radius $\rho >0$ such that, for every $N \geq \bar{N}$, $\theta^{*,N} \in B_\rho(\theta^*)$ and the functional $J^N$ is $\eta-$strongly convex in $B_\rho(\theta^*)$, with $\eta>0$. 
Then, there exists a constant $C>0$ such that, for every $N\geq \bar N$,  we have
\begin{equation}\label{eq:gen_error}
\bigg|\int_{\mathbb R^{2d}} \ell(x,y) \, d\mu^{\theta^{*,N}}_T(x,y)- \int_{\mathbb R^{2d}} \ell(x,y) \,d\mu^{\theta^*}_T(x,y) \bigg| \leq C \left( W_1(\mu_0^N,\mu_0)
+ \frac1{\sqrt{\eta}}
\sqrt{W_1(\mu_0^N,\mu_0)}
\right).
\end{equation}
\end{thm}

\begin{proof}
According to our assumptions, the control $\theta^{*,N}\in B_\rho(\theta^*)$ is a local minimizer for $J^N$, and, being $J^N$ strongly convex on $B_\rho(\theta^*)$ for $N\geq\bar N$, we deduce that $\{ \theta^{*,N} \} =\arg\min_{B_\rho(\theta^*)}J^N$. 
Furthermore, from the $\eta$-strong convexity of $J^N$, it follows that for every $\theta_1, \theta_2 \in B_\rho(\theta^*)$, it holds
\begin{equation*}
    \langle \nabla_\theta J^N(\theta_1) - \nabla_\theta J^N(\theta_2), \, \theta_1-\theta_2 \rangle \geq \eta ||\theta_1-\theta_2||^2_{L^2}.
\end{equation*}
According to Proposition \ref{prop:unif_conv_J}, we can pass to the limit in the latter and deduce that
\begin{equation*}
    \langle \nabla_\theta J(\theta_1) -\nabla_\theta J(\theta_2), \, \theta_1-\theta_2 \rangle \geq \eta ||\theta_1-\theta_2||^2_{L^2}
\end{equation*}
for every $\theta_1,\theta_2\in B_\rho(\theta^*)$.
Hence, $J$ is $\eta$-strongly convex in $B_\rho(\theta^*)$ as well, and that $\{ \theta^{*} \} =\arg\min_{B_\rho(\theta^*)}J$.
Therefore, from the $\eta$-strong convexity of $J^N$ and $J$, we obtain
\begin{align*}
    J^N(\theta^*)-J^N(\theta^{*,N}) &\geq \frac{\eta}{2}||\theta^{*,N}- \theta^*||_{L^2}^2 \\
    J(\theta^{*,N})-J^N(\theta^*) &\geq \frac{\eta}{2}||\theta^{*,N}- \theta^*||_{L^2}^2.
\end{align*}
Summing the last two inequalities, we deduce that
\begin{equation}\label{eq:comp_theta_bound}
    \eta ||\theta^{*,N}- \theta^*||_{L^2}^2 \leq \left (J^N(\theta^*)-J(\theta^*) \right) + \left (J^N(\theta^{*,N})-J(\theta^{*,N}) \right)  \leq 2C_1  W_1(\mu_0^N, \mu_0),
\end{equation}
where the second inequality follows from the local uniform convergence of Proposition \ref{prop:unif_conv_J}. We are now in position to derive a bound on the generalization error:
\begin{equation}\label{eq:comp_gen_err}
\begin{split}
\left| \int_{\mathbb R^{2d}} \ell(x,y) \, \big (d\mu^{\theta^{*,N}}_T - d\mu^{\theta^*}_T \big )(x,y) \right| & = \left| \int_{\mathbb R^{2d}} \ell(\Phi_{(0,T)}^{\theta^{*,N}}(x),y) \, d\mu^{N}_0(x,y) - \int_{\mathbb R^{2d}} \ell(\Phi_{(0,T)}^{\theta^*}(x),y) \,d\mu_0(x,y) \right|\\
& \leq \int_{\R^{2d}} \left|\ell(\Phi^{\theta^{*,N}}_{(0,T)}(x),y) -\ell(\Phi^{\theta^{*}}_{(0,T)}(x),y)\right| \, d\mu_0^N(x,y)\\
& \quad + \left|\int_{\R^{2d}} \ell(\Phi^{\theta^{*}}_{(0,T)}(x),y) \big( d\mu_0^N(x,y)-d\mu_0(x,y) \big ) \right|\\
& \leq \bar{L}  \sup_{x \in \supp(\mu_0^N)} \left|\Phi^{\theta^{*,N}}_{(0,T)}(x)- \Phi^{\theta^*}_{(0,T)}(x)\right|   + \bar{L}_R  W_1(\mu_0^N, \mu_0),
\end{split}
\end{equation}
where $\bar{L}$ and $\bar{L}_R$ are constants coming from Assumption \ref{ass:ell} and Proposition \ref{prop:flow_basics}. Then, we combine Proposition~\ref{prop:flow_basics} with the estimate in \eqref{eq:comp_theta_bound}, in order to obtain
\begin{equation*}
    \sup_{x \in \supp(\mu_0^N)} \left|\Phi^{\theta^{*,N}}_{(0,T)}(x)- \Phi^{\theta^*}_{(0,T)}(x)\right| \leq C_2 \| \theta^{*,N}-\theta^*\|_{L^2} \leq C_2  \sqrt{\frac{2C_1}{\eta} W_1(\mu_0^N,\mu_0)}.
\end{equation*}
Finally, from the last inequality and \eqref{eq:comp_gen_err}, we deduce \eqref{eq:gen_error}.
\end{proof}

\begin{rmk}
Since the functional $J:L^2([0,T],\R^m)\to \R$ defined in \eqref{eq:mf_functional} is continuous (and, in particular, lower semi-continuous) with respect to the strong topology of $L^2$, the locally uniform convergence of the functionals $J^N$ to $J$ (see Proposition~\ref{prop:unif_conv_J}) implies that $J^N$ is $\Gamma$-converging to $J$ with respect to the strong topology of $L^2$. However, this fact is of little use, since the functionals $J,J^N$ are not strongly coercive.
On the other hand, if we equip $L^2$ with the weak topology, in general the functional $J$ is not lower semi-continuous. In our framework, the only circumstance where one can hope for $\Gamma$-convergence with respect to the weak topology corresponds to the highly-regularized scenario, i.e., when the parameter $\lambda>0$ is sufficiently large. Therefore, in the situations of practical interest when $\lambda$ is small, we cannot rely on this tool, and the crucial aspect is that the dynamics \eqref{eq:ode} is nonlinear with respect to the control variable. Indeed, in the case of affine-control systems considered in \cite{scag2023optimal}, it is possible to establish $\Gamma$-convergence results in the $L^2$-weak topology (see \cite{scag2023deep} for an application to diffeomorphisms approximation). 
Finally, we report that in \cite{thorpe2023deep}, in order to obtain the $L^2$-strong equi-coercivity of the functionals, the authors introduced in the cost the $H^1$-seminorm of the controls. 
\end{rmk}

\subsection{Convex regime and previous result}
In order to conclude our mean-field analysis, we now compare our results with the ones obtained in the similar framework of \cite{bonnet2023measure}, where the regularization parameter $\lambda$ was assumed to be \textit{sufficiently large}, leading to a convex regime in the sublevel sets (see Remark \ref{rmk:exist_minim}). We recall below the main results presented in \cite{bonnet2023measure}.

\begin{thm}\label{thm:main_old_resutls} Given $T, R, R_T>0 $, and an initial datum $\mu_0 \in \mathcal{P}_c(\R^{2d})$ with $\supp(\mu_0) \subset B_R(0)$, let us consider a terminal condition $\psi_T:\R^d\times\R^d\to\R$ such that $\supp(\psi_T)\subset B_{R_T}(0)$ and $\psi_T(x,y) = \ell(x,y)$ $\forall x,y \in B_R(0)$. Let $\MF$ satisfy \cite[Assumptions 1-2]{bonnet2023measure} and $\ell \in C^2(\R^d \times \R^d,\R)$. Assume further that $\lambda>0$ is large enough. Then, there exists a triple $(\mu^*, \theta^*, \psi^*) \in \mathcal{C}([0, T ], \mathcal{P}_c(\R^{2d})) \times Lip([0, T ], \R^m) \times \mathcal{C}^1([0, T ], \mathcal{C}_c^2(\R^{2d}))$ solution of 
\begin{equation}\label{eq:old_pmp}
    \left\{
    \begin{aligned}
    &\partial_t\mu_t^*(x,y) + \nabla_x\cdot (\mathcal F(t,x,\theta^*(t))\mu_t^*(x,y))=0, & \mu_t^*|_{t=0}(x,y)=\mu_0(x,y), \\
    &\partial_t\psi_t^*(x,y) + \nabla_x\psi_t^*(x,y)\cdot \mathcal F(t,x,\theta^\ast(t))=0, & \psi_t^*|_{t=T}(x,y)=\ell(x,y) , \\
    &\theta^{*\top}(t) = -\frac{1}{2\lambda}\int_{\RR^{2d}}\nabla_x\psi_t^*(x,y)\cdot \nabla_\theta \mathcal F(t,x,\theta^*(t)) \,d\mu_t^*(x,y), &
    \end{aligned}
    \right.
\end{equation}
where $\psi^* \in \mathcal{C}^1([0,T],\mathcal{C}^2_c(\R^{2d}))$ is in characteristic form. Moreover, the control solution $\theta^*$ is unique in a ball $\Gamma_C \subset L^2([0,T], \R^m)$ and continuously dependent on the initial datum $\mu_0$.
\end{thm}

We observe that the condition on $\lambda >0$ to be large enough is crucial to obtain local convexity of the cost functional and, consequently, existence and uniqueness of the solution. However, in the present paper we have not done assumptions on the magnitude of $\lambda$, hence, as it was already noticed in Remark~\ref{rmk:exist_minim}, we might end up in a non-convex regime. Nevertheless, in Proposition~\ref{prop:equivalence_pmp} we show that, in the case of $\lambda$ sufficiently large, the previous approach and the current one are ``equivalent".

\begin{proposition}\label{prop:equivalence_pmp}
Under the same hypotheses as in Theorem~\ref{thm:main_old_resutls}, let $J:L^2([0,T])\to \R$ be the functional defined in \eqref{eq:mf_functional}. Then, $\theta^*$ satisfies \eqref{eq:old_pmp} if and only if it is a critical point for $J$.
\end{proposition}
\begin{proof}
According to Lemma~\eqref{lem:FinalCostReg}, the gradient of the functional $J$ at $\theta \in L^2([0,T],\R^m)$ is defined for a.e. $t \in [0,T]$ as
\begin{equation*}
    \nabla_\theta J(\theta)[t] = \int_{\R^{2d}} \nabla_{\theta} \MF ^ \top \big( t , \Phi^{\theta}_{(0,t)}(x),\theta(t) \big) \cdot \mathcal{R}^{\theta}_{(t,T)}(x)^{ \top}  \cdot \nabla_x \ell ^{ \top} \big(\Phi^{\theta}_{(0,T)}(x), y \big) \,d\mu_0(x,y) + 2\lambda \theta(t).
\end{equation*}
Hence, if we set the previous expression equal to zero, we obtain the characterization of the critical point 
\begin{equation}\label{eq:comp_th_new}
    \theta(t) = -\frac{1}{2\lambda} \int_{\R^{2d}} \nabla_{\theta} \MF ^ \top \big( t , \Phi^{\theta}_{(0,t)}(x),\theta(t) \big) \cdot \mathcal{R}^{\theta}_{(t,T)}(x)^{ \top}  \cdot \nabla_x \ell ^{ \top} \big(\Phi^{\theta}_{(0,T)}(x), y \big) \,d\mu_0(x,y)
\end{equation}
for a.e. $t\in[0,T]$.
On the other hand, according to Theorem \ref{thm:main_old_resutls}, the optimal $\theta$ satisfies for a.e. $t \in [0,T]$ the following
\begin{equation}\label{eq:comp_th_old}
    \begin{split}
    \theta(t) &= -\frac{1}{2\lambda}\int_{\R^{2d}} \big( \nabla_x\psi_t(x,y)
    \cdot
    \nabla_\theta \MF(t,x, \theta(t)) \big )^{ \top} \,  d\mu_t(x,y) \\
    &= -\frac{1}{2\lambda} \int_{\R^{2d}} \nabla_\theta \MF^{ \top}(t,\Phi^\theta_{(0,t)}(x), \theta(t)) \cdot \nabla_x\psi^{ \top}_t(\Phi^\theta_{(0,t)}(x),y) \,d\mu_0(x,y).
    \end{split}
\end{equation}
Hence, to conclude that $\nabla_\theta J=0$ is equivalent to condition stated in Theorem \ref{thm:main_old_resutls}, we are left to show that
\begin{equation}\label{eq:comp_equiv}
    \mathcal{R}^{\theta}_{(t,T)}(x)^{\top}
    \cdot 
    \nabla_x \ell ^\top \big(\Phi^{\theta}_{(0,T)}(x), y \big) = \nabla_x \psi^\top_t( \Phi^\theta_{(0,t)}(x),y),
\end{equation}
where the operator $\mathcal{R}^{\theta}_{(t,T)}(x)$ is defined as the solution of \eqref{eq:ode_R}.
First of all, we recall that $(t,x,y)\mapsto\psi(t, \Phi^\theta_{(0,t)}(x),y)$ is defined as the characteristic solution of the second equation in \eqref{eq:old_pmp} and, as such, it satisfies
\begin{equation*}
\psi_t(x,y) = \ell (\Phi_{(t,T)}^\theta(x),y),
\end{equation*}
for every $t\in[0,T]$ and for every $x,y\in B_{R_T}(0)$. By taking the gradient with respect to $x$, we obtain that 
\begin{equation*}
    \nabla_x \psi_t(x,y) = \nabla_x \ell(\Phi_{(t,T)}^\theta(x),y) )\cdot \nabla_x    \Phi_{(t,T)}^\theta \big|_x,
\end{equation*}
for all $x,y\in B_{R_T}(0)$. Hence, using \eqref{eq:different_flow}, we deduce that
\begin{equation*}
\begin{split}
     \nabla_x \psi_t(\Phi_{(0,t)}^\theta(x),y)  & =
     \nabla_x \ell(\Phi_{(t,T)}^\theta\circ \Phi_{(0,t)}^\theta(x),y) )\cdot \nabla_x    \Phi_{(t,T)}^\theta \big|_{\Phi_{(0,t)}^\theta(x)}
      = \nabla_x \ell(\Phi_{(0,T)}^\theta(x),y) )\cdot \mathcal{R}^\theta_{(t,T)}(x)
\end{split}
\end{equation*}
which proves \eqref{eq:comp_equiv}.
\end{proof}


\section{Algorithm} \label{sec:Alg}
In this section, we present our training procedure, which is derived from the necessary optimality conditions related to the minimizing movement scheme (see \eqref{eq:nec_opt_mms}). 
Since the mean-field optimal control problem as presented in \eqref{eq:mf_functional} is numerically intractable (especially in high-dimension), in the practice we always consider the functional corresponding to the finite particles approximation (see \eqref{eq:cost_finite_sample}).
For its resolution, we employ an algorithm belonging to the family of shooting methods, which consists in the forward evolution of the trajectories, in the backward evolution of the adjoint variables, and in the update of the controls.
Variants of this method have already been employed in different works, e.g. \cite{li2017maximum, haber2017stable, benning2019deep, bonnet2023measure, chernousko1982method}, with the name of \textit{method of successive approximations}, and they have been proven to be an alternative way of performing the training of NeurODEs for a range of tasks, including high-dimensional problems.\\
In our case, we start with a random guess for the control parameter $\theta_0\in L^2([0,T],\R^m)$. Subsequently, we solve the necessary optimality conditions specified in equation \eqref{eq:nec_opt_mms} for a suitable $\tau>0$ to obtain an updated control parameter $\theta_1$. More precisely, since the last identity in \eqref{eq:nec_opt_mms} has the form $\theta_1=\Lambda_{\theta_0}^{\tau}(\theta_1)$, the computation of $\theta_1$ is performed via fixed-point iterations of the mapping $\Lambda_{\theta_0}^\tau$, which is defined as in \eqref{eq:Lambda_fixed}. In this regard, we recall that $\Lambda_{\theta_0}^\tau$ is a contraction if $\tau$ is small enough. The scheme that we implemented is presented is Algorithm~\ref{alg:shooting}.

\begin{algorithm}[ht!]
\scriptsize
\caption{Shooting method}\label{alg:shooting}
\KwData{$\{(X^i_0, Y^i_0) \}^N_{i=0}$ data with labels\; $\MF$ controlled vector field \; $\theta^0$ initial guess for controls\; $N_{iter}$ number of shooting iterations\; $dt$ time-discretization of the interval $[0,T]$\;
$\lambda$ regularization parameter\;
$\tau$ memory parameter\;
}
\KwResult{$\theta^{N_{iter}}$}
$N_t \gets \frac{T}{dt}$\;
\For{$k = 1, \ldots, N_{iter}$}{
\For{$j = 1, \ldots, N_t$}{
    \For{$i = 1, \ldots, N$}{
        $X^i(t_{j+1}) \gets X^i(t_j) + dt\, \MF(t_j, X^i(t_j), \theta^k(t_j))$ \Comment*[r]{Solve forward }
    }
}
\For{$i=1, \ldots, N$}{ 
    $P^i(N_t) \gets -\nabla_x \ell(X^i(N_t), Y^i(0))$ \Comment*[r]{Update the co-state at the final time}
}
\For{$j = N_t, \ldots, 1$}{
    \For{$i = 1, \ldots, N$}{
        $P^i(t_{j}) \gets P^i(t_{j+1}) + dt\, P^i(t_{j+1})\cdot \nabla_x\MF(t_{j+1}, X^i(t_{j+1}), \theta^k(t_{j+1}))$ \Comment*[r]{Solve backward}
    }
}
\For{$j = 1, \ldots, N_t$}{
    $I^{\theta^{k+1}} \gets \sum_{i=0}^N \nabla_\theta \MF(t_j, X^i(t_j), \theta^{k+1}(t_j))^{\top} \cdot  P^i(t_j)^{\top}$ \Comment*[r]{Approximate the integral \eqref{eq:diff_Jell_with_p}}
    $\theta^{k+1}(t_j) \gets \Lambda \big [- \frac{1}{1 +2\lambda\tau}( \theta^k(t_j) - \frac{\tau}{N} I^{\theta^{k+1}}) \big ]$ \Comment*[r]{Update the control via fixed-point $\Lambda$}
    }
}
\end{algorithm}

\begin{rmk} \label{rmk:high_reg_alg}
It is interesting to observe that, in the \textit{highly-regularized} regime considered in \cite{bonnet2023measure}, the authors managed to obtain a contractive map directly from the necessary conditions for optimality, and they did not need to consider the minimizing movements scheme. This is rather natural since, when the parameter $\lambda>0$ that tunes the $L^2$-penalization is large enough, the functional associated to the optimal control problem is strongly convex in the sublevel set corresponding to the control $\theta\equiv 0$, as discussed in Remark \ref{rmk:exist_minim}.
However, as reported in \cite{bonnet2023measure}, determining the appropriate value for $\lambda$ in each application can be challenging. On the other hand, from the practitioners' perspective, dealing with high regularization is not always desirable, since the machine learning task that the system should learn is encoded in the final-time cost. The authors highlighted the complexity involved in selecting a regularization parameter that is large enough to achieve contractivity, while ensuring that the resulting controls are not excessively small (due to high regularization) and of little use.
\end{rmk}
These considerations motivated us to consider a scenario where the regularization parameter does not need to be set sufficiently large.
From a numerical perspective, the parameter $\tau$ in Equation \eqref{eq:nec_opt_mms} (coming from the minimizing movement scheme) plays the role of the \textit{learning rate}, and it provides the lacking amount of convexity, addressing the stability issues related to the resolution of optimal control problems in non-convex regime.
These kinds of instabilities were already known in the Soviet literature of numerical optimal control (see the review paper \cite{chernousko1982method}), and various solutions have been proposed to address them.
For example, in \cite{sakawa1980global} the authors proposed an iterative method based on the Maximum Principle and on an augmented Hamiltonian, with an approach that is somehow reminiscent of minimizing movements.
More recently, in the framework of NeurODEs, in \cite{li2017maximum} it was proposed another stabilization strategy, which is different from ours since it enforces similarity between the evolution of state and co-state variables after the control update. 
Implicitly, the approach of \cite{li2017maximum} 
leads to a penalization of significant changes in the controls. On the other hand, in our approach this penalization is more explicit, and it is enforced via the memory term of the minimizing movement scheme. To the best of our knowledge, this is the first instance where a regularization based on the minimizing movement scheme is employed for training NeurODEs.

\begin{rmk}
Although we formulate and analyze theoretically our problem within the mean-field framework, it is not advantageous to numerically solve the forward equation as a partial differential equation. In \cite{bonnet2023measure}, various numerical methods for solving PDEs were employed and compared. However, these methods encounter limitations when applied to high-dimensional data, which is often the case in Machine Learning scenarios. Therefore, in this study, we employ a particle method to solve both the forward partial differential equation and the backward dynamics. This particle-based approach involves reformulating the PDE as a system of ordinary differential equations in which particles represent mathematical collocation points that discretize the continuous fields. By employing this particle method, we address the challenges associated with high-dimensional data, enabling efficient numerical solutions for the forward and backward dynamics.
\end{rmk}
To conclude this section, we briefly present the forward and the backward systems that are solved during the execution of the method. For the sake of simplicity, we will focus on the case of an encoder. The objective is to minimize the following function:
\begin{equation}\label{eq:loss_th_zero}
    J(\theta) = \frac{1}{N} \sum_{i=1}^{N} \ell(X^i_{\A_r}(T), Y^i(0)) + \frac{\lambda}{2} \norm{\theta}^2_2,
\end{equation}
where $\A_r$ denotes the active indices in the bottleneck, i.e. at $t_r=T$, of the state-vector  $X^i(T)$. The latter denotes the encoded output at time $T$ for the $i$-th particle, while $Y^i(0)$ represents the corresponding target at time $0$ (which we recall is the same at time $T$, being $\dot Y^i \equiv 0$ for every $i=1,\ldots,N$). For each $i$-th particle and every $t$ such that $t_j \leq t \leq t_{j+1}$, the forward dynamics can be described as follows:
\begin{equation}\label{eq:fwd_const}
    \begin{cases}
        \dot{X}^i_{\I_j}(t) = 0, \\
        \dot{X}^i_{\A_j}(t) = \MG_j\big(t, X^i_{\A_j}(t), \theta(t) \big),
    \end{cases}
\end{equation}
subject to the initial condition $X^i(0) = X_{\A_0}^i(0) = X_0^i \in \RR^d$. In the same interval $t_j \leq t \leq t_{j+1}$, the backward dynamics reads
\begin{equation}\label{eq:bkw_const}
    \begin{cases}
    \dot{P}^i_{\I_j}(t) = 0, \\
    \dot{P}^i_{\A_j}(t) = -P^i_{\A_j}(t) \cdot \nabla_{x_{\A_j}} \MG_j \big(t, X^i_{\A_j}(t), \theta(t) \big),
    \end{cases}
\end{equation}
where the final co-state is 
\begin{equation*}
    P^i(T) = 
\begin{cases}
 - \partial_k \ell(X^i_{\A_r}(T),Y^i(0)), & \text{if } k \in \A_r,\\
 0, & \text{if } k \notin \A_r.
\end{cases}
\end{equation*}
We notice that, for $t_j \leq t \leq t_{j+1}$ and every $i \in \{0, \ldots, N \}$, we have
\begin{equation}\label{eq:sparse_F}
    \MF(t,X^i(t),\theta(t)) = \MF(t, (X^i_{\A_r},X^i_{\I_j})(t), \theta(t)) =
\begin{pmatrix} \MG_j(t, X^i_{\A_j}(t),\theta(t))\\  0 \end{pmatrix},
\end{equation}
and, consequently , we deduce that
\begin{equation}\label{eq:sparse_grad}
    \nabla_x \MF(t,X^i(t),\theta(t)) = \begin{pmatrix} \nabla_{x_{\A_j}} \MG_j(t,X^i_{\A_j}(t),\theta(t)) & 0 \\ 0 & 0 \end{pmatrix},
\end{equation}
where the null blocks are due to the fact that, for $t_j \leq t \leq t_{j+1}$, $\nabla_{x}\MF_k (t,x,\theta)=0$ if $k\in \I_j$, and $\nabla_{x_{\I_j}}\MG_j (t,x,\theta)=0$.
In the case of an Autoencoder, the structure of the forward and backward dynamics is analogous.

\begin{rmk} \label{rmk:overwriting}
From the calculations reported above it is evident that the matrices and the vectors involved in our forward and backward dynamics are quite sparse (see \eqref{eq:sparse_grad} and \eqref{eq:sparse_F}), and that the state and co-state variables contain components that are constant in many sub-intervals (see \eqref{eq:fwd_const} and \eqref{eq:bkw_const}). Hence, in the practical implementation, especially when dealing with an Autoencoder, we do not actually need to double the original state variables and to introduce the shadow ones, but we can simply overwrite those values and, in this way, we obtain a more memory-efficient code. A similar argument holds as well for the co-state variables.
Moreover, we expect the control variable $\theta$ to have several null components during the evolution. 
This relates to Remark \ref{rmk:dim_theta} and descends from the fact that, even though in our model $\theta \in \R^m$ for every $ t \in [0,T]$, in the internal sub-intervals $[t_j,t_{j+1}]$ only few of its components are influencing the dynamics.
Hence, owing to the $L^2$-squared regularization on $\theta$, it results that, if in an interval $[t_j,t_{j+1}]$ a certain component of $\theta$ is not affecting the velocity, then it is convenient to keep it null.   
\end{rmk}


\section{Numerical Experiments}\label{sec:num_exp}
In this section, we present a series of numerical examples to illustrate the practical application of our approach. We consider datasets of varying dimensions, ranging from low-dimensional data to a more typical Machine Learning dataset such as MNIST. Additionally, we provide justifications and insights into some of the choices made in our theoretical analysis. 
For instance, we examine the process of choosing the components to be deactivated during the modeling phase, and we investigate whether this hand-picked selection 
can lead to any issues or incorrect results. 
In this regard, in our first experiment concerning a classification task, we demonstrate that this a priori choice does not pose any problem, as the network effectively learns to separate the dataset into two classes before accurately classifying them.
Furthermore, as we already pointed out, we have extended some of the assumptions from \cite{bonnet2023measure} to accommodate the use of a smooth approximation of the ReLU function. This extension is not merely a theoretical exercise, since in our second numerical example we show how valuable it is to leverage unbounded activation functions.
While both of these examples involve low-dimensional data and may not be representative of typical tasks for an Autoencoder architecture, we address this limitation in our third experiment by performing a reconstruction task on the MNIST dataset. Lastly, we present noteworthy results obtained from analyzing the performance on MNIST, highlighting specific behaviors that warrant further investigation in future research.\\

\noindent
The layers of the networks that we employ in all our experiments have the form:
\begin{equation*}
    \R^d\ni X = \left(X_{\A_j}, X_{\I_j} \right)^\top \mapsto
    \phi_n^{W,b}(X) =
    \left(X_{\A_j}, X_{\I_j} \right)^\top
    + h \Big( \sigma\left(
    W_{\A_j}\cdot
    X_{\A_j} +
     b_{\A_j}\right), 0 \Big)^\top,
\end{equation*}
where $\A_j,\I_j$ are, respectively, the sets of active and inactive components at the layer $n$, $b_{\A_j}$ are the components of $b\in \R^d$ belonging to $\A_j$, while $W_{\A_j}$ is the square sub-matrix of $W\in \R^{d\times d}$ corresponding to the active components. Finally, the activation function $\sigma$ will be specified case by case.

\subsection*{Bidimensional Classification}
In our initial experiment, we concentrate on a bidimensional classification task that has been extensively described in \cite{bonnet2023measure}. Although this task deviates from the typical application of Autoencoders, where the objective is data reconstruction instead of classification, we believe it gives valuable insights on how our model works.
The objective is to classify particles sampled from a standard Gaussian distribution in $\R^2$ based on the sign of their first component. Given an initial data point $x_0 \in \R^2$, denoted by $x_0[i]$ with $i=1,2$ representing its $i-$th component, we assign a positive label $+1$ to it if $x_0[1] > 0$, and a negative label $-1$ otherwise. 
To incorporate the labels into the Autoencoder framework, we augment the labels to obtain a positive label $[1,0]$ and a negative one $[-1,0]$. In such a way, we obtain target vectors in $\R^2$, i.e., with the same dimension as the input data-points in the first layer.\\
The considered architecture is an Autoencoder comprising twenty-one layers, corresponding to $T=2$ and $dt=0.05$. The first seven layers maintain a constant active dimension equal to $2$, followed by seven layers of active dimension $1$. Finally, the last seven layers, representing the prototype of a decoder, have again constant active dimension $2$, restoring the initial one. 
A sketch of the architecture is presented on the right side of Figure \ref{fig:classification_natural}.
\begin{figure}[ht!]
\begin{minipage}{0.3\linewidth} 
\includegraphics[scale = 0.35]{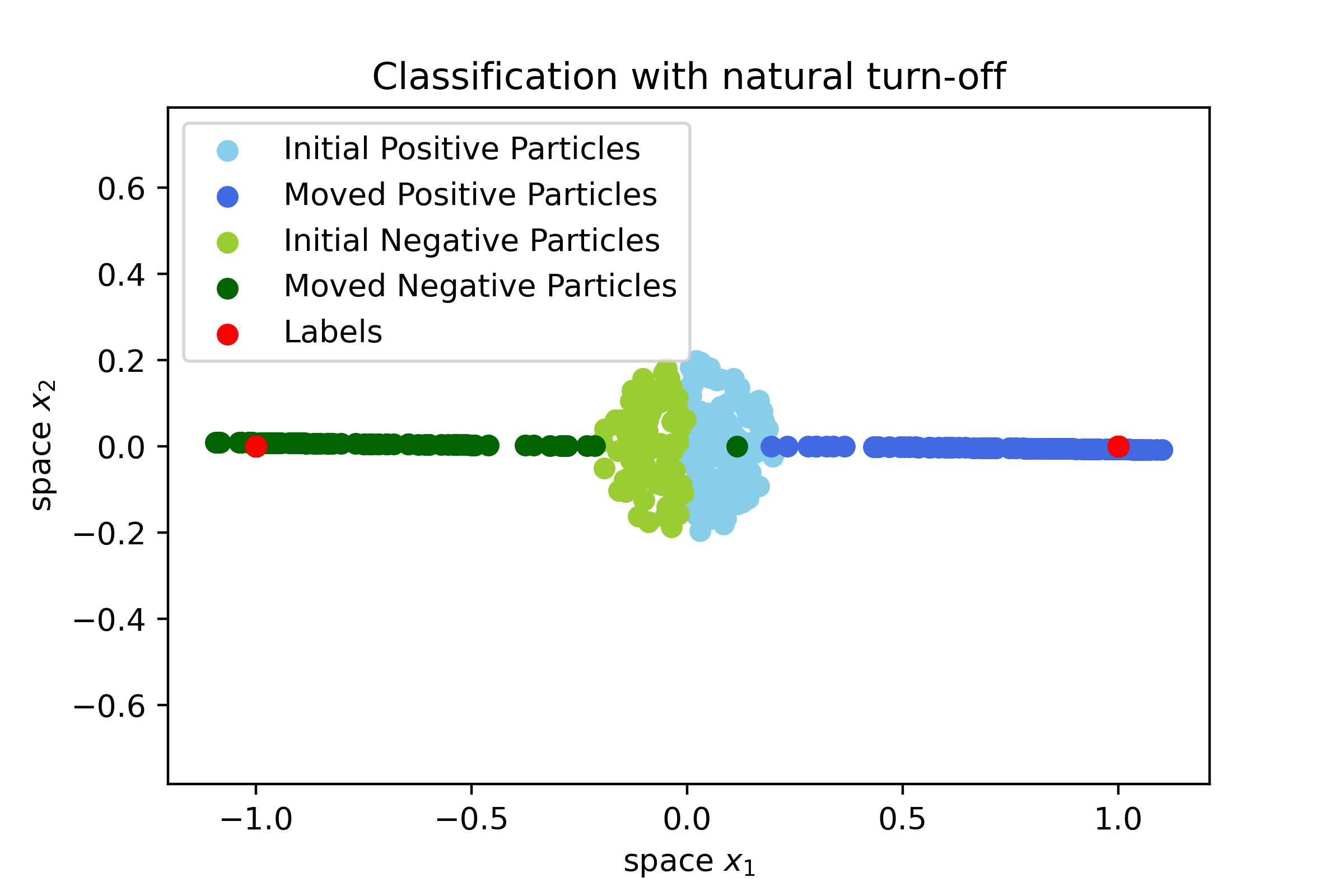}
\end{minipage}
\begin{minipage}{0.3\linewidth} 
\includegraphics[scale = 0.25]{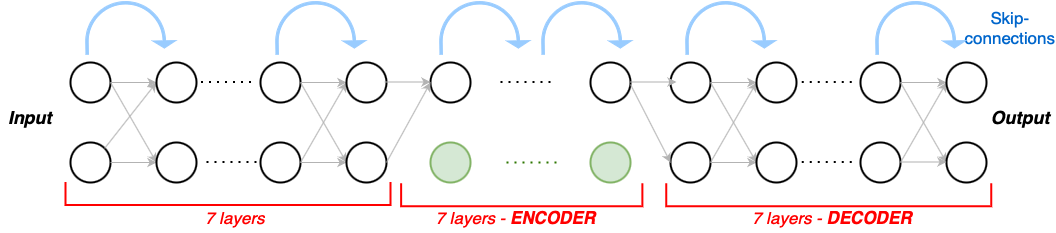}
\end{minipage}
\caption{Left: Classification task performed when the turned off component is the natural one. \\
Right: sketch of the AutoencODE architecture considered.}
\label{fig:classification_natural}
\end{figure}
We underline that we make use of the observation presented in Remark~\ref{rmk:overwriting} to construct the implemented network, and we report that we employ the hyperbolic tangent as activation function. \\
The next step is to determine which components to deactivate, i.e., we have to choose the sets $\I_j$ for $j=1, \ldots, 2r$: the natural choice is to deactivate the second component, since the information which the classification is based on is contained in the first component (the sign) of the input data-points.
Since we use the memory-saving regime of Remark~\ref{rmk:overwriting}, we observe that, in the decoder, the particles are ``projected" onto the $x$-axis, as their second component is deactivated and set equal to $0$. Then, in the decoding phase, both the components have again the possibility of evolving. This particular case is illustrated on the left side of Figure~\ref{fig:classification_natural}.\\
Now, let us consider a scenario where the network architecture remains the same, but instead of deactivating the second component, we turn-off the first component.
This has the effect of ``projecting" the particles onto the $y$-axis in the encoding phase.
The results are presented in Figure~\ref{fig:classification_unnatural}, where an interesting effect emerges.
\begin{figure}[ht]
\centering
\begin{minipage}[b]{0.3\linewidth} 
\includegraphics[scale = 0.35]{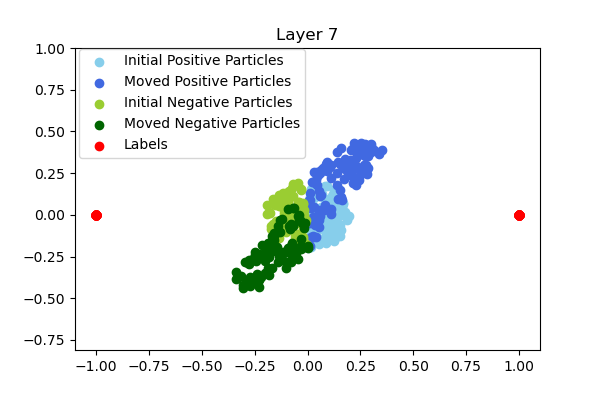}
\end{minipage}
\begin{minipage}[b]{0.3\linewidth} 
\includegraphics[scale = 0.35]{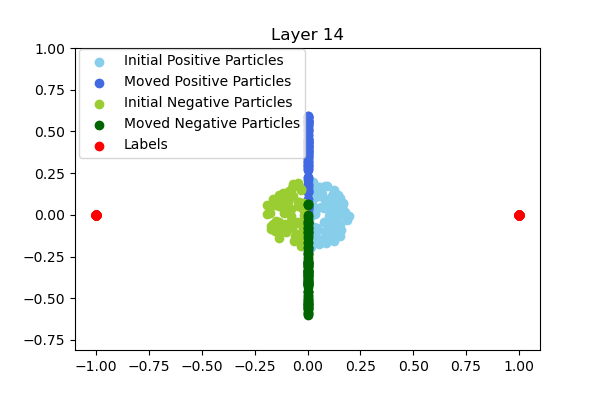}
\end{minipage}
\begin{minipage}[b]{0.3\linewidth} 
\includegraphics[scale = 0.35]{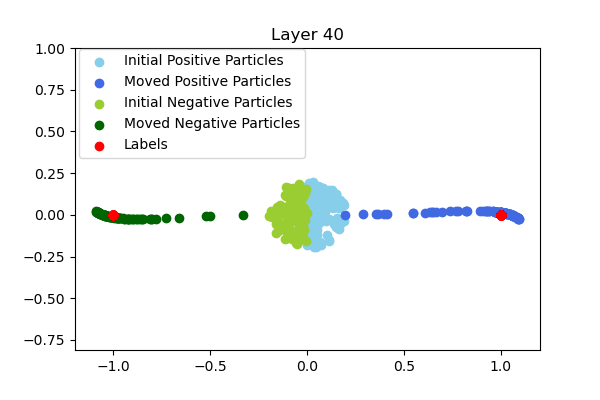}
\end{minipage}
\caption{Left: Initial phase, i.e., separation of the data along the $y$-axis. Center: Encoding phase, i.e., only the second component is active. Right: Decoding phase and classification result after the ``unnatural turn off''. Notice that, for a nice clustering of the classified data, we have increased the number of layers from $20$ to $40$. However, we report that the network accomplishes the task even if we use the same structure as in Figure~\ref{fig:classification_natural}.}
\label{fig:classification_unnatural}
\end{figure}
In the initial phase (left), where the particles can evolve in the whole space $\R^2$, the network is capable of rearranging the particles in order to separate them. More precisely, in this part, the relevant information for the classification (i.e., the sign of the first component), is transferred to the second component, that will not be deactivated.
Therefore, once the data-points are projected onto the $y$-axis in the bottleneck (middle), two distinct clusters are already formed, corresponding to the two classes of particles.
Finally, when the full dimension is restored, the remaining task consists in moving these clusters towards the respective labels, as demonstrated in the plot on the right of Figure~\ref{fig:classification_unnatural}. \\
This numerical evidence confirms that our a priori choice (even when it is very unnatural) of the components to be deactivated does not affect the network's ability to learn and classify the data.
\noindent
Finally, while studying this low-dimensional numerical example, we test one of the assumptions that we made in the tehoretical setting.
In particular, we want to check if it is reasonable to assume that the cost landscape is convex around local minima, as assumed in Theorem~\ref{thm:gen_error}.
In Table~\ref{table:hessian}, we report the smallest and highest eigenvalues of the Hessian matrix of the loss function recorded during the training process, i.e., starting from a random initial guess, until the convergence to an optimal solution.
\begin{table}[ht]
\scriptsize
\centering
\begin{tabular}{ |c|c|c|c|c|c|c|c|c|c|c| }
\hline
Epochs & 0 & 80 & 160 & 240 & 320 & 400 & 480 & 560 & 640 & 720\\
\hline
Min Eingenvalue & -1.72e-2 & -1.19e-2 & -1.09e-2 & -8.10e-3 & -3.44e-3 & -6.13e-3 & 6.80e-4  & 7.11e-4 & 7.25e-4 & 7.33e-4\\
Max Eigenvalue & 3.78e-2 & 2.84e-1 & 7.30e-1 & 9.34e-1 & 1.11 & 1.18 & 1.22 & 1.25 & 1.26 & 1.27\\ 
\hline
\end{tabular}
\caption{Minimum and maximum eigenvalues of the Hessian matrix across epochs.}
\label{table:hessian}
\end{table}

\subsection*{Parabola Reconstruction}
In our second numerical experiment, we focus on the task of reconstructing a two-dimensional parabola.
To achieve this, we sample points from the parabolic curve and we use them as the initial data for our network. The network architecture consists of a first block of seven layers with active dimension $2$, followed by seven additional layers with active dimension $1$.
Together, these two blocks represent the encoding phase in which the set of active components are $\A_j = \{0\}$ for $j=7,\ldots, 14$.
Similarly as in the previous example, the points at the $7$-th layer are ``projected" onto the $x-$axis, and for the six subsequent layers  they are constrained to stay in this subspace. 
After the $14$-th layer, the original active dimension is restored, and the particles can move in the whole space $\R^2$, aiming at reaching their original positions.
Despite the low dimensionality of this task, it provides an interesting application that allows us to observe the distinct phases of our mode, which are presented in Figure~\ref{fig:parabola}.
\begin{figure}[ht]
\centering
\begin{minipage}[b]{0.3\linewidth} 
\includegraphics[scale = 0.35]{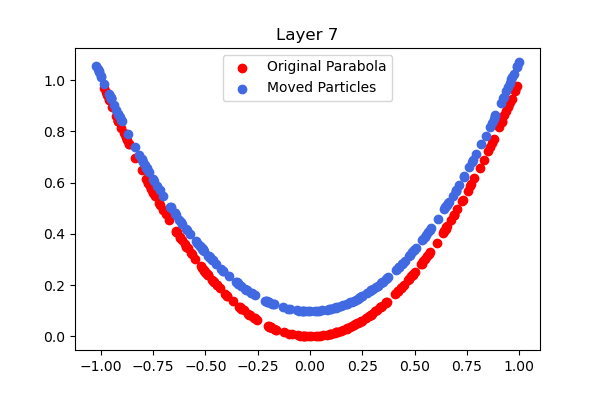}
\end{minipage}
\begin{minipage}[b]{0.3\linewidth} 
\includegraphics[scale = 0.35]{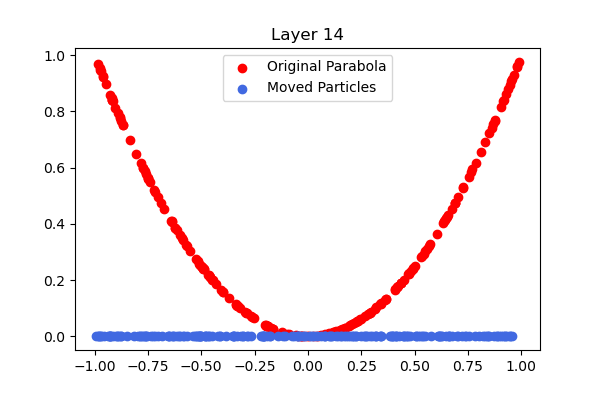}
\end{minipage}
\\
\begin{minipage}[b]{0.3\linewidth} 
\includegraphics[scale = 0.35]{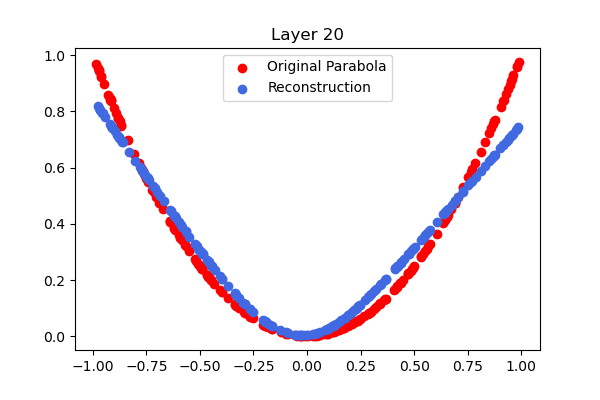}
\end{minipage}
\begin{minipage}[b]{0.3\linewidth} 
\includegraphics[scale = 0.35]{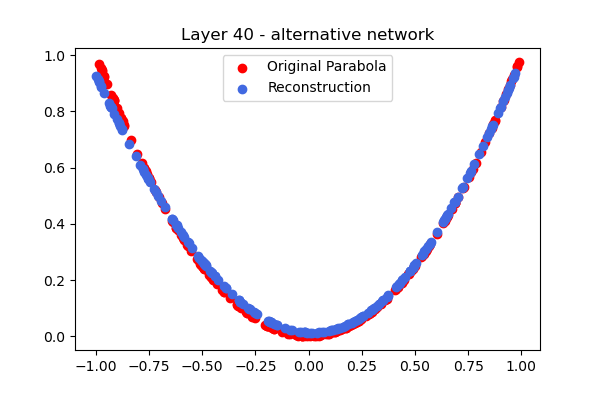}
\end{minipage}
\caption{Top Left: Initial phase. Top Rights: Encoding phase. Bottom Left: Decoding phase. Bottom Right: network's reconstruction with alternative architecture.}
\label{fig:parabola}
\end{figure}

\noindent
Notably, in the initial seven layers, the particles show quite tiny movements (top left of Figure~\ref{fig:parabola}). This is since the relevant information to reconstruct the position is encoded in the first component, which is kept active in the bottleneck.
On the other hand, if in the encoder we chose to deactivate the first component instead of the second one, we would expect that the points need to move considerably before the projection takes place, as it was the case in the previous classification task. 
During the second phase (top right of Figure~\ref{fig:parabola}), the particles separate along the $x$-axis, preparing for the final decoding phase, which proves to be the most challenging to learn (depicted in the bottom left of Figure~\ref{fig:parabola}). 
Based on our theoretical knowledge and the results from initial experiments, we attempt to improve the performance of the AutoencODE network by modifying its structure.
One possible approach is to design the network in a way that allows more time for the particles to evolve during the decoding phase, while reducing the time spent in the initial and bottleneck phases. Indeed, we try to use $40$ layers instead of $20$, and most of the new ones are allocated in the decoding phase. 
The result is illustrated in the bottom right of Figure~\ref{fig:parabola}, where we observe that changing the network's structure has a significant positive impact on the reconstruction quality, leading to better results. This result is inspired by the heuristic observation that the particles ``do not need to move'' in the first two phases. On this point, a more theoretical analysis of the network's structure will be further discussed in the next paragraph, where we perform a sanity check, and we relate the need for extra layers to the Lipschitz constant of the trained network.\\
This experiment highlights an important observation regarding the choice of activation functions. Specifically, it becomes evident that certain bounded activation functions, such as the hyperbolic tangent, are inadequate for moving the particles back to their original positions during the decoding phase. 
The bounded nature of these activation functions limits their ability to move a sufficiently large range of values, which can lead to the points getting stuck at suboptimal positions and failing to reconstruct the parabolic curve accurately. 
To overcome this limitation and achieve successful reconstruction, it is necessary to employ unbounded activation functions that allow for a wider range of values, in particular the well-known Leaky Relu function. 
An advantage of our approach is that our theory permits the use of smooth approximations for well-known activation functions, such as the Leaky ReLU \eqref{eq:leaky_relu_def}. 
Specifically, we employ the following smooth approximation of the Leaky ReLU function:
\begin{equation}\label{eq:leaky_relu_smooth}
\sigma_{smooth}(x) = \alpha x + \frac{1}{s}\log \big(1+e^{sx} \big),
\end{equation}
where $s$ approaching infinity ensures convergence to the original Leaky ReLU function. 
While alternative approximations are available, we employed \eqref{eq:leaky_relu_smooth} in our study. 
This observation emphasizes the importance of considering the characteristics and properties of activation functions when designing and training neural networks, and it motivates our goal in this work to encompass unbounded activation functions in our working assumptions.

\subsection*{MNIST Reconstruction}
In this experiment, we apply the AutoencODE architecture and our training method to the task of reconstructing images from the MNIST dataset. The MNIST dataset contains $70000$ grayscale images of handwritten digits ranging from zero to nine. Each image has a size of $28\times28$ pixels and has been normalized. 
This dataset is commonly used as a benchmark for image classification tasks or for evaluating image recognition and reconstruction algorithms. 
However, our objective in this experiment is not to compare our reconstruction error with state-of-the-art results, but rather to demonstrate the applicability of our method to high-dimensional data, and to highlight interesting phenomena that we encounter. 
In general, when performing an autoencoder reconstruction task, the goal is to learn a lower-dimensional representation of the data that captures its essential features.
On the other hand, determining the dimension of the lower-dimensional representation, often referred to as the \textit{latent dimension}, requires setting a hyperparameter, i.e., the width of the bottleneck's layers, which might depend on the specific application. \\
We now discuss the architecture we employed and the choice we made for the latent dimension. 
Our network consists of twenty-three layers, with the first ten layers serving as encoder, where the dimension of the layers is gradually reduced from the initial value $d_0=784$ to a latent dimension of $d_r=32$.
Then, this latent dimension is kept in the bottleneck for three layers, and the last ten layers act as decoder, and, symmetrically to the encoder, it increases the width of the layers from $32$ back to $d_{2r}=784$.
Finally, for each layer we employ a smooth version of the Leaky Relu, see \eqref{eq:leaky_relu_smooth}, as activation function. 
The architecture is visualized in Figure~\ref{fig:mnist_architecture}, while the achieved reconstruction results are presented in Figure~\ref{fig:mnist}.
We observe that, once again, we made use of Remark~\ref{rmk:overwriting} for the implementation of the AutoencoODE-based model.
\begin{figure}[ht!]
\begin{center}
\includegraphics[width=0.7\textwidth]{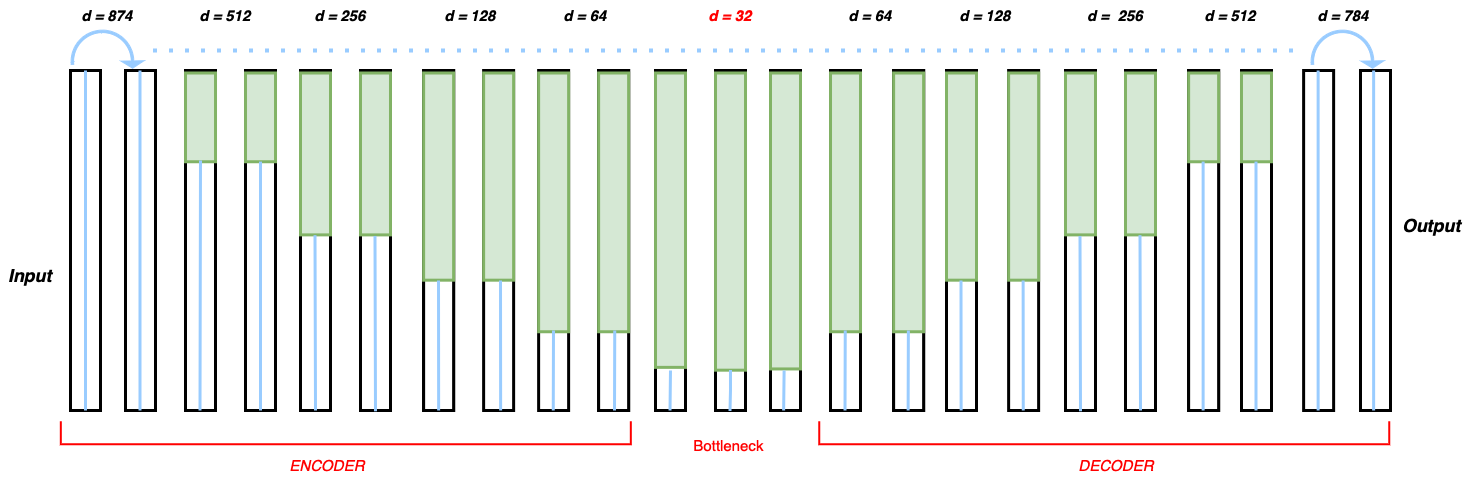}
\end{center}
\caption{Architecure used for the MNIST reconstruction task. The inactive nodes are marked in green.}\label{fig:mnist_architecture}
\end{figure} 
\begin{figure}[ht!]
\centering
\includegraphics[width=0.6\textwidth]{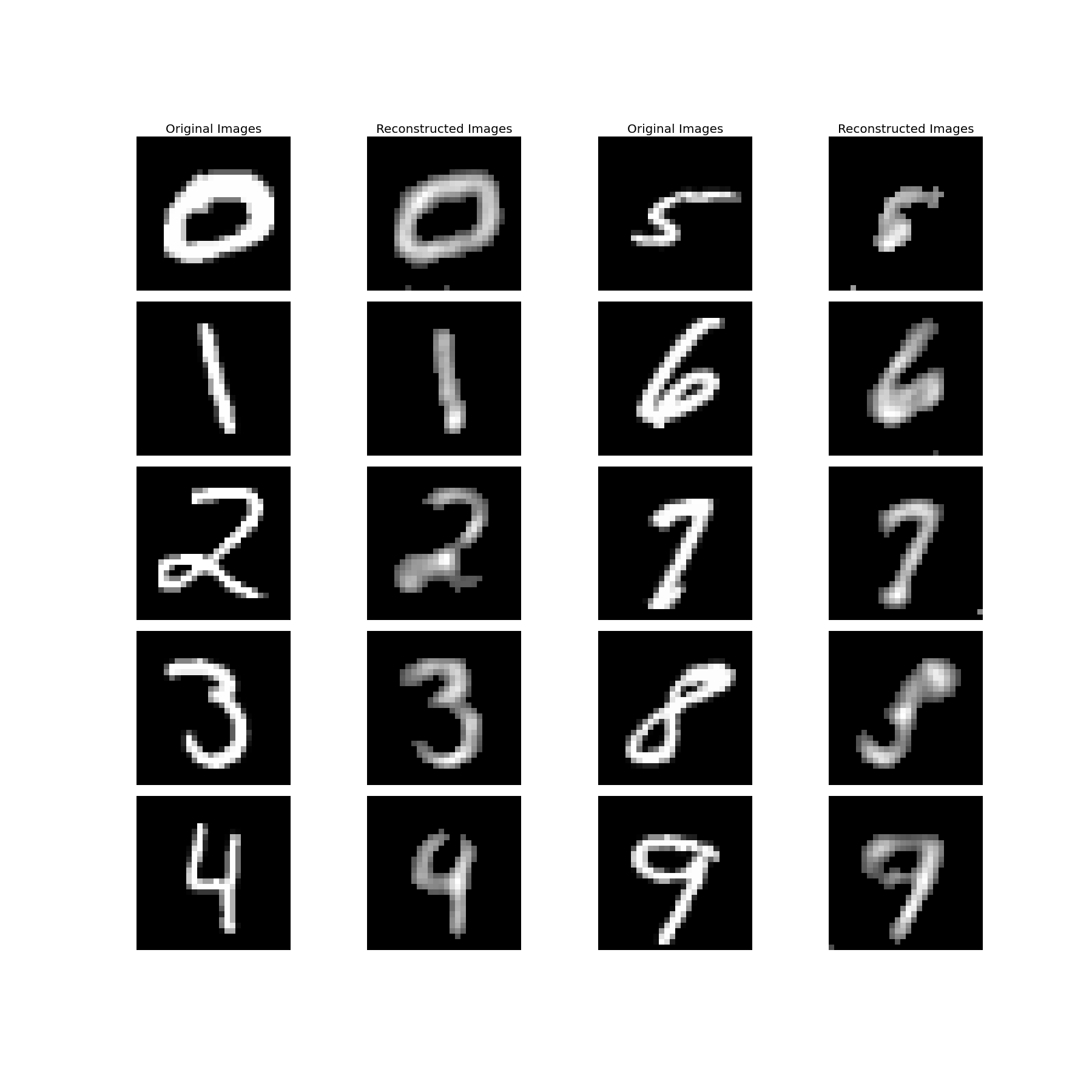}
\caption{Reconstruction of some numbers achieved by the Autoencoder.}\label{fig:mnist}
\end{figure} 

\paragraph{Latent dimensionality in the bottleneck:} 
One of the first findings that we observe in our experiments pertains to the latent dimension of the network and to the intrinsic dimension of the dataset. 
The problem of determining the intrinsic dimension has been object of previous studies such as \cite{zheng2022learning,Costa2004Learning, denti2022generalized}, where it was estimated to be approximately equal to $13$ in the case of MNIST dataset.
On this interesting topic, we also report the paper \cite{levina2004maximum}, where a maximum likelihood estimator was proposed and datasets of images were considered, and the recent contribution \cite{macocco2023intrinsic}. Finally, the model of the \textit{hidden manifold} has been formulated and studied in \cite{goldt2020modeling}.\\
Notably, our network exhibits an interesting characteristic in which, starting from the initial guess of weights and biases initialized at $0$, the training process automatically identifies an intrinsic dimensionality of $13$.
Namely, we observe that the latent vectors of dimension $32$ corresponding to each image in the dataset are sparse vectors with $13$ non-zero components, forming a consistent support across all latent vectors derived from the original images. 
To further analyze this phenomenon, we compute the means of all the latent vectors for each digit and we compare them, as depicted in the left and middle of Figure~\ref{fig:latent_means}. These mean vectors always have exactly the same support of dimension $13$, and, interestingly, we observe that digits that share similar handwritten shapes, such as the numbers $4$ and $9$ or digits $3$ and $5$, actually have latent means that are close to each other. 
Additionally, we explore the generative capabilities of our network by allowing the latent means to evolve through the decoding phase, aiming to generate new images consistent with the mean vector. 
On the right of Figure~\ref{fig:latent_means}, we present the output of the network when using a latent vector corresponding to the mean of all latent vectors representing digit $3$. \\
This intriguing behavior of our network warrants further investigation into its ability to detect the intrinsic dimension of the input data, and into the exploration of its generative potential.
Previous studies have demonstrated that the ability of neural networks to converge to simpler solutions is significantly influenced by the initial parameter values (see e.g. \cite{chou2023more}). 
Indeed, in our case we have observed that this phenomenon only occurs when initializing the parameters with zeros. 
Moreover, it is worth mentioning that this behavior does not seem to appear in standard Autoencoders without residual connections.
\begin{figure}[ht!]
\centering
\begin{minipage}[b]{0.3\linewidth} 
\includegraphics[scale = 0.35]{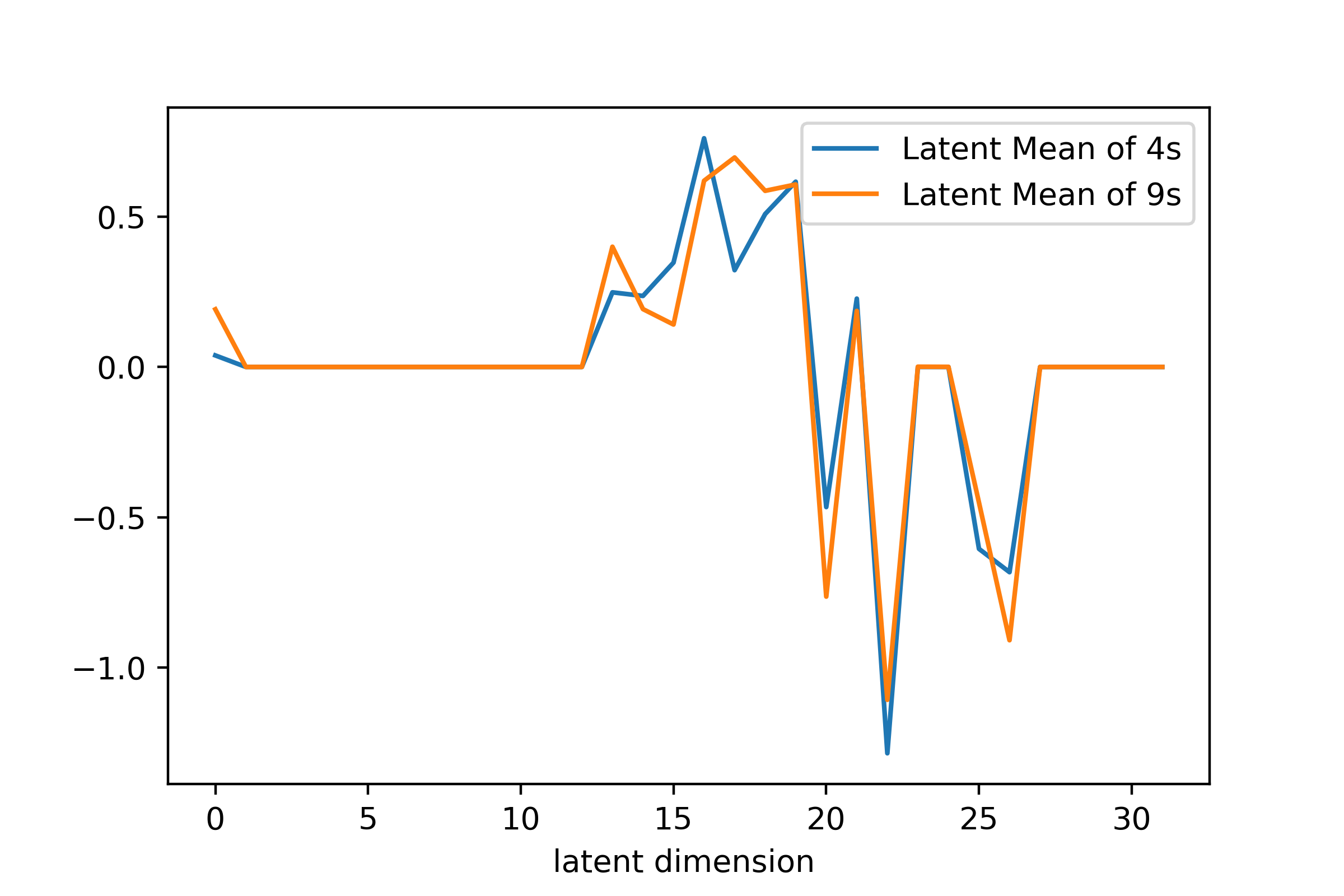}
\end{minipage}
\begin{minipage}[b]{0.3\linewidth} 
\includegraphics[scale = 0.35]{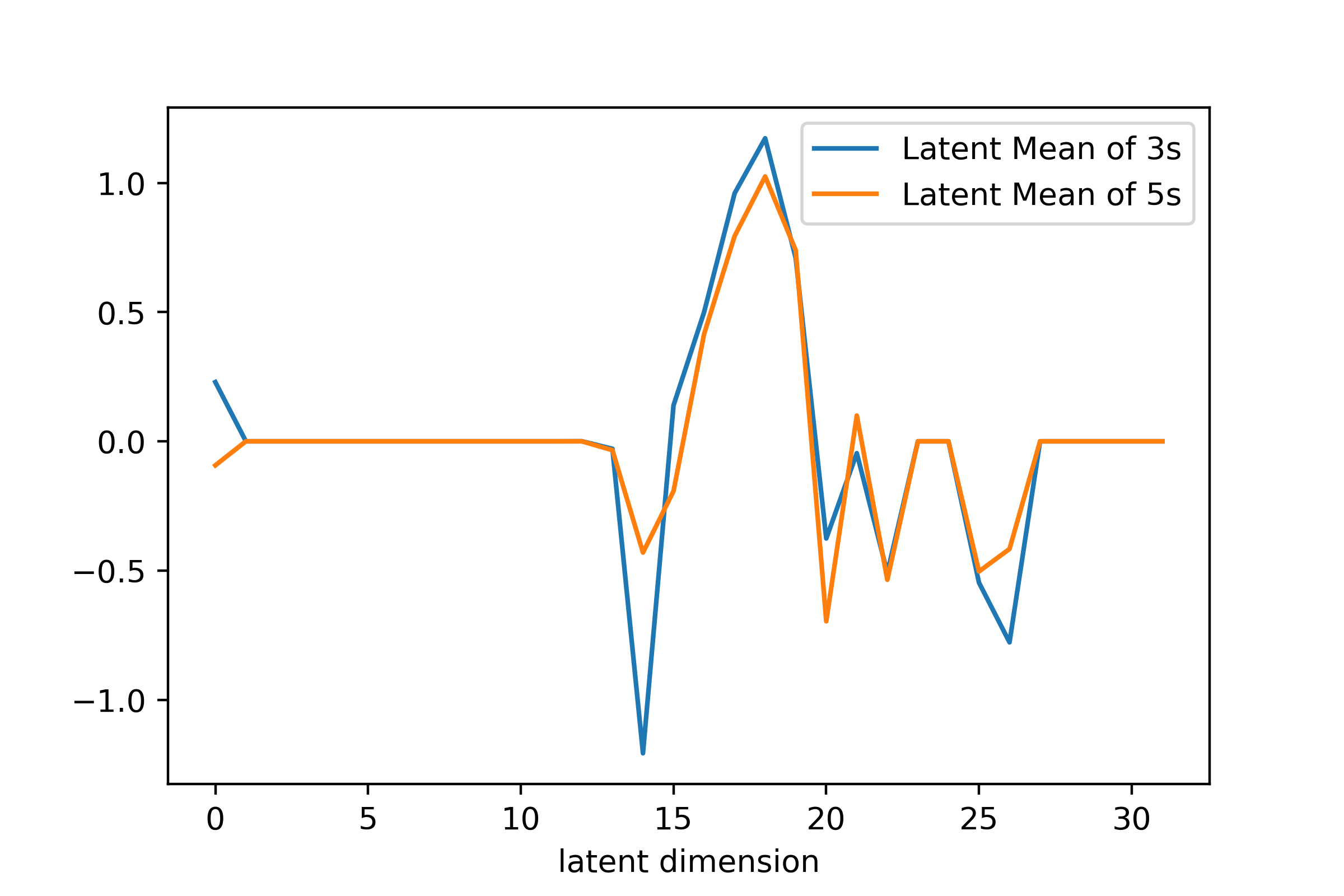}
\end{minipage}
\begin{minipage}[b]{0.3\linewidth} 
\includegraphics[scale = 0.35]{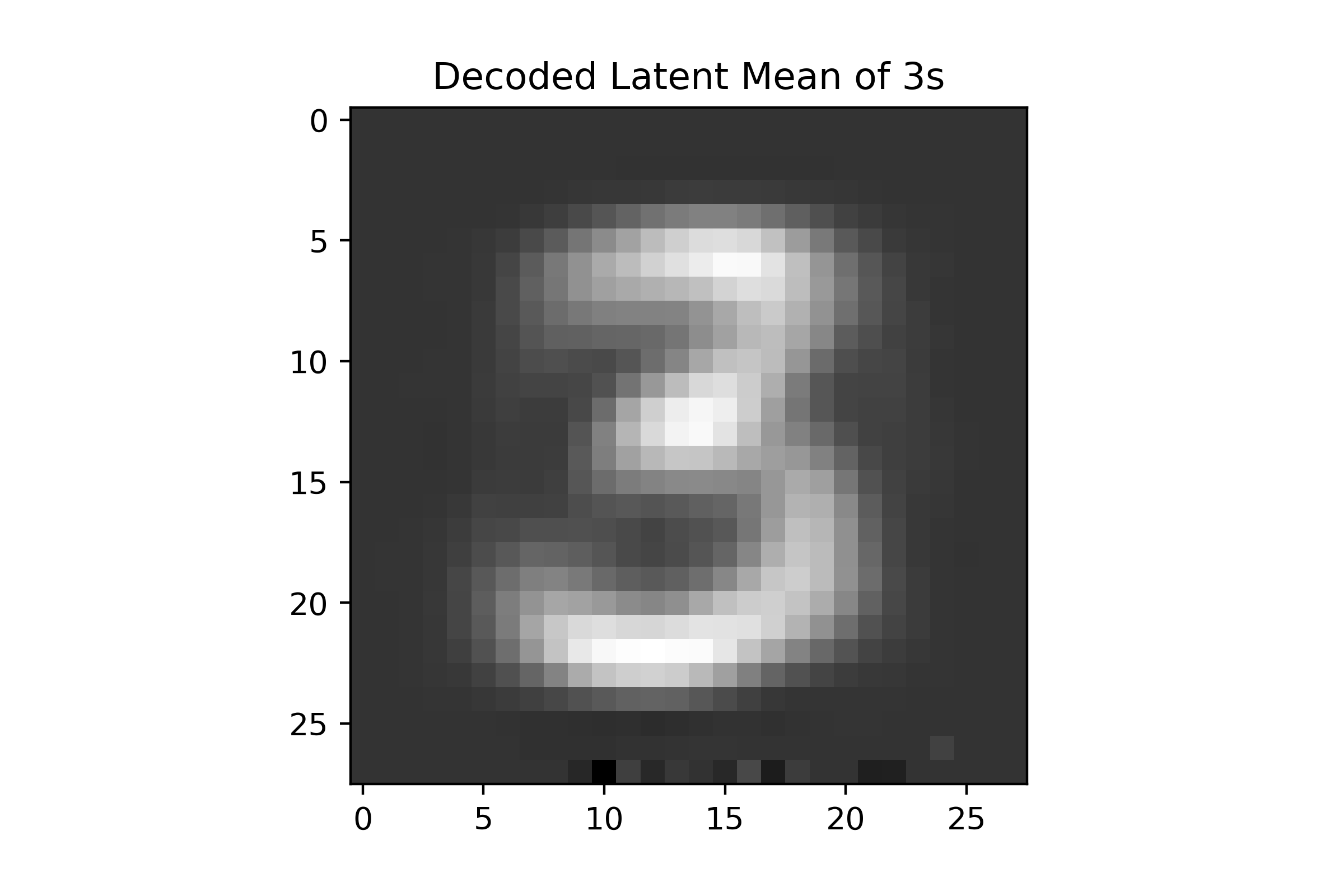}
\end{minipage}
\caption{Left: Comparing two similar latent means. Center: again two similar latent means. Right: Output of the decoding of one latent mean.}
\label{fig:latent_means}
\end{figure}
\paragraph{Sanity check of the network's architecture.}
An advantage of interpreting neural networks as discrete approximations of dynamical systems is that we can make use of typical results of numerical resolutions of ODEs in order to better analyze our results.
Indeed, we notice that, according to well-known results, in order to solve a generic ODEs we need to take as discretization step-size $dt$ a value smaller than the inverse of the lipschitz constant of the vector field driving the dynamics.
We recall that the quantity $dt$ is related to the number of layers of the network through the relation $n_{\mathrm{layers}}= \frac{T}{dt}$, where $T$ is the right-extreme of the evolution interval $[0,T]$.\\
In our case, we choose  \textit{a priori} the amplitude of $dt$, we train the network and, once we have computed $\theta^*$, we can compare \textit{a posteriori} the discretization step-size chosen at the beginning with the quantity $\Delta = \frac{1}{Lip(\MF(t,X,\theta^*)}$ for each time-node $t$ and every datum $x$. 
\begin{figure}[ht]\label{fig:inverse_lip}
\centering
\begin{minipage}[b]{0.3\linewidth} 
\includegraphics[scale = 0.35]{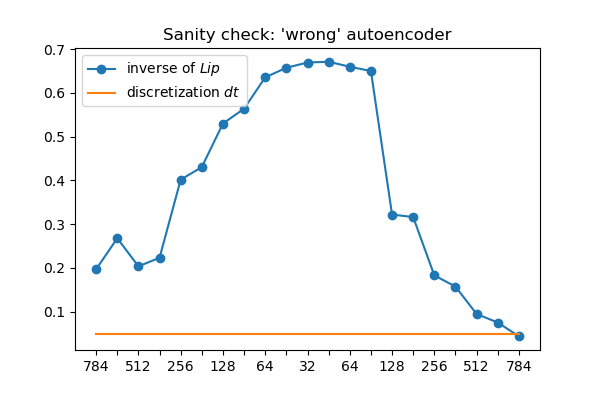}
\end{minipage}
\begin{minipage}[b]{0.3\linewidth} 
\includegraphics[scale = 0.35]{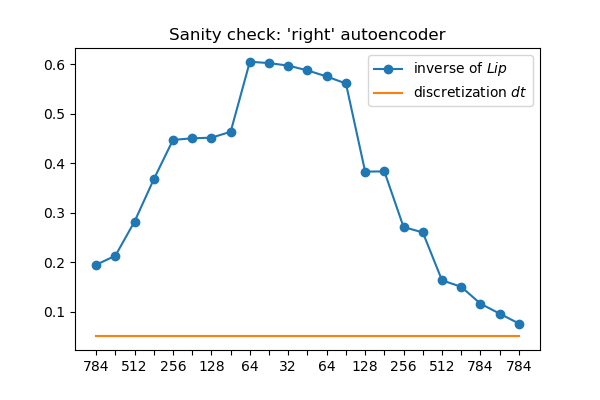}
\end{minipage}
\caption{Left: wrong autoencoder detected with $\Delta$. Right: correct version of the same autoencoder.}
\end{figure}
In Figure~\ref{fig:inverse_lip}, we show the time discretization $dt$ in orange and in blue the quantity $\Delta$, for the case of a wrongly constructed autoencoder (on the left) and the correct one (on the right). 
From this plots, we can perform a ``sanity check" and we can make sure that the number of layers that we chose is sufficient to solve the task. 
Indeed, in the wrong autoencoder on the left, we see that in the last layer the quantity $\Delta$ is smaller than $dt$, and this violates the condition that guarantees the stability of the explicit Euler discretization.\\
Indeed, the introduction of two symmetric layers to the network (corresponding to the plot on the right of Figure~\ref{fig:inverse_lip}) allows the network to satisfy everywhere the relation $\Delta > dt$.
Moreover, we also notice that during the encoding phase the inverse of the Lipschitz constant of $\MF$ is quite high, which means that the vector field does not need to move a lot the points.
This suggests that we could get rid of some of the layers in the encoder and only keep the necessary ones, i.e.,  the ones in the decoder where $\Delta$ is small and a finer discretization step-size is required. 
We report that this last observation is consistent with the results recently obtained in \cite{bungert2021neural}.
Finally, we also draw attention to the work  \cite{sherry2023designing}, which shares a similar spirit with our experiments, since the Lipschitz constant of the layers is the main subject of investigation. In their study, the authors employ classical results on the numerical integration of ordinary differential equations in order to understand how to constrain the weights of the network with the aim of designing stable architectures.
This approach leads to networks with non-expansive properties, which is highly advantageous for mitigating instabilities in various scenarios, such as testing adversarial examples \cite{goodfellow2014explaining}, training generative adversarial networks \cite{arjovsky2017wasserstein}, or solving inverse problems using deep learning.

\paragraph{Entropy across layers.}
We present our first experiments on the study of the information propagation within the network, where some intriguing results appear.
This phenomenon is illustrated in Figure~\ref{fig:entropy}, where we examine the entropy across the layers after the network has been trained.
We introduce two different measures of entropy, depicted in the two graphs of the figure. In first place, we consider the well-known Shannon entropy, denoted as $H(E)$, which quantifies the information content of a discrete random variable $E$, distributed according to a discrete probability measure $p: \Omega \to [0,1]$ such that $p(e) = p(E=e)$. The Shannon entropy is computed as follows:
\begin{equation*}
    H(E) = \mathbb{E}\big[-\log(p(E) \big] = \sum_{e \in E} -p(e)\log(p(e))
\end{equation*}
In our context, the random variable of interest is $E =  \sum_{j=1}^N 1\!\!1_{|X_0^i- X_0^j| \leq \epsilon }$, where $X_0^i$ represents a generic image from the MNIST dataset.
Additionally, we introduce another measure of entropy, denoted as $\mathcal{E}$, which quantifies the probability that the dataset can be partitioned into ten clusters corresponding to the ten different digits. This quantity has been introduced in \cite{MassimoPascalGiacomo} and it is defined as 
\begin{equation*}
    \mathcal{E} = 
    \mathbb{P} \left(X \in \bigcup_{i=1}^k B_{\varepsilon}(X_0^i) \right),
\end{equation*}
where $\varepsilon>0$ is a small radius, and $X_0^1,\ldots,X_0^k$ are samplings from the dataset.
Figure~\ref{fig:entropy} suggests the existence of a distinct pattern in the variation of information entropy across the layers, which offers a hint for further investigations. 
\begin{figure}[ht]\label{fig:entropy}
\centering
\begin{minipage}[b]{0.3\linewidth} 
\includegraphics[scale = 0.35]{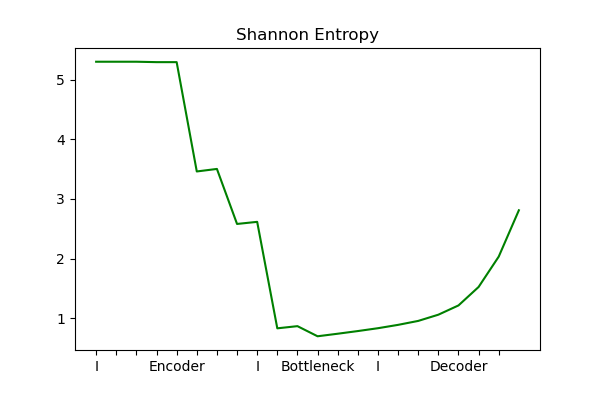}
\end{minipage}
\begin{minipage}[b]{0.3\linewidth} 
\includegraphics[scale = 0.35]{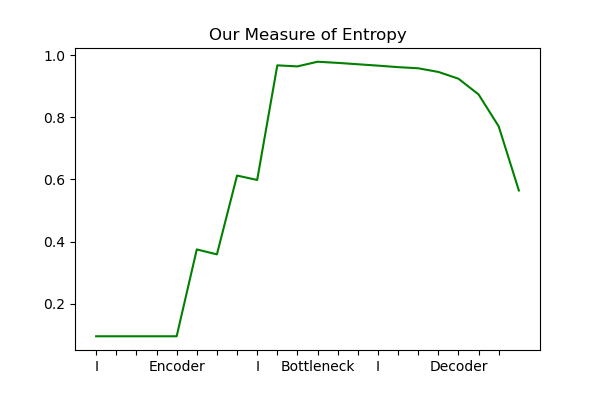}
\end{minipage}
\caption{Left: Shannon entropy across layers. Right: Our measure of entropy across layers.}
\end{figure}
Let us first focus on the Shannon entropy: as the layers' dimensionality decreases in the encoding phase, there is an expected decrease of entropy, reflecting the compression and reduction of information in the lower-dimensional representation.
The bottleneck layer, where the dimension is kept constant, represents a critical point where the entropy reaches a minimum. This indicates that the information content is highly concentrated and compressed in this latent space. 
Then, during the decoding phase, the Shannon entropy do not revert to its initial value but instead exhibit a slower increase. This behavior suggests that the network retains some of the learned structure and information from the bottleneck layer.
Something similar happens for the second measure of entropy: at the beginning, the data is unlikely to be highly clustered, since two distinct images of the same digit may be quite distant one from the other. 
In the inner layers, this probability increases until it reaches its maximum (rather close to $1$) in the bottleneck, where the data can then be fully partitioned into clusters of radius $\epsilon$. 
As for the Shannon entropy, the information from the bottleneck layer is retained during the decoding phase, which is why the entropy remains constant for a while and then decreases back in a slower manner.\\
It is worth noticing that in both cases the entropy does not fully return to its initial level.
This might be attributed to the phenomenon of mode collapse, where the network fails to capture the full variability in the input data and instead produces similar outputs for different inputs, hence inducing some sort of \textit{implicit bias}. 
Mode collapse is often considered undesirable in generative models, as it hinders the ability to generate diverse and realistic samples. 
However, in the context of understanding data structure and performing clustering, the network's capability to capture the main modes or clusters of the data can be seen as a positive aspect. 
The network learns to extract salient features and represent the data in a compact and informative manner, enabling tasks such as clustering and classification. 
Further investigation is needed to explore the relationship between the observed entropy patterns, mode collapse, and the overall performance of the network on different tasks.

\subsection*{Acknowledgments}
The authors would like to thank Prof. Giuseppe Savar\'e for the fruitful discussions during his permanence in Munich. Moreover, the authors are grateful to Dr. Oleh Melnyk for the suggestion on the extension of the dynamical model to the U-net architecture.\\

\noindent
This work has been funded by the German Federal Ministry of Education and Research and the Bavarian State Ministry for Science and the Arts. C.C. and M.F. acknowledge also the partial support of the project “Online Firestorms And Resentment Propagation On Social Media: Dynamics, Predictability and Mitigation” of the TUM Institute for Ethics in Artificial Intelligence and of the DFG Project “Implicit Bias and Low Complexity Networks” within the DFG SPP 2298 “Theoretical Foundations of Deep Learning”. A.S. acknowledges the partial support from INdAM-GNAMPA.

\appendix
\section{Appendix}
\begin{lem}[Boundedness of trajectories]
\label{lem:bound_traj}
Let us consider the controlled system
\begin{equation*}
    \dot{x} = \MF (t,x,\theta), \quad x(0)=x_0,
\end{equation*}
where $\MF:[0,T]\times\R^n\times\R^m\to \R^n$ satisfies Assumptions~\ref{ass:block1}, and $\theta \in L^2([0,T], \R^m)$. 
Then, for every $R>0$ and any $x_0 \in B_R(0)$, we have that $x(t) \in B_{\Bar{R}}(0)$ for every  $ t \in [0,T]$, where $\Bar{R} = (R + L_R (1+ \|\theta\|_{L^1}))e^{L_R(1+\| \theta \|_{L^1})}$.
\end{lem}
\begin{proof}
According to Assumption~\ref{ass:block1}$-(ii)$ on $\MF$, the trajectories can be bounded as follows:
\begin{equation*}
\begin{split}
|x(t)| 
& \leq |x_0| + \int_0^t |\MF(s,x(s), \theta(s))| \,ds 
 \leq |x_0| + L_R \int_0^t(1+ |x(s)|)(1+ |\theta(s)|) \,ds 
\end{split}
\end{equation*}
for every $t\in[0,T]$.
Using Gronwall's lemma, it follows that
\begin{equation*}
    |x(t)| \leq \Big( |x_0| + L_R (1+ \norm{\theta}_{L^1}) \Big) e^{L_R(1+ \norm{\theta}_{L^1})}.
\end{equation*}
\end{proof}

\begin{lem}[Flow's dependency on initial datum]\label{lem:lip_flow_init}
For every $ t \in [0,T]$, let us consider the flow mapping $\Phi^\theta_{(0,t)} : \R^d \to \R^d$ defined in \eqref{eq:flow} and driven by the control $\theta \in L^2([0,T],\R^m)$. 
Let us assume that the controlled dynamics $\MF:[0,T]\times\R^d\times\R^m\to \R^d$ satisfies Assumption~\ref{ass:block1}. 
Then, for every $R>0$, and every  $x_1, x_2 \in B_R(0)$, it follows that
\begin{equation*}
    |\Phi^\theta_{(0,t)}(x_1) -\Phi^\theta_{(0,t)}(x_2)| \leq e^{L_{\Bar{R}}(1+ \norm{\theta}_{L1})}|x_1-x_2|,
\end{equation*}
where $\Bar{R}$ is defined as in Lemma~\ref{lem:bound_traj}, and $L_{\bar R}$ is prescribed by Assumption~\ref{ass:block1}-($ii$).
\end{lem}
\begin{proof}
Let us denote with $t \mapsto x_1(t), t \mapsto x_2(t)$ the solutions of \eqref{eq:ode} driven by $\theta$ and starting, respectively, from $x_1(0) = x_1, x_2(0) = x_2$.
Then, for every $t\in [0,T]$, we have
\begin{align*}
|x_1(t)-x_2(t)| & \leq|x_1-x_2| + \int_0^t |\MF(s,x_1(s), \theta(s)) - \MF(s,x_2(s), \theta(s))| \,ds \\
& \leq |x_1-x_2| + L_{\Bar{R}} \int_0^t (1+ |\theta(s)|)|x_1(s)-x_2(s))| \,ds,
\end{align*}
by using Assumption \ref{ass:block1}$-(ii)$. As before, the statement follows from Gronwall's Lemma.
\end{proof}

\begin{lem}[Flow's dependency on time]\label{lem:lip_flow_time}
Under the same assumptions and notations as in Lemma~\ref{lem:lip_flow_init}, for every $R>0$, for every $x \in B_R(0)$ and for every $\theta \in L^2([0,T],\R^m)$, we have that
\begin{equation*}
    |\Phi^\theta_{(0,t_2)}(x) - \Phi^\theta_{(0,t_1)}(x)| \leq L_{\BarR} (1+ \BarR)(1 + \norm{\theta}_{L^2}) |t_2-t_1|^{\frac{1}{2}}
\end{equation*}
for every $0 \leq t_1 < t_2 \leq T$, where $\Bar{R}$ is defined as in Lemma~\ref{lem:bound_traj}, and $L_{\bar R}$ is prescribed by Assumption~\ref{ass:block1}-($ii$).
Moreover, if $\theta \in L^2([0,T], \R^m) \cap L^\infty([0,T],\R^m)$, then, for every $0 \leq t_1 < t_2 \leq T$, it holds:
\begin{equation*}
    |\BPhi^\theta_{(0,t_2)}(x) - \BPhi^\theta_{(0,t_1)}(x)| \leq L_{\BarR} (1+ \BarR)(1 + \norm{\theta}_{L^2}) |t_2-t_1|.
\end{equation*}
\end{lem}
\begin{proof}
If we denote by $t \mapsto x(t)$ the solution of \eqref{eq:ode} driven by the control $\theta$, then
\begin{equation*}
    |x(t_2)-x(t_1)| \leq \int_{t_1}^{t_2} |\MF(s,x(s),\theta(s)|\,ds  \leq \int_{t_1}^{t_2} L_{\BarR} (1+\BarR) (1 + |\theta(s)|) \,ds.
\end{equation*}
The thesis follows by using Cauchy-Schwarz for $\theta \in L^2$, or from basic estimates if $\theta \in L^\infty$.
\end{proof}

\begin{lem}[Flow's dependency on controls]\label{lem:lip_flow_ctrl}
For every $t\in [0,T]$, let $\Phi^{\theta_1}_{(0,t)}, \Phi^{\theta_2}_{(0,t)}: \R^d \to \R^d$ be the flows defined in \eqref{eq:flow} and driven, respectively, by $\theta_1,\theta_2\in L^2([0,T],\R^m)$. 
Let us assume that the controlled dynamics $\MF:[0,T]\times\R^n\times\R^m\to \R^n$ satisfies Assumption~\ref{ass:block1}.
Then, for every $R>0$ and for every $x \in B_R(0)$, it holds that
\begin{equation*}
|\Phi^{\theta_1}_{(0,t)}(x) - \Phi^{\theta_2}_{(0,t)}(x)| \leq L_{\BarR}(1+ \norm{\theta_1}_{L^2} + \norm{\theta_2}_{L^2}) e^{L_{\BarR}(1+ \norm{\theta_1}_{L^1})} \norm{\theta_1-\theta_2}_{L^2},
\end{equation*}
where $\Bar{R}$ is defined as in Lemma~\ref{lem:bound_traj}, and $L_{\bar R}$ is prescribed by Assumption~\ref{ass:block1}-($ii$).
\end{lem}

\begin{proof}
By using Assumption \ref{ass:block1}$-(ii),(iii)$ and the triangle inequality, we obtain that
\begin{align*}
    |\Phi^{\theta_1}_{(0,t)}(x) - \Phi^{\theta_2}_{(0,t)}(x)| 
    & \leq \int_0^t |\MF(s,x_1(s), \theta_1(s))-\MF(s,x_2(s), \theta_2(s))| \,ds \\
    & \leq \int_0^t |\MF(s,x_1(s), \theta_1(s))-\MF(s,x_2(s), \theta_1(s))| \,ds  \\
    &\quad + \int_0^t |\MF(s,x_2(s), \theta_1(s))-\MF(s,x_2(s), \theta_2(s))| \,ds \\
    &\leq L_{\BarR}\int_0^t (1 + \theta_1(s)) |x_1(s)-x_2(s)| \,ds + L_{\BarR}(1+ \norm{\theta_1}_{L^2} + \norm{\theta_2}_{L^2}) \norm{\theta_1-\theta_2}_{L^2} .
\end{align*}
The statement follows again by applying Gronwall's Lemma.
\end{proof}

\begin{proposition}[Differentiability with respect to trajectories perturbations]\label{prop:diff_traj}
Let us assume that the controlled dynamics $\MF$ satisfies Assumptions~\ref{ass:block1}-\ref{ass:block2}.
Given an admissible control $\theta\in L^2([0,T],\R^m)$ and a trajectory $t\mapsto x (t) = \Phi_{(0,t)}^\theta(x_0)$ with $x_0\in B_R(0)$, let $\xi:[0,T]\to \R^d$ be the solution of the linearized problem
\begin{equation} \label{eq:ode_lin1}
\begin{cases} 
 \dot{\xi}(t) = \nabla_x \MF(t,x(t),\theta(t))\xi(t), \\
 \xi(\Bart) = v,
\end{cases} 
\end{equation}
where $\Bart \in [0,T]$ is the instant of perturbation and $v$ is the direction of perturbation of the trajectory.
Then, for every $t \in (\Bart,T)$, it holds
\begin{equation*}
    |\Phi_{(\Bart,t)}^\theta(x(\Bart)+\epsilon v)- \Phi_{(\Bart,t)}^\theta(x(\Bart)) - \epsilon \xi(t)| \leq C |v|^2 \epsilon^2
\end{equation*}
where $C$ is a constant depending on $T,R, \norm{\theta}_{L^2}$.
\end{proposition}
\begin{proof}
For $t\geq \Bart$, let us denote with $t\mapsto y(t) := \Phi_{(\Bart,t)}^\theta(x\left(\Bart)+\epsilon v\right)$ the solution of the modified problem, obtained by perturbing the original trajectory with $\epsilon v$ at instant $\Bart$. 
Then, since $\xi$ solves \eqref{eq:ode_lin1}, we can write
\begin{align*}
|y(t)-x(t)-\epsilon \xi(t)| & = |\Phi_{(\Bart,t)}^\theta(x(\Bart)+\epsilon v)- \Phi_{(\Bart,t)}^\theta(x(\Bart)) - \epsilon \xi(t)|  \\
& \leq \int_{\Bart}^t |\MF(s, y(s),\theta(s))-\MF(s,x(s), \theta(s)) - \epsilon \nabla_x \MF(s,x(s),\theta(s)) \xi(s)|ds\\
&\leq \int_{\Bart}^t |\MF(s,y(s),\theta(s))-\MF(s,x(s),\theta(s)) - \nabla_x \MF(s,x(s),\theta(s))(y(s)-x(s))| ds\\
 & \quad + \int_{\Bart}^t |\nabla_x \MF(s,x(s),\theta(s))||y(s)-x(s)-\epsilon \xi(s)| ds\\
 & \leq \int_{\Bart}^t \left [ \int_0^1 |\nabla_x \MF(s, x(s) + \tau(y(s)-x(s)), \theta(s)- \nabla_x \MF(s,x(s),\theta(s)| |y(s)-x(s)| d\tau \right ] ds \\
 &\quad + \int_{\Bart}^t|\nabla_x \MF(s,x(s),\theta(s))||y(s)-x(s)-\epsilon \xi(s)|ds
\end{align*}
for every $t\geq\Bart$.
We now address the two integrals separately. Using Assumption \ref{ass:block2}$-(iv)$ and the result of Lemma \ref{lem:lip_flow_ctrl}, we obtain the following bound
\begin{align*}
& \int_{\Bart}^t \left [ \int_0^1 |\nabla_x \MF(s, x(s) + \tau(y(s)-x(s))), \theta(s)- \nabla_x \MF(s,x(s),\theta(s))| |y(s)-x(s)| d\tau \right ] ds \\
&\leq \int_{\Bart}^t L_{\BarR}(1+ |\theta(s)|^2)\frac{1}{2}|y(s)-x(s)|^2 ds \\
 &\leq  \frac{1}{2}
 L_{\BarR}\left(1+\|\theta\|_{L^2}^2\right)e^{2L_{\BarR}(1+ \norm{\theta}_{L^1})}|\epsilon v|^2
\end{align*}
Similarly, for the second integral, owing to Assumption~\ref{ass:block2}$-(iv)$, we can compute:
\begin{align*}
\int_{\Bart}^t|\nabla_x \MF(s,x(s),\theta(s))||y(s)-x(s)-\epsilon \xi(s)|ds \leq \int_{\Bart}^t L_{\BarR}(1+ |\theta(s)|^2)(1+ \BarR)|y(s)-x(s)-\epsilon \xi(s)| ds
\end{align*}
Finally, by combining the two results together and using Gronwall's Lemma, we prove the statement.
\end{proof}

\begin{proposition}[Differentiability with respect to control perturbations]\label{prop:diff_ctrl}
Consider the solution $\xi$ of the linearized problem
\begin{equation}\label{eq:ode_lin2}
\begin{cases}
  \dot{\xi}(t) = \nabla_x \MF(t,x^\theta(t),\theta(t))\xi(t) + \nabla_\theta \MF(t,x^\theta(t), \theta(t)) \nu(t)\\
 \xi(0) = 0
\end{cases}
\end{equation}
where the control $\theta$ is perturbed at the initial time with $\theta + \epsilon \nu$, when starting with an initial datum $x_0 \in B_R(0)$. Then, 
\begin{equation} \label{eq:diff_flow_ctrls}
    |\Phi_{(0,t)}^{\theta+ \epsilon \nu}(x_0)- \Phi_{(0,t)}^\theta(x_0) - \epsilon \xi(t)| \leq C ||\nu||_{L^2}^2 \epsilon^2
\end{equation}
where $C$ is a constant depending on $T,\BarR, L_{\BarR}, \norm{\theta}_{L^1}$. Moreover, we have that for every $t\in[0,T]$
\begin{equation}\label{eq:xi_identity}
    \xi(t) = \int_0^t \mathcal{R}^\theta_{(s,t)}(x_0)\cdot \nabla_\theta \MF(s,x^\theta(s), \theta(s)) \nu(s) \,ds,
\end{equation}
where $\mathcal{R}^\theta_{(s,t)}(x_0)$ has been defined in \eqref{eq:ode_R}.
\end{proposition}
\begin{proof}
We first observe that the dynamics in \eqref{eq:ode_lin2} is affine in the $\xi$ variable. Moreover, Assumptions~\ref{ass:block1}-\ref{ass:block2} guarantee that the coefficients are $L^1$-regular in time. Hence, from the classical Caratheodory Theorem we deduce the existence and the uniqueness of the solution of \eqref{eq:ode_lin2}. Finally, the identity \eqref{eq:xi_identity} follows as a classical application of the resolvent map $(\mathcal{R}^{\theta}_{(s,t)}(x_0)_{s,t \in [0,T]}$ (see ,e.g., in \cite[Theorem~2.2.3]{Bressan_Piccoli}).

\noindent
Let us denote with $t\mapsto x(t)$ and $t\mapsto y(t)$ the solutions of Cauchy problem \eqref{eq:ode} corresponding, respectively, to the admissible controls $\theta$ and $\theta+\epsilon \nu$. We observe that, in virtue of Lemma~\ref{lem:bound_traj}, we have that there exists $\bar R>0$ such that $x(t),y(t)\in B_{\bar R}(0)$ for every $t\in [0,T]$.
Then, recalling the definition of the flow map provided in \eqref{eq:flow}, we compute
\begin{align*}
|y(t)-x(t)-\epsilon \xi(t)| & = |\Phi_{(0,t)}^{\theta+ \epsilon \nu}(x_0)- \Phi_{(0,t)}^\theta(x_0) - \epsilon \xi(t)| \\
& \leq \int_0^t | \MF(s,y(s),\theta(s)+ \epsilon \nu(s))- \MF(s,x(s), \theta(s)) 
-\epsilon \dot \xi(s)|
\,ds \\
& \leq \int_0^t |\MF(s,y(s),\theta(s)+ \epsilon \nu(s))- \MF(s,x(s), \theta(s) + \epsilon\nu(s)) \\
& \qquad \qquad -\epsilon \nabla_x \MF(s,x(s),\theta(s)+\epsilon\nu(s))\cdot(y(s)-x(s))| \, ds \\
& \quad + \int_0^t|\MF(s,x(s),\theta(s)+\epsilon\nu(s)) - \MF(s,x(s),\theta(s))-\epsilon\nabla_\theta \MF(s,x(s),\theta(s)) \cdot \nu(s)|\, ds \\
& \quad + \int_0^t |\nabla_x \MF(s,x(s),\theta(s)+ \epsilon \nu(s))-\nabla_x \MF(s,x(s),\theta(s))||y(s)-x(s)|ds\\
& \quad + \int_0^t |\nabla_x\MF(s,x(s),\theta(s))||y(s)-x(s)-\epsilon \xi(s)|\, ds.
\end{align*}
We now handle each term separately:
\begin{equation}\label{eq:comp_int1}
    \begin{split}
    & \int_0^t |\MF(s,y(s),\theta(s)+ \epsilon \nu(s))- \MF(s,x(s), \theta(s) + \epsilon\nu(s)) -\epsilon \nabla_x \MF(s,x(s),\theta(s)+\epsilon\nu(s))(y(s)-x(s))|\,ds \\
    & \leq \int_0^t \left[ \int_0^1L_{\BarR}(1+ |\theta(s)+\epsilon \nu(s)|^2) \tau |y(s)-x(s)|^2d\tau \right] ds\\
    & \leq L_{\BarR}^3 (1+ \norm{\theta}_{L^2}+\epsilon \norm{\nu}_{L^2})^4e^{2L_{\BarR}(1+\norm{\theta}_{L^1})}\norm{\nu}^2_{L^2} \epsilon^2
    \end{split}
\end{equation}
where we used Assumption~\ref{ass:block2}$-(iv)$ and Lemma~\ref{lem:lip_flow_ctrl}. By using Assumption \ref{ass:block2}$-(v)$, we obtain the following bounds for the second integral:
\begin{equation}\label{eq:comp_int2}
    \begin{split}
& \int_0^t|\MF(s,x(s),\theta(s)+\epsilon\nu(s)) - \MF(s,x(s),\theta(s))-\nabla_\theta \MF(s,x(s),\theta(s)) \cdot \epsilon \nu(s)|\,ds \\
& \leq \int_0^t \left [\int_0^1 L_{\BarR} |\nu(s)|^2\epsilon^2\tau d\tau\right]ds 
= \frac{1}{2} L_{\BarR}  \norm{\nu}_{L^2}^2\epsilon^2.
    \end{split}
\end{equation}
Similarly, the third integral can be bounded by using Assumption \ref{ass:block2}$-(vi)$ and Lemma~\ref{lem:lip_flow_ctrl}, and it yields
\begin{equation}\label{eq:comp_int3}
    \begin{split}
&\int_0^t |\nabla_x \MF(s,x(s),\theta(s)+ \epsilon \nu(s))-\nabla_x \MF(s,x(s),\theta(s))||y(s)-x(s)|\,ds\\
& \leq \int_0^t L_{\BarR}(1+ |\theta(s)|+\epsilon|\nu(s)|)\epsilon|y(s)-x(s)||\nu(s)|\, ds\\
& \leq L_{\BarR}^2(1+\norm{\theta}_{L^2}+\epsilon\norm{\nu}_{L^2})^2 e^{L_{\BarR}(1+\norm{\theta}_{L^1})}\norm{\nu}_{L^2}^2\epsilon^2.
    \end{split}
\end{equation}
Finally, the fourth integral can be bounded using Assumption \ref{ass:block2}$-(iv)$ as follows:
\begin{equation}\label{eq:comp_int4}
    \begin{split}
\int_0^t |\nabla_x\MF(s,x(s),\theta(s))||y(s)-x(s)-\epsilon \xi(s)| ds &\leq \int_0^t L_{\BarR}(1+\BarR)(1+|\theta(s)|^2)|y(s)-x(s)-\epsilon\xi(s)| \, ds.
    \end{split}
\end{equation}
Hence, by combining \eqref{eq:comp_int1}, \eqref{eq:comp_int2}, \eqref{eq:comp_int3} and \eqref{eq:comp_int4}, the thesis follows from Gronwall Lemma.
\end{proof}

\begin{proposition}[Properties of the resolvent map]\label{prop:resolvent}
Let us assume that the controlled dynamics $\MF$ satisfies Assumptions~\ref{ass:block1}-\ref{ass:block2}.
Given an admissible control $\theta\in L^2([0,T],\R^m)$ and a trajectory $t\mapsto x (t) = \Phi_{(0,t)}^\theta(x)$ with $x\in B_R(0)$, for every $\tau \in [0,T]$ the resolvent map $\mathcal{R}^\theta_{(\tau,\cdot)}(x):[0,T]\to \R^{d\times d}$ is the curve
$s \mapsto \mathcal{R}^\theta_{(\tau,s)}(x_0)$ that solves 
\begin{equation}\label{eq:resolvent_app}
    \begin{cases}
      \frac{d}{ds}\mathcal{R}^\theta_{(\tau,s)}(x) = \nabla_x \MF(s, \Phi_{(0,s)}^\theta(x),\theta(s)) 
      \cdot \mathcal{R}_{(\tau,s)}^\theta(x) &
      \mbox{for a.e. } s\in[0,T],\\
      \mathcal{R}_{(\tau,\tau)}^\theta(x) =
      \mathrm{Id}.
    \end{cases}
\end{equation}
Then for every $\tau,s\in [0,T]$, there exists a constant $C_1$ depending on $T,R, \norm{\theta}_{L^2}$ such that
\begin{equation} \label{eq:resolv_bound}
    |\mathcal{R}^\theta_{(\tau,s)}(x)| := \sup_{v\neq 0} \frac{|\mathcal{R}^\theta_{(\tau,s)}(x)\cdot v|}{|v|} \leq C_1.
\end{equation}
Moreover, for every $x,y\in B_R(0)$, there exists a constant $C_2$ depending on $T,R, \norm{\theta}_{L^2}$ such that
\begin{equation}\label{eq:resolv_lip_space}
    |\mathcal{R}^\theta_{(\tau,s)}(x) - \mathcal{R}^\theta_{(\tau,s)}(y)| := \sup_{v\neq 0} \frac{|\mathcal{R}^\theta_{(\tau,s)}(x)\cdot v - \mathcal{R}^\theta_{(\tau,s)}(y)\cdot v|}{|v|} \leq C_2|x-y|.
\end{equation}
Finally, if $\theta_1,\theta_2$ satisfy $\|\theta_1 \|,\|\theta_2 \|\leq \rho$, then there exists a constant $C_3$ depending on $T,R,\rho$ such that
\begin{equation}\label{eq:resolv_lip_ctrls}
    |\mathcal{R}^{\theta_1}_{(\tau,s)}(x)-\mathcal{R}^{\theta_2}_{(\tau,s)}(x)| := \sup_{v\neq 0} \frac{|\mathcal{R}^{\theta_1}_{(\tau,s)}(x)\cdot v -\mathcal{R}^{\theta_2}_{(\tau,s)}(x)\cdot v|}{|v|} \leq C_3 \|\theta_1-\theta_2\|_{L^2}.
\end{equation}
\end{proposition}
\begin{proof}
We first prove the boundedness of the resolvent map.
Let us fix $v\in \R^d$ with $v\neq 0$, and let us define $\xi(s):=\mathcal{R}^{\theta}_{(\tau,s)}(x)\cdot v$ for every $s\in [0,T]$.
Then, in virtue of Assumption~\ref{ass:block2}$-(vi)$, we have:
\begin{equation*}
    |\xi(s)| 
    \leq |\xi(\tau)| + \int_\tau^s \left|\nabla_x \MF(\sigma,\Phi_{(0,\sigma)}^\theta(x), \theta(\sigma))\right||\xi(\sigma)| \, d\sigma 
    \leq |v|+ L_{\BarR}\int_0^t (1+ \theta(\sigma)^2)|\xi(\sigma)|\, d\sigma,
\end{equation*}
and, by Gronwall's Lemma, we deduce \eqref{eq:resolv_bound}.
Similarly as before, given $x,y\in B_R(0)$ and $v\neq 0$, let us define $\xi^x(s):=\mathcal{R}^{\theta}_{(\tau,s)}(x)\cdot v$ and $\xi^y(s):=\mathcal{R}^{\theta}_{(\tau,s)}(y)\cdot v$ for every $s\in [0,T]$.
Then, we have that
\begin{equation*}
    \begin{split}
    |\xi^x(s) - \xi^y(s)| & \leq \int_\tau^s \left|\nabla_x \MF(\sigma,\Phi_{(0,\sigma)}^\theta(x), \theta(\sigma)) 
    \xi^x(\sigma) -
    \nabla_x \MF(\sigma,\Phi_{(0,\sigma)}^\theta(y), \theta(\sigma)) 
    \xi^y(\sigma)
    \right|
    \, d\sigma \\
    & \leq 
    \int_\tau^s \left|\nabla_x \MF(\sigma,\Phi_{(0,\sigma)}^\theta(x), \theta(\sigma))  -
    \nabla_x \MF(\sigma,\Phi_{(0,\sigma)}^\theta(y), \theta(\sigma)) 
    \right| |\xi^y(\sigma)|
    \, d\sigma \\
    & \quad +
        \int_\tau^s \left|\nabla_x \MF(\sigma,\Phi_{(0,\sigma)}^\theta(x), \theta(\sigma))
    \right| |\xi^x(\sigma)- \xi^y(\sigma)|
    \, d\sigma \\
    & \leq C_1 |v| \int_\tau^s L_{\bar R}(1 + \theta(\sigma)^2 ) \left|\Phi_{(0,\sigma)}^\theta(x) - \Phi_{(0,\sigma)}^\theta(y)
    \right|
    \, d\sigma \\
    & \quad + 
    \int_\tau^s L_{\bar R}(1 + \theta(\sigma)^2 ) |\xi^x(\sigma)- \xi^y(\sigma)| 
    \, d\sigma,
    \end{split}
\end{equation*}
where we used \eqref{eq:resolv_bound} and Assumption~\ref{ass:block2}-($iv$). Hence, combining Lemma~\ref{lem:lip_flow_init} with Gronwall's Lemma, we deduce \eqref{eq:resolv_lip_space}.
Finally, we prove the dependence of the resolvent map on different controls $\theta_1, \theta_2 \in L^2([0,T];\R^m)$. Given $x\in B_R(0)$ and $v\neq 0$, let us define $\xi^{\theta_1}(s):=\mathcal{R}^{\theta_1}_{(\tau,s)}(x)\cdot v$ and $\xi^{\theta_2}(s):=\mathcal{R}^{\theta_2}_{(\tau,s)}(x)\cdot v$ for every $s\in [0,T]$. Then, we compute
\begin{align*}
    |\xi^{\theta_1}(s)-\xi^{\theta_2}(s)| & \leq \int_\tau^s \left|\nabla_x \MF(\sigma,\Phi_{(0,\sigma)}^{\theta_1}(x), \theta_1(\sigma)) 
    \xi^{\theta_1}(\sigma) -
    \nabla_x \MF(\sigma,\Phi_{(0,\sigma)}^{\theta_2}(x), \theta_2(\sigma)) 
    \xi^{\theta_2}(\sigma)
    \right|
    \, d\sigma \\
    & \leq \int_\tau^s \left|\nabla_x \MF(\sigma,\Phi_{(0,\sigma)}^{\theta_1}(x), \theta_1(\sigma)) 
     -
    \nabla_x \MF(\sigma,\Phi_{(0,\sigma)}^{\theta_2}(x), \theta_2(\sigma)) 
    \right| |\xi^{\theta_1}(\sigma)|
    \, d\sigma \\
    & \quad + \int_\tau^s \left|
    \nabla_x \MF(\sigma,\Phi_{(0,\sigma)}^{\theta_2}(y), \theta(\sigma)) 
    \right|
    |\xi^{\theta_1}(\sigma)- \xi^{\theta_2}(\sigma)|
    \, d\sigma \\
        & \leq C_1 |v| \int_\tau^s 
        L_{\bar R}
        (1 + \theta_1(\sigma)^2 ) \left|\Phi_{(0,\sigma)}^{\theta_1}(x) - \Phi_{(0,\sigma)}^{\theta_2}(x)
    \right|\, d\sigma \\
    & \quad +  C_1 |v|
    \int_\tau^s  
        L_{\bar R}(1+ |\theta_1(\sigma)|+ |\theta_2(\sigma)|) |\theta_1(\sigma)- \theta_2(\sigma)|
    \, d\sigma \\
    & \quad + 
    \int_\tau^s L_{\bar R}(1 + \theta(\sigma)^2 )     |\xi^{\theta_1}(\sigma)- \xi^{\theta_2}(\sigma)| 
    \, d\sigma,
\end{align*}
where we used Assumption~\ref{ass:block2}$-(iv)$-$(vi)$.
\end{proof}

\printbibliography
\end{document}